\documentclass[twocolumn, 10pt, journal, romanappendices]{IEEEtran}

\usepackage{tikz}
\usepackage{amsthm}
\usepackage{blindtext}
\usepackage{graphicx}
\usepackage{amsmath}
\usepackage{mathrsfs}
\usepackage[capitalise]{cleveref}
\usepackage[sans]{dsfont}
\usepackage{amsbsy}
\usepackage{amssymb}
\usepackage{chngcntr}
\usepackage{tikz}
\usetikzlibrary{shapes,arrows,backgrounds,fit,positioning}
\usepackage{cite}
\usepackage{enumitem}
\usepackage{subfigure}

\usetikzlibrary{patterns}

\newtheorem{assumption}{Assumption}
\newtheorem{problem}{Problem}
\newtheorem{proposition}{Proposition}
\newtheorem{remark}{Remark}
\newtheorem{theorem}{Theorem}
\newtheorem{definition}{Definition}
\newtheorem{corollary}{Corollary}
\newtheorem{lemma}{Lemma}
\newtheorem{example}{Example}

\newcommand{\QED}{\hfill$\blacksquare$}

\DeclareMathOperator*{\argmax}{arg\,max}

\begin{document}

\title{Payoff Dynamics Model and Evolutionary Dynamics Model: Feedback and Convergence to Equilibria}

\author{Shinkyu Park, Nuno C. Martins, and Jeff S. Shamma
\thanks{Shinkyu Park is with the Department of Mechanical and Aerospace Engineering, Princeton University, Princeton, NJ 08544, USA. \texttt{shinkyu@princeton.edu}}
\thanks{Nuno Miguel Lara Cintra Martins is with the Department of Electrical and Computer Engineering and the Institute for Systems Research, University of Maryland, College Park, MD 20742, USA. \texttt{nmartins@umd.edu}}
\thanks{Jeff S. Shamma is with King Abdullah University of Science and Technology (KAUST), Computer, Electrical and Mathematical Science and Engineering Division (CEMSE), Thuwal 23955-6900, Saudi Arabia.\texttt{jeff.shamma@kaust.edu.sa}}}

\maketitle

\begin{abstract}
This tutorial article puts forth a framework to analyze the noncooperative strategic interactions among the members of a large population of bounded rationality agents. Our approach hinges on, unifies and generalizes existing methods and models predicated in evolutionary and population games. It does so by adopting a system-theoretic formalism that is well-suited for a broad engineering audience familiar with the basic tenets of nonlinear dynamical systems, Lyapunov stability, storage functions, and passivity. The framework is pertinent for engineering applications in which a large number of agents have the authority to select and repeatedly revise their strategies. A mechanism that is inherent to the problem at hand or is designed and implemented by a coordinator ascribes a payoff to each possible strategy. Typically, the agents will prioritize switching to strategies whose payoff is either higher than the current one or exceeds the population average. The article puts forth a systematic methodology to characterize the stability of the dynamical system that results from the feedback interaction between the payoff mechanism and the revision process. This is important because the set of stable equilibria is an accurate predictor of the population's long-term behavior. The article includes rigorous proofs and examples of application of the stability results, which also extend the state of the art because, unlike previously published work, they allow for a rather general class of dynamical payoff mechanisms. The new results and concepts proposed here are thoroughly compared to previous work, methods and applications of evolutionary and population games.

\end{abstract}


% USE CONSISTENT NOMENCLATURE: REST POINTS FOR DYNAMICS AND EQUILIBRIUM NE FOR GAME.
% USE CONTRACTIVE GAMES \cite{Sandholm2015Handbook-of-gam}
% USE population STATE (OF POPULATION) CONSISTENTLY

\section{Introduction}

This tutorial puts forth a system-theoretic framework and methods to characterize the stability properties of the noncooperative strategic interactions among the members of a large population of bounded rationality agents, in response to a payoff mechanism. Although, to simplify our notation, we assume that there is a single population, the techniques and methods put forth in this article can be readily adapted to the multi-population case~\cite{Park2019From-Population,Arcak2020Dissipativity-T}.

We adopt the evolutionary dynamics paradigm well-documented in~\cite{Sandholm2015Handbook-of-gam,Sandholm2010Population-Game,Weibull1995Evolutionary-ga, Hofbauer1998Evolutionary-ga} according to which the members of the population, which we call agents, repeatedly revise their strategy choices as specified by a so-called revision protocol (protocol for short). The population is protocol-homogeneous because all of its agents follow the same protocol, but is otherwise strategy-heterogeneous considering that the agents are allowed to concurrently select distinct strategies. Although mixed strategies are not allowed, as at every instant each agent chooses exactly one out of a set of \underline{$n$ possible strategies}, the strategy selection typically involves randomization, in which case the protocol is probabilistic. The identity of each agent is unimportant, and consequently all the information that is relevant, at any given instant, is encapsulated in the so-called {\it population state} vector whose $n$ entries are proportional to the share of the population adopting each strategy. Henceforth, we refer to the set of all possible population states, or equivalently the state space, as the strategy profile set. Hence, every strategy profile is a vector with $n$ nonnegative entries adding up to the so-called population mass~$m$.

In most related prior work, a so-called {\it population game}~\cite{Blume2018The-economy-as-} determines at any given instant the $n$-dimensional payoff vector, which quantifies the payoff associated with each strategy, as a function of the population state. A population game may represent a pricing scheme that is implemented by a coordinator, or it may be inherent to the problem at hand, such as when it results from the interactions among the agents and with the environment. Examples include congestion population games for traffic assignment~\cite{Beckmann1956Studies-in-the-,Sheffi1985Urban-transport} and others, such as with wars of attrition~\cite{Bishop1978A-generalized-w}, that are obtained via matching from a stage normal-form game. The set of Nash equilibria of a population game can be functionally defined in the usual way, and may be interpreted in the mass-action sense discussed in~\cite{Weibull1995The-mass-action}, which was originally suggested in~\cite{Jr.1951Non-Cooperative}.

%The following subsection surveys results for the stochastic (Markov) approach
\subsection{Population games, protocols and stability}
\label{subsec:IntroPopGamesProtStab}
Our work focuses on the protocol classes surveyed in~\cite{Sandholm2010Population-Game,Weibull1995Evolutionary-ga}, which are inspired, to a great extent, on elementary bounded rationality mechanisms of evolutionary biology~\cite{Hofbauer1998Evolutionary-ga}. As is explained in~\cite[\S2.3]{Sandholm2010Population-Game} and~\cite{Sandholm2010Pairwise-compar}, the protocol models how much each agent knows about the payoff vector and the population state, and how it uses the available information to revise its strategy. Consequently, under the widely-used assumption that agents revise their strategies at times determined by a Poisson process, the resulting population state can be viewed as the Markov process associated with the system formed by the feedback interconnection between the evolutionary dynamics specified by the protocol and the population game. The analysis in~\cite{Taylor1978Evolutionarily-}, which shows that the evolutionarily stable strategies (ESS) characterized in~\cite{Maynard-Smith1973The-logic-of-an} are local attractors for the expectation of the population state, motivated the subsequent early work summarized in~\cite{Blume2018The-economy-as-} exploring a Markovian approach to establish in a probabilistic sense that in the long run the population state remains near, and in some cases converges to, certain Nash equilibria. Furthermore, the concept of stochastically stable set was defined and characterized in~\cite{Foster1990Stochastic-evol} to account for the effect of persistent stochastic perturbations and~\cite{Kandori1993Learning-mutati} extended these notions to examine equilibrium selection. 
% Followup by benaim, etc

% Rationale for the following subsection: start with a general argument for the mean population state, then in the subsubsection
% particularize the discussion for population games introducing the notion of mean dynamics, which later is then extended
% to mean closed loop model
\subsection{Mean population sate and deterministic payoff}
\label{sec:IntroMeanDynamicApproach}
In this article, we consider the case in which the population size is large enough that Khinchin's law of large numbers can be invoked to conclude that realizations of the population state at each instant will remain close to its expected value with high probability. This justifies our decision to focus on the so-called \underline{\it mean population state}, which akin to the approach in~\cite{Taylor1978Evolutionarily-} represents the expectation of the population state that can be obtained as the deterministic solution of a system of ordinary differential equations. 

Likewise, we restrict our analysis to \underline{\it deterministic payoffs} that are obtained in terms of the mean population state according to a soon to be described payoff dynamics model (PDM) of which population games are a particular case. As we argue precisely in \S\ref{sec:LargePopulationLimit}, the deterministic payoff will approximate the payoff of an associated finite population formulation with increasing fidelity as the size of the population tends to infinity.

\subsubsection{Nash equilibria of population games and long term behavior of the mean population state}
\label{sec:MeanDynamicsPopulationGames}
Before we outline the main methods and concepts put forth in this article, we proceed to surveying a few existing concepts and convergence results for the particular case in which a population game governs the deterministic payoff. In this context, the set of ordinary differential equations governing the mean population state, and consequently also the payoff, is commonly referred to as \underline{\it mean dynamic}.

As is discussed in~\cite{Sandholm2003Evolution-and-e} for a general class of protocols and population games, and further refined in~\cite{Benaim2003Deterministic-a}, results in~\cite{Kurtz1970Solutions-of-Or} guarantee that the mean population state approximates the population state with arbitrary accuracy uniformly over any given finite time horizon with probability approaching one as the number of agents tends to infinity, which highlights the importance of analyzing the long-term evolution of the mean population state and characterizing the associated globally attracting sets. Notably, the time horizon in the uniform approximation results reported in~\cite{Sandholm2003Evolution-and-e,Benaim2003Deterministic-a} can be selected large enough to guarantee that with high probability the population state of a sufficiently large population will tend to the smallest globally attractive set. As noted in~\cite{Blume2018The-economy-as-}, this conclusion has special relevance when a subset of Nash equilibria is globally attractive. It demonstrates the interesting fact that, even though the agents follow simple bounded rationality protocols that do not require knowledge of the structure of the population game, the population is still capable of self-organizing in a way that the population state tends to a subset of Nash equilibria. This not only reinforces the importance of the mass-action interpretation of Nash equilibria, but it also justifies regarding the globally attractive subsets of Nash equilibria as reliable predictors of the long-term behavior of the population. 

In many cases of interest it is possible to construct a Lyapunov function~\cite{Khalil1995Nonlinear-syste} establishing that the entire set of Nash equilibria is globally asymptotically stable and hence also attractive. The historical perspective and analysis in~\cite{Hofbauer2000From-Nash-and-B} popularized and fostered very significant work that uses the Lyapunov approach in conjunction with the so-called positive correlation and Nash stationarity concepts, which are particularly insightful in the context of evolutionary dynamics, to establish global attractivity or, in some cases, global asymptotic stability of the Nash equilibrium set for widely-used protocols and population game classes, such as potential~\cite{Sandholm2001Potential-games,Monderer1996Potential-games} and contractive\footnote{This class of population games is also referred to as ``stable" in~\cite{Hofbauer2007Stable-games}. Following the recent naming convention in~\cite{Sandholm2015Handbook-of-gam}, we prefer to use the qualifier ``contractive" to avoid confusion with other concepts of stability used throughout the article. Similarly, we will use ``strictly contractive" to qualify a population game that would be ``strictly stable" according to~\cite{Hofbauer2007Stable-games}.} games~\cite{Hofbauer2009Stable-games-an}. Of key importance in~\cite{Hofbauer2009Stable-games-an,Sandholm2001Potential-games} is the fact that, for a given constant payoff, the mean population state governed by a Nash stationary protocol is also constant if and only if it is a best response, which establishes an equivalence between the rest points of the mean dynamic and the Nash equilibria of the population game. The concepts of Nash stationarity and positive correlation are also explained in detail in~\cite[Chapter 5]{Sandholm2010Population-Game}, where they are viewed as properties of the mean dynamic associated with each protocol. Instead, here we adopt the convention, which will be useful later on, of viewing Nash stationarity and positive correlation as properties of the protocol.

Theorems~\cite[12.B.3]{Sandholm2010Population-Game} and~\cite[12.B.5]{Sandholm2010Population-Game}, which are derived from work in~\cite{Benaim1998Recursive-algor} and~\cite{Benaim1999Stochastic-appr}, offer yet another rationale associating a globally asymptotically stable set of equilibria of the mean dynamic, when one exists, with the long-term behavior of the population state for large populations. More specifically, these results ascertain under unrestrictive conditions that the measure, with respect to the stationary distribution of the population state, tends to one within any open set containing a globally asymptotically stable set as the population size tends to infinity. This implies that as the population grows, the stationary distribution of the population state tends to concentrate around the smallest globally asymptotically stable set.

\subsection{Outline of key concepts and methods}
\label{subsec:OutlineOfContrib}

From a dynamical systems theory perspective~\cite{Rugh1996Linear-system-t,Khalil1995Nonlinear-syste}, a population game would be qualified as {\it memoryless} because it acts as an instantaneous map from the mean population state to the deterministic payoff vector. Consequently, population games cannot capture dynamics in the payoff mechanism such as when there is inertia or anticipation effects in the agents' perception of the reward for each strategy. 

The following is an outline of the results, methods and concepts that undergird the framework put forth in this article:
\begin{enumerate}[label=(\roman*)]
	\item In \S\ref{sec:DPG}, we propose a class of payoff dynamics models (PDM) that includes as particular cases the dynamically modified payoffs analyzed in~\cite{Fox2013Population-Game}, of which the so-called smoothed and anticipatory payoff modifications modeling inertia and anticipatory effects, respectively, are examples. According to our formulation, each PDM is associated with a so-called {\it stationary} population game that determines the deterministic payoff in the stationary regime characterized by when the mean population state converges as time tends to infinity. 
	\item Given a PDM and a protocol satisfying Nash stationarity, in \S\ref{sec:NashStationaryEDMGlobalConvergence} we provide sufficient conditions determining when the set of Nash equilibria of the stationary population game is either globally asymptotically stable or globally attractive. We also specialize our results for the classes of integrable excess payoff target and impartial pairwise comparison protocols, of which the Brown-von Neumann-Nash and Smith protocols are, respectively, well-known examples.
	\item In \S\ref{sec:PBR}, we obtain sufficient conditions similar to those in (ii) for perturbed best response protocols, for which Nash stationarity does not hold. In this case, the conditions determine when a perturbed equilibrium set, which can be viewed as an approximation of Nash's, is either globally asymptotic stable or globally attractive.
	\item In \S\ref{sec:AnticipatoryPDM}, we determine the parameters of a PDM class, whereof the smoothed and anticipatory dynamically modified payoff considered in~\cite{Fox2013Population-Game} are particular cases, under which the sufficient conditions outlined in (i) and (ii) are satisfied. This includes cases in which the stationary population game is either affine, but unlike what is assumed in~\cite{Fox2013Population-Game} may not be strictly contractive, or has a strictly concave potential.
	
\end{enumerate}

\subsubsection*{Outline of our technical approach}

 We establish the results in (i) - (iv) by recognizing that in our paradigm the mean population state and the deterministic payoff are governed by a mean closed loop model, which is precisely defined in \S\ref{sec:closedloop} as the feedback interconnection between the PDM and the so-called \underline{\it evolutionary dynamics model}~(EDM) that captures the strategy revision dynamics as specified by the protocol.

As we discussed in \S\ref{sec:MeanDynamicsPopulationGames} for the case in which the payoff is determined by a population game, Lyapunov functions~\cite{Khalil1995Nonlinear-syste} are often used to establish convergence towards certain equilibria of the mean dynamic. The typical argument follows techniques derived from the classical principles in~\cite{Lasalle1960Some-extensions}. However, these techniques are not immediately applicable to our formulation because the deterministic payoff is no longer a memoryless function of the mean population state, which can no longer be characterized as a solution of the mean dynamic. Instead we resort to a well-known compositional approach~\cite{Arcak2016Networks-of-Dis} rooted on passivity principles with which convergence properties for the mean closed loop model can be ascertained by separately establishing certain passivity properties of the EDM and the PDM. More specifically, the search for a Lyapunov function, which was central for establishing convergence to equilibria of the mean dynamic, is superseded in our context by the construction of certain types of storage functions for the EDM and the PDM. 

Although the aforementioned passivity-based methodology, including the construction of storage functions for the individual blocks of a feedback loop to ascertain its stability~\cite{Willems1972Dissipative-dyn}, has been widely used~\cite{Sepulchre1997Constructive-no}, we propose our own form of passivity inspired on the approach in~\cite{Fox2013Population-Game}, which is both equilibrium independent and is compatible with the EDM and PDM classes we are interested in. In \S\ref{sec:PassivityConcepts} we introduce the passivity concepts used throughout the article. It includes in \S\ref{sec:PassivityComparison} a comparison with the related concepts of equilibrium independent passivity~\cite{Arcak2016Networks-of-Dis}, differential passivity~\cite{Forni2014A-differential-,Forni2013On-differential} and incremental passivity~\cite{Pavlov2008Incremental-pas}.

Most of the key results presented here date back to~\cite{Park2015Distributed-est}, and preliminary versions of the results in this article were reported in~\cite{Park2018Passivity-and-E2,Park2018Passivity-and-e}.

\subsection{Structure of the article}

%************************************************
% IMPORTANT: PUT SECTION LINKS!!!!!!!!!!!!!!!!!!!
% ***********************************************
% ADD MFG SECTION
%************************************************

Besides the introduction, the article has nine sections briefly described as follows:
\begin{itemize}
	\item In \S\ref{sec:PrelimsAndFormulation} we start by describing key preliminary concepts, including the detailed definitions of PDM, EDM and the mean closed loop model. We conclude the section with an overarching problem formulation that motivates the results derived throughout the paper.
	
	\item In \S\ref{sec:EngineeringExamplesEvDyn} we provide 3 specific examples on decision-making in large populations to which our results are applicable. In addition, we discuss further engineering examples illustrating	
% 	The examples discussed in \S\ref{sec:EngineeringExamplesEvDyn} illustrate
	applications of particular instances of the framework considered here -- such as when the PDM is a population game.
% 	and the EDM is based on replicator mechanisms
% 	-- to engineering problems.
	
	\item In \S\ref{sec:ComparionPopulationMeanF} we provide a comparison among evolutionary games, population games -- which we generalize here -- and mean field games. 
	\item Using an explicit finite population framework, in \S\ref{sec:LargePopulationLimit} we ascertain that the mean population state and deterministic payoff resulting from solutions of the mean closed loop model indeed approximate the population state and the payoff with arbitrary fidelity over any finite time horizon uniformly, as the number of agents tends to infinity.
	\item In \S\ref{sec:PassivityConcepts} we introduce the main concepts used throughout the article to characterize relevant passivity properties for any given PDM and EDM. More specifically, we define $\delta$-passivity allowing a possible surplus and $\delta$-antipassivity admitting a possible deficit for an EDM and a PDM, respectively. A weak version of $\delta$-antipassivity is also proposed and invoked in Lemma~\ref{lem:MainLemma} to ascertain sufficient conditions that will be used in \S\ref{sec:NashStationaryEDMGlobalConvergence} and~\S\ref{sec:PBR} to establish global attractivity results. 
	\item  A detailed analysis in \S\ref{sec:NashStationaryEDMGlobalConvergence} for the case in which the EDM satisfies Nash stationarity leads to methods to determine whether the set of Nash equilibria of the stationary game associated with a given PDM is globally attractive or globally asymptotically stable. We also specialize our results for the integrable excess payoff target (EPT) and impartial pairwise comparison (IPC) protocols.
	\item In \S\ref{sec:PBR} we present results that are analogous to \S\ref{sec:NashStationaryEDMGlobalConvergence} for perturbed best response (PBR) protocols. In addition, we establish a fundamental trade-off according to which a perturbed equilibrium set is globally attractive even when the PDM has a $\delta$-antipassivity deficit provided that it is no larger than the $\delta$-passivity surplus of the EDM. This trade-off is not possible in the context of \S\ref{sec:NashStationaryEDMGlobalConvergence} because, as is shown in~\cite[Corollary~IV.3]{Park2018Passivity-and-e}, any EDM stemming from an  EPT or IPC protocols never have a positive $\delta$-passivity surplus. 
	
	\item In \S\ref{sec:AnticipatoryPDM} we propose the so-called smoothing-anticipatory PDM class and characterize its $\delta$-antipassivity properties. 
% 	We also include examples illustrating how the $\delta$-antipassive properties of the aforementioned PDM can be used in conjunction with the results in \S\ref{sec:NashStationaryEDMGlobalConvergence} or~\S\ref{sec:PBR} to establish global stability of the Nash equilibria of the stationary game or a perturbed version, respectively.
	
	\item In \S\ref{sec:exmaples} we provide numerical examples and carry out simulations to demonstrate application of our framework and key stability results, discussed in \S\ref{sec:NashStationaryEDMGlobalConvergence} or~\S\ref{sec:PBR}, to large population decision-making problems. In particular, using the examples, we qualitatively examine the effects of payoff anticipation and smoothing on convergence properties of the mean closed loop model, and investigate the role of the $\delta$-passivity surplus in establishing stability of Nash equilibria.
	
\end{itemize}

\section{Preliminary Concepts and Problem Formulation}
\label{sec:PrelimsAndFormulation}

% DEFINE THE MAIN COMPONENTS: MODIFIED PAYOFF MAP AND EVOLUTIONARU DYNAMIC
% SAY THAT THE CLASSICAL $P=F(X)$ IS AN EXAMPLE, BUT HERE WE ALLOW THE
% CASES DESCRIBED IN SECTION ... 

Before we specify the class of problems considered throughout this article, we proceed to defining and discussing basic properties of the sets and maps needed for our analysis. 

In \S\ref{sec:DPG}, we will propose a class of payoff dynamics models (PDM) comprising a state that evolves in (continuous) nonnegative real-valued time $\mathbb{T} = \mathbb{R}_+$. In our single-population framework, the strategy profile set is specified as follows:
$$\mathbb{X} \overset{\mathrm{def}}{=} \left \{z \in \mathbb{R}^n_+ \ \left| \ \sum_{j=1}^n z_j = m \right. \right \} $$
where $n$ is the number of strategies and $m$ is the population mass. Every element of $\mathbb{X}$ is interpreted as a n-tuple representing the portions of the population adopting each strategy. 

 For any vector $r$ in $\mathbb{R}^n$, we specify the following norm:  $$ \| r \| \overset{\mathrm{def}}{=} \max_{1 \leq i \leq n} |r_i|$$  For any trajectory $v: \mathbb{T} \rightarrow \mathbb{R}^n$, we adopt the following norm $$\|v\| \overset{\mathrm{def}}{=} \sup_{t \in \mathbb{T}} \|v(t) \|$$

 We restrict our analysis to trajectories in the following set: 
 $$\mathfrak{R}^n \overset{\mathrm{def}}{=} \{ \text{All differentiable maps from $\mathbb{T}$ to $\mathbb{R}^n$}\}$$

\begin{definition} We reserve ${x: \mathbb{T} \rightarrow \mathbb{X}}$ to denote the so-called mean population state trajectory. At a particular time $t$ the mean population state is represented by $x(t)$. The set of admissible mean population state trajectories $\mathfrak{X}$ is defined as follows:  
$$\mathfrak{X} \overset{\mathrm{def}}{=} \big \{ x \in \mathfrak{R}^n \ \big | \ \| \dot{x} \| < \infty \ \text{and} \ x(\tau) \in \mathbb{X}, \forall\tau \in \mathbb{T} \big \}$$
\end{definition}

\begin{definition} We use ${p: \mathbb{T} \rightarrow \mathbb{R}^n}$ to denote the so-called deterministic payoff trajectory. The deterministic payoff at time~$t$ is represented by $p(t)$, whose entries ascribe a reward to each strategy. The set of possible deterministic payoff trajectories $\mathfrak{P}$ is specified as:
$$\mathfrak{P} \overset{\mathrm{def}}{=} \big \{ p \in \mathfrak{R}^n \ \big | \ \| p \| < \infty \ \text{and} \ \| \dot{p} \| < \infty \big \}$$
\end{definition}

In \S\ref{sec:LargePopulationLimit}, we describe a scenario in which the mean population state and deterministic payoff are high fidelity approximations of the population state and payoff of a finite population, respectively, provided that the number of agents is sufficiently large.

We now proceed with defining the main models and concepts needed to complete the description of our framework.

\subsection{Payoff Dynamics Model (PDM)}
\label{sec:DPG}

We start by defining population games as used in the framework described in \S\ref{subsec:IntroPopGamesProtStab}.

\begin{definition} {\bf (Population game)} A population game is specified by a continuous map $\mathcal{F}:\mathbb{X} \rightarrow \mathbb{R}^n$. It determines the deterministic payoff trajectory as $p(t)=\mathcal{F}\big(x(t)\big),~t \geq 0$. 
\end{definition}

An important class of population games is defined below:

\begin{definition}
\label{def:StableGame} {\bf (Contractive Population Game)} A given population game $\mathcal{F}: \mathbb{X} \rightarrow \mathbb{R}^n$ is qualified as contractive if it satisfies the following condition:
$$ (z-\bar{z})^T \big( \mathcal{F}(z)-\mathcal{F}(\bar{z}) \big) \leq 0, \quad z,\bar{z} \in \mathbb{X} $$ The population game is said to be strictly contractive if the inequality above is strict for $z \neq \bar{z}$.
\end{definition}

In contrast with the population game formulation in which $p$ is obtained via a memoryless\footnote{Here we adopt a convention that is common in systems theory, according to which a system is called \underline{memoryless} when the output is an instantaneous (point-wise in time) function of the input.} map of $x$, we consider that a payoff dynamics model (PDM), as defined below, governs the deterministic payoff trajectory in terms of the mean population state trajectory. 
\begin{definition} 
\label{def:PDM}
{\bf (PDM)}
We consider that a payoff dynamics model (PDM) is defined as follows:\footnote{For concise presentation, throughout the paper, we suppose that the dimension of the vector $q(t)$ is set to be $n$, which is same as the number of strategies. In general, the domain of $q(t)$ can be any finite-dimensional space of real vectors.}
\begin{equation}
\label{eq:PDM}
\quad \begin{matrix}
\dot{q}(t) = & \mathcal{G} \big (q(t),u(t) \big ) \\
p(t)  = &\mathcal{H} \big (q(t),u(t) \big )
\end{matrix}, \quad t \geq 0, \ q(0) \in \mathbb{R}^n, \ u \in \mathfrak{X}
\end{equation}
where $\mathcal{G}: \mathbb{R}^n \times \mathbb{X} \rightarrow \mathbb{R}^n $ is Lipschitz continuous and ${\mathcal{H}: \mathbb{R}^n \times \mathbb{X} \rightarrow \mathbb{R}^n}$ is continuously differentiable and Lipschitz continuous. The state-space representation (\ref{eq:PDM}) specifies a map from the input $u$ and the initial condition $q(0)$ to the state $q$ and output $p$. Henceforth, we identify a PDM by the pair $(\mathcal{G},\mathcal{H})$ that specifies it. 
\end{definition}

The PDM is a finite-dimensional dynamical system that, as we will see in \S\ref{sec:closedloop}, will be used in a closed loop configuration in which $u$ is set to the mean population state trajectory $x$.

In addition, we require that the PDM satisfies the following assumption for reasons that will become clear in Proposition~\ref{prop:uniquesoln}.

\begin{assumption} \label{assump:BoundedPDM} {\bf (PDM Boundedness)} In this article we consider that every PDM is a bounded map in the following sense:
\begin{equation}
\sup_{u \in \mathfrak{X}} \|q\| < \infty, \quad q(0) \in \mathbb{R}^n
\end{equation} where $q$ is the solution of (\ref{eq:PDM}) for input $u$ and initial condition~$q(0)$.  
\end{assumption}

\begin{remark} 
\label{rem:bounded}
Boundedness of a PDM guarantees that for each initial condition $q(0)$ in $\mathbb{R}^n$ there is a finite upper-bound for $\|q\|$ that holds for all inputs $u$ in $\mathfrak{X}$. In addition, because $\mathcal{G}$ is Lipschitz continuous, $\mathcal{H}$ is continuously differentiable, and the elements of $\mathfrak{X}$ are bounded, we infer that if the PDM is bounded then $p$, $\dot{q}$, and $\dot{p}$ are also bounded, for each initial condition $q(0)$.
\end{remark}

We also assume that every PDM is associated with a so-called \underline{stationary population game} according to the following assumption.

\begin{assumption} {\bf (Stationary Population Game of a PDM)} \label{assumption:stationary_game}
  \label{assump:StationaryGame} In this article, we assume that for every PDM there is a continuous map $\bar{\mathcal{F}}:\mathbb{X} \rightarrow \mathbb{R}^n$ for which the following implication holds for any initial condition $q(0)$ in $\mathbb{R}^n$:
  \begin{equation} \lim_{t \rightarrow \infty} \| \dot{u}(t) \| = 0 \implies \lim_{t \rightarrow \infty} \left \| p(t) -\bar{\mathcal{F}} \big ( u(t) \big) \right \| = 0, \quad u \in \mathfrak{X}
  \end{equation}
  in addition, we require that the following set $$\left\{ (s,z) \in \mathbb R^n \times \mathbb X \,\big| \,\mathcal H(s,z) = \bar{\mathcal F}(z) \right\}$$ is either a compact subset of or the entire set $\mathbb R^n \times \mathbb X$. We refer to $\bar{\mathcal{F}}$ as the \underline{stationary population game} of the PDM.
\end{assumption}

The following definition clarifies the way in which population games can be regarded as a particular PDM class.

\begin{definition}
\label{def:MemorylessPDM}
{ \bf (Memoryless PDM)} Let $\mathcal{F}:\mathbb{X} \rightarrow \mathbb{R}^n$ be a continuously differentiable\footnote{Notice that since $\mathbb{X}$ is compact, continuous differentiability of $\mathcal{F}$ implies that it is also Lipschitz continuous.} map that defines a population game. We refer to $\mathcal{F}$ interchangeably as a population game or a \underline{memoryless PDM}.
\end{definition}

It suffices to choose $\mathcal{H}(s,z) = \mathcal{F}(z)$ and, although it is not necessary, adopt $\mathcal{G}(s,z) = \mathcal F(z) - s$ to conclude that the memoryless PDM complies with Definition~\ref{def:PDM}. It also follows immediately from this construct that the stationary population game $\bar{\mathcal{F}}$ is~$\mathcal{F}$.

\begin{definition}
\label{def:NashForPDM} Given a PDM, the set of Nash equilibria of its stationary population game is defined as follows:
$$ \mathbb{NE}(\bar{\mathcal{F}}) \overset{\mathrm{def}}{=} \big\{ z \in \mathbb{X} \ \big| \ z^T \bar{\mathcal{F}}(z) \geq \bar{z}^T \bar{\mathcal{F}}(z), \ \bar{z} \in \mathbb{X} \big\} $$
\end{definition}

If a PDM is memoryless then $\mathbb{NE}(\bar{\mathcal{F}})$ is the Nash equilibrium set of the underlying population game $\mathcal{F}$. 

% IMPORTANT TO DO ********************
%
% COMMENT HERE THAT FOR POPULATION GAMES (MEMORYLESS PDM) THIS CONCEPT OF 
% NASH EQUILIBRIUM IS IDENTICAL TO THE USUAL ONE FOR POPULATION GAMES
% ACTUALLY MAKE THE DEFINITION TO INCLUDE THE DEFINITION OF NASH EQ FOR
% POPULATION GAMES. HERE ALSO REMIND OF THE MASS ACTION PRINCIPLE

 The results in this article are valid for any PDM satisfying the assumptions above, and in \S\ref{sec:MemorylessPDM} and~\S\ref{sec:AnticipatoryPDM}, we define and analyze in more detail two cases of interest.

\subsection{Evolutionary Dynamics Model (EDM)}
\label{Sec:ED}

Our model for how the state of the population evolves over time is specified as follows.

\begin{definition} \label{def:edm} {\bf (EDM)} The evolutionary dynamics of the population is specified by an evolutionary dynamics model (EDM) of the following form:
\begin{equation} 
\label{PopulationDynamic}
\dot{x}(t) = \mathcal{V} \big (x(t),w(t) \big ), \quad x(0) \in \mathbb{X}, \ t \geq 0, \ w \in \mathfrak{P}
\end{equation}
where $\mathcal{V}: \mathbb{X} \times \mathbb{R}^n \rightarrow \mathbb{R}^n $ is a given Lipschitz 
continuous map that also satisfies the following constraint:
\begin{equation} 
\label{ConeConstraint}
\mathcal{V}(z,r) \in \mathbb{TX}(z), \quad z \in \mathbb{X}, \  r \in \mathbb{R}^n
\end{equation}
where $\mathbb{TX}(z)$ is the cone of viable displacement from $z$ specified as: 
$$\mathbb{TX}(z) \overset{\mathrm{def}}{=} \big \{\alpha (\bar{z}-z) \ | \ \bar{z} \in \mathbb{X}  \ \text{and} \ \alpha \geq 0 \big \}, \quad z \in \mathbb{X}$$ Here, $x(t)$ is the state, which is also the output of the model, $x(0)$ is the initial condition, and $w$ is the input that is typically taken to be a given deterministic payoff trajectory.
\end{definition} 

Inspired by the notation in~\cite{Sandholm2010Population-Game}, we also represent the subspace tangent to $\mathbb{X}$ as follows:
$$ \mathbb{TX} \overset{\mathrm{def}}{=} \left \{ \tilde z \in \mathbb{R}^n \ \Bigg| \ \sum_{i=1}^n \tilde z_i =0  \right \} $$ 

\begin{remark}{\bf (``EDM" versus ``deterministic evolutionary dynamics")} We have decided to use ``evolutionary dynamics model", or EDM, to refer to~(\ref{PopulationDynamic}) to distinguish it from the closely related ``deterministic evolutionary dynamics" concept discussed in~\cite[\S~4.4]{Sandholm2010Population-Game}, which refers to a given map from the set of population games to a set of trajectories consistent with the mean dynamic specified in~\cite[(M), p.124]{Sandholm2010Population-Game}. The fact that we can assign the deterministic payoff as the input of the EDM, irrespective of whether it is generated by a PDM or a population game, will facilitate a rigorous analysis for our framework later on. 
\end{remark}

In this article, we will consider a subclass of EDM that stems from the following class of revision protocols.

\begin{definition}{\bf (Protocol)} \label{def:protocol}
  A revision protocol, or {\it protocol} for short, is specified by a Lipschitz continuous map ${\mathcal{T}:\mathbb{R}^n \times \mathbb{X} \rightarrow \mathbb{R}^{n \times n}_{+}}$. It models how members of the population revise their strategies in response to the population state and payoff. In particular, it determines the EDM as follows:
\begin{multline}
\label{eq:MeanDynamic} \mathcal{V}_i(z,r) \overset{\mathrm{def}}{=} \sum_{j=1}^n z_j \mathcal{T}_{ji}(r,z) -  \Bigg ( \sum_{j=1}^n \mathcal{T}_{ij} (r,z) \Bigg ) z_i , \\  z \in \mathbb{X}, \ r \in \mathbb{R}^n, \  1 \leq i \leq n
\end{multline}
\end{definition}

In \S\ref{sec:closedloop}, we model the feedback mechanism that is established when the PDM is actually played by a population characterized by a given protocol and corresponding EDM. We also show in \S\ref{sec:LargePopulationLimit} how the mean population state approximates the state of a finite population with arbitrary accuracy as the number of agents tends to infinity.

Most of our basic notation is summarized in Tables~\ref{tab:BasicNotationPartI} and~\ref{tab:BasicNotationPartII}, while most acronyms and notation associated with EDM and PDM, some of which will be introduced later on, are summarized in Table~\ref{tab:EDMPDMAcronymsNotation}. 

\begin{table}

\begin{tabular}{l | l}
n & number of strategies. \\
m & population mass. \\
$\mathbb{R}^n$ & $n$-th dimensional real coordinate space. \\
$\mathbb{R}^n_{+}$ & $n$-th dimensional nonnegative real coordinate space. \\
$\mathbb{R}^n_{-}$ & $n$-th dimensional nonpositive real coordinate space. \\
$\mathbb{R}^n_{*}$ & represents viable excess payoff vectors $\big( \mathbb{R}^n-\mathrm{int}(\mathbb{R}^n_{-}) \big)$. \\
$\mathbb{X}$ & strategy profile set.  \\
$\mathrm{int}(\mathbb{X})$ & is defined as $\{ z \in \mathbb{X} \ | \ z_i >0, \ 1\leq i \leq n\}$. \\
$\mathrm{bd}(\mathbb{X})$ & is given by $\mathbb{X} - \mathrm{int}(\mathbb{X})$. \\
\end{tabular}
\vspace{.1in}
\caption{Partial list of basic notation: Part~I.}
\label{tab:BasicNotationPartI}
\end{table}

\begin{table}

\begin{tabular}{l | l}
$\mathcal{F}$ & payoff map specifying a population game. \\
$z$ &  represents a given strategy profile. \\
$r$ & represents a given payoff vector.\\
$\mathcal{DF}$ & Jacobian of $\mathcal{F}$. \\
$\mathbb{TX}$ & subspace tangent to $\mathbb{X}$. \\
$\mathbb{TX}(z)$ & cone tangent to $\mathbb{X}$ at $z$. \\
$\mathbb{T}$ & nonnegative continuous time algebraically equivalent to $\mathbb{R}_+$. \\
$\mathfrak{X}$ & set of possible mean population state trajectories. \\
$x$ & mean population state trajectory.\\
$\mathfrak{P}$ & set of possible deterministic payoff trajectories. \\
$p$ & deterministic payoff trajectory.\\
\end{tabular}
\vspace{.1in}
\caption{Partial list of basic notation: Part~II.}
\label{tab:BasicNotationPartII}
\end{table}

\begin{table} 

\begin{tabular}{l | l}
PDM & payoff dynamics model. \\
$q$ & state trajectory of the PDM. \\
% $q$ & element of $\mathfrak{P}$ representing the state trajectory of the PDM. \\
$\mathcal{G}$ & map specifying the differential equation for $q$. \\
$\mathcal{H}$ & map specifying $p$ based on $q$ and input to the PDM. \\ 
$\bar{\mathcal{F}}$ & stationary population game of a PDM. \\
$\mathbb{NE}(\bar{\mathcal{F}})$ & set of Nash equilibria of $\bar{\mathcal{F}}$. \\ \hline \hline
EDM & evolutionary dynamics model. \\
$\mathcal{V}$ & map specifying the differential equation modeling the EDM. \\
$\mathcal{T}$ & protocol determining $\mathcal{V}$. \\ \hline
EPT & EDM associated with an excess payoff target protocol. \\
BNN & (Brown-von Neumann-Nash) particular case of EPT EDM. \\ \hline
IPC & EDM associated with impartial pairwise comparison protocol.  \\
Smith & particular case of IPC EDM. \\ \hline
PBR & EDM associated with perturbed best response protocol. \\
$\mathcal{Q}$ & perturbation map associated with the PBR EDM. \\
$\tilde{\mathcal{F}}$ & $\bar{\mathcal{F}}$ perturbed by $\mathcal{Q}$ as $\tilde{\mathcal{F}}(z)=\bar{\mathcal{F}}(z) - \nabla \mathcal{Q}(z)$. \\
$\mathbb{PE}(\bar{\mathcal{F}},\mathcal{Q})$ & set of perturbed equilibria $\mathbb{NE}(\tilde{\mathcal{F}})$. \\
Logit & particular case of PBR EDM. \\
\end{tabular}
\vspace{.1in}
\caption{Partial list of most acronyms and notation associated with PDM and EDM.}
\label{tab:EDMPDMAcronymsNotation}
\end{table}

\subsection{Solutions of the Mean Closed Loop Model}
\label{sec:closedloop}

We wish to study the trajectories of the deterministic payoff and mean population state
when the PDM is interconnected in feedback with an EDM. That is to
say that we consider that the input of the PDM is the mean population state and, in turn, the resulting deterministic payoff is the input of the EDM (see Fig.~\ref{Fig:ClosedLoop}). A state space closed loop model of this feedback interconnection is constructed as follows.

\tikzstyle{system} = [draw, ultra thick, minimum width=6em, text centered, rounded corners, fill=yellow!10,  minimum height=3em]

\def\edgedist{1.1}
\begin{figure} [t]
\centering
\begin{tikzpicture}[auto, node distance=7em,>=latex]

\node [system, minimum height=10ex] (games) {$\begin{array} {c} \textit{Payoff Dynamics Model} \\ \begin{array} {c} \dot q(t) = \mathcal{G} \left( q(t), u(t) \right) \\ p(t) = \mathcal{H} \left( q(t), u(t) \right) \end{array} \textit{(PDM)} \end{array}$};
\draw [->,color=green!15!black] (games)+(0,8ex) node[above,color=green!15!black]{$q(0)$} -- (games);

\node [system, minimum height=8ex, below of=games] (evo_dynamics) {$\begin{array} {c} \textit{Evolutionary Dynamics Model} \\ \begin{array} {c} \dot{x}(t) = \mathcal V(x(t), w(t)) \end{array} \textit{(EDM)} \end{array}$};
\draw [->,color=green!15!black] (evo_dynamics)+(0,-7ex) node[below,color=green!15!black]{$x(0)$} -- (evo_dynamics);

\draw[ultra thick,color=green!20!black, ->] (games.east) -- ++(1*\edgedist,0) node [pos=.5] {$p$} |- (evo_dynamics.east) node [above, pos=.25, rotate=-90] {\textit{deterministic payoff}} node [near end] {\footnotesize $(w=p)$};
\draw[ultra thick,color=green!20!black, ->] (evo_dynamics.west) -- ++(-1*\edgedist,0) node [pos=.5] {$x$} |- (games.west) node [above, pos=.25, rotate=90] {\textit{mean population state}} node [near end] {\footnotesize $(u=x)$};

\end{tikzpicture}
\caption{Diagram representing a feedback interconnection between a PDM and an EDM. The resulting system is referred to as the \underline{mean closed loop model}.}
\label{Fig:ClosedLoop}
\end{figure}
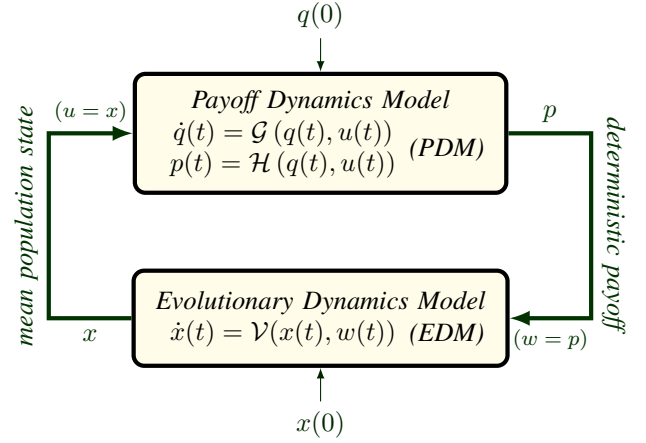

% \begin{figure}
% \begin{center}
% \includegraphics[width=3in]{FigClosedLoopPDF.pdf}
% \caption{Diagram representing a closed loop interconnection between a PDM and an EDM representing a dynamic population game and an evolutionary dynamic, respectively.}
% \label{Fig:ClosedLoop}
% \end{center}
% \end{figure}

\begin{definition} {\bf (Mean Closed Loop Model)} Given a PDM and an EDM, the associated mean closed loop model
is obtained by using $x$ as an input to (\ref{eq:PDM}) and $p$ as an input to (\ref{PopulationDynamic}), or equivalently, setting $u=x$ and $w=p$. The following is the state space representation of the mean closed loop model:
\begin{equation}
\label{eq:ClosedLoop}
\quad \begin{matrix}
\dot{q}(t) = & \mathcal{G} \big (q(t),x(t) \big ) \\
\dot{x}(t) = &\mathcal{V}^{\mathcal{H}} \big (q(t),x(t) \big ) \\
p(t)  = &\mathcal{H} \big ( q(t),x(t) \big ) 
\end{matrix}, \quad t \geq 0, \ \big (q(0),x(0) \big ) \in \mathbb{R}^n \times \mathbb{X}, 
\end{equation} where $\mathcal{V}^{\mathcal{H}}: \mathbb{R}^n \times \mathbb{X} \rightarrow \mathbb{R}^n$ is defined as:
$$ \mathcal{V}^{\mathcal{H}}(s,z) \overset{\mathrm{def}}{=} \mathcal{V} \left(z,\mathcal{H}(s,z) \right), \quad \ s \in \mathbb{R}^n, z \in \mathbb{X}$$
Notice that the state of the mean closed loop model is $\big( q(t),x(t) \big)$. The mean closed loop model has no input, and the state trajectory $(q,x )$ is the solution of the associated initial value problem in response to the pre-selected initial condition $\big( q(0),x(0) \big)$.
\end{definition} 

In \S\ref{sec:LargePopulationLimit}, we show that solutions of~(\ref{eq:ClosedLoop}) approximate uniformly, over any finite time horizon, with arbitrary accuracy the population state of a finite population subjected to the same protocol and PDM with probability tending to one as the number of agents tends to infinity. 

It is clear from our PDM and EDM definitions that $\mathcal{G}$ and $\mathcal{V}^{\mathcal{H}}$ are Lipschitz continuous, which, from Picard-Lindel{\"o}f Theorem, guarantees the existence and uniqueness of solutions of the initial value problem in (\ref{eq:ClosedLoop}). Also the boundedness of the PDM from Assumption~\ref{assump:BoundedPDM} and Remark~\ref{rem:bounded} guarantees that $p$ and $x$ remain in $\mathfrak{P}$ and $\mathfrak{X}$, respectively. We assert these facts in the following proposition, which is an immediate consequence of the Picard-Lindel{\"o}f Theorem and Remark~\ref{rem:bounded}.

\begin{proposition}
\label{prop:uniquesoln} Consider that a mean closed loop model (\ref{eq:ClosedLoop}) is formed by an EDM and a PDM. For each initial condition $\big (q(0),x(0) \big )$ in $\mathbb{R}^n \times \mathbb{X}$ there is a unique solution $(q,x)$. The solution is such that $p$ and $x$ are in $\mathfrak{P}$ and $\mathfrak{X}$, respectively.
\end{proposition}

Furthermore, under the conditions of Proposition~\ref{prop:uniquesoln}, an immediate adaptation of the argument in \cite[Theorem~4.A.3]{Sandholm2010Population-Game} can be used to show that the map ${\phi_t: \big (q(0),x(0) \big ) \mapsto (q(t),x(t) \big)}$ is Lipschitz continuous for any $t$ in $\mathbb{T}$.

\subsection{Background on the Approach of~\cite{Fox2013Population-Game}}

As we discussed earlier, standard population games are specified by a continuous map $\mathcal{F}:\mathbb{X} \rightarrow \mathbb{R}^n$ that determines the deterministic payoff at time $t$ according to $p(t) = \mathcal{F} \big ( x(t) \big)$. The map is memoryless because the (deterministic) payoff at time $t$ depends solely on the (mean) population state also at time $t$. 

% An important class of population games is defined below:

% \begin{definition}
% \label{def:StableGame} {\bf (Contractive Population Game)} A given population game $\mathcal{F}: \mathbb{X} \rightarrow \mathbb{R}^n$ is qualified as contractive if it satisfies the following condition:
% $$ (z-\bar{z})^T \big( \mathcal{F}(z)-\mathcal{F}(\bar{z}) \big) \leq 0, \quad z,\bar{z} \in \mathbb{X} $$ The population game is said to be strictly contractive if the inequality above is strict for $z \neq \bar{z}$.
% \end{definition}
As shown in~\cite{Hofbauer2009Stable-games-an}, when a contractive population game $\mathcal{F}$ is played according to an EDM derived from a target protocol that is integrable, or equivalently has a revision potential, the trajectory of the mean population state is guaranteed to have global convergence properties, which are established using Lyapunov-like functions that are constructed from the revision potentials.
Work in \cite{Fox2013Population-Game} proposed a framework to extend the contractive population game methodology to certain dynamically modified payoffs, which in our formulation could be modeled as PDMs. The underlying idea in \cite{Fox2013Population-Game} is to use techniques related to $\delta$-passivity for establishing sufficient conditions that determine when the time derivative of the state trajectory of~(\ref{eq:ClosedLoop}) is $\mathcal{L}_2$ stable.

While, in~\cite{Hofbauer2009Stable-games-an}, contractivity was a requirement for convergence in the case of population games, the approach by~\cite{Fox2013Population-Game} for dynamically modified payoffs requires that the PDM satisfies an inequality defining a condition called {\it $\delta$-antipassivity}. Interestingly, \cite{Fox2013Population-Game} shows that $\delta$-antipassivity can be rightly viewed as a generalization of contractivity because any contractive population game is also $\delta$-antipassive. Notably, the analysis in~\cite{Fox2013Population-Game} proves that when the PDM is $\delta$-antipassive and EDM results from an integrable target protocol then the revision potential can be used to construct a storage function to establish $\mathcal{L}_2$ stability of the solutions of~(\ref{eq:ClosedLoop}). In general, when such a storage function can be constructed, the EDM is shown to be $\delta$-passive, or equivalently, it satisfies an inequality that is the antisymmetric of the one used to define $\delta$-antipassivity. We will revisit these concepts in \S\ref{sec:PassivityConcepts}.

\subsection{Problem Statement and Comparison with~\cite{Hofbauer2009Stable-games-an} and~\cite{Fox2013Population-Game}} 

The analysis in \cite{Fox2013Population-Game} investigates stability of the derivative of the mean population state in the $\mathcal{L}_2$ sense. Remarkably, stability guarantees of the aforementioned sort do not imply convergence of the mean population state towards any specific equilibria as was done, for instance, in~\cite{Hofbauer2009Stable-games-an}. In particular, we cannot use existing versions of Barbalat's lemma~\cite{Farkas2016Variations-on-B} because they presume constraints on the second derivative of the mean population state that may not be verified in our framework in which protocols may not be differentiable. In addition, even if some version of Barbalat's lemma could be constructed to determine when the derivative of the mean population state tends to zero\footnote{In fact, even for elementary real-valued functions, convergence of the derivative to zero as time tends to infinity does not, in general, guarantee that the function converges to a constant. The function defined for nonnegative $t$ by $\log(t+1)$ is a simple example illustrating this fact. }, it would not suffice to ascertain whether Nash equilibrium set, or perturbed versions of it, are globally attractive or globally asymptotically stable.

\begin{problem}
\label{prob:MainProblem}
{\bf (Main Problem)} Consider that a PDM is given and that $\bar{\mathcal{F}}$ is its stationary population game. For the EDM classes associated with the protocol classes considered in~\cite{Hofbauer2009Stable-games-an}, we seek to determine the conditions on the PDM to guarantee that a set of equilibria $\mathbb{E}(\bar{\mathcal{F}})$, which in this article is either $\mathbb{NE}(\bar{\mathcal{F}})$ or a perturbed version of it, is globally attractive or globally asymptotically stable according to the following definitions.
\end{problem}

\begin{definition} \label{def:stability} {\bf (Stability Concepts)} Let an EDM and a PDM be given. The stability of a set $\mathbb{E}(\bar{\mathcal{F}})$, which in our formulation is either $\mathbb{NE}(\bar{\mathcal{F}})$ or a perturbed version of it, is classified as follows:
\begin{itemize}
\item {\bf (Global Attractiveness)} The set $\mathbb{E}(\bar{\mathcal{F}})$ is \underline{globally attractive} if for every initial condition $(q(0),x(0))$ in $\mathbb{R}^n \times \mathbb{X}$, the solution trajectory $(q,x)$ of the mean closed loop model (\ref{eq:ClosedLoop}) is such that the following limits hold:
\begin{subequations}
\begin{align}
\lim_{t \rightarrow \infty} \left ( \inf_{z \in \mathbb{E}(\bar{\mathcal{F}}) } \|x(t) -z \| \right ) &=0 \\
\lim_{t \rightarrow \infty} \big \|p(t) - \bar{\mathcal{F}}\big( x(t) \big ) \big \| &= 0
\end{align}
\end{subequations}
\item {\bf (Lyapunov Stability)} The set $\mathbb{E}(\bar{\mathcal{F}})$ is Lyapunov stable if for every open set $\mathbb{O}$ in $\mathbb{R}^{2n}$ that contains  
\begin{multline*}
  \mathbb{A} \overset{\mathrm{def}}{=} \{(s,z) \in \mathbb R^n \times \mathbb X \ \big| \ z \in \mathbb{E}(\bar{\mathcal{F}}) \\ \text{ and } \mathcal H(s,z) = \bar{\mathcal{F}}(z)\},  
\end{multline*}
there is another open set $\mathbb{O}'$ that contains $\mathbb{A}$ for which the following holds:
\begin{multline}
 (q(0),x(0)) \in \mathbb{O}' \cap (\mathbb{R}^n \times \mathbb{X}) \ \implies  \\ (q(t),x(t)) \in \mathbb{O}, \ t \geq 0
\end{multline}
\item {\bf (Global Asymptotic Stability)} The set $\mathbb{E}(\bar{\mathcal{F}})$ is globally asymptotically stable if it is globally attractive and Lyapunov stable.
\end{itemize}

This definition stated above adapts to our framework analogous concepts used in~\cite{Hofbauer2009Stable-games-an}.
\end{definition}

\subsubsection*{Our Approach to Solving Problem~\ref{prob:MainProblem}}
In addition to leveraging the use of and generalizing the passivity techniques proposed in \cite{Fox2013Population-Game} to solve Problem~\ref{prob:MainProblem}, in \S\ref{sec:PassivityConcepts} we introduce new key concepts that ultimately lead to Lemma~\ref{lem:MainLemma}. The latter is a central convergence result that will allow us to propose solutions to Problem~\ref{prob:MainProblem} for the two important EDM classes studied in \S\ref{sec:NashStationaryEDMGlobalConvergence} and~\S\ref{sec:PBR}. It will follow from our analysis in these sections that, in the particular case when the PDM is a memoryless map specifying a population game, our sufficient conditions for global attractiveness and global asymptotic stability are analogous to those in~\cite{Hofbauer2009Stable-games-an}.

In \S\ref{sec:AnticipatoryPDM}, we also extend and adapt to our context an important PDM class proposed in~\cite{Fox2013Population-Game}. We also provide extensions of \cite[Theorem~4.6]{Fox2013Population-Game}.

%%%%%%
%%%%% section{Examples of Application of Population Games and Evolutionary Dynamics in Engineering Systems and Optimization}
%%%%%%%%%%%%

\section{Population Games and Evolutionary Dynamics: Examples in Engineering Systems and Optimization}
\label{sec:EngineeringExamplesEvDyn}

We start by describing three examples (Examples~\ref{example:conjection5link}-\ref{ex:TaskAllocation}) that illustrate how population games can be used to formulate decision-making problems in multi-agent engineering systems. Later on in \S\ref{sec:exmaples}, we will illustrate how our methodology can be used to ascertain the stability of the mean closed loops whose PDMs are the population games specified in the examples. More importantly, in \S\ref{sec:exmaples} we also explain how our methods can be used to characterize stability when the PDM is a dynamic modification of the population games in the examples. 

%After we present our approaches on passivity-based stability analysis, in \S\ref{sec:exmaples}, using the same examples we discuss how the PDM framework can be used to model agents' learning, inertia, and anticipation, and explain the implication of our stability results.

\begin{figure} [t]
\centering

\subfigure[]{
\begin{tikzpicture}[auto, node distance=7em,>=latex']

\node [circle,draw=black, fill=green!20, inner sep=0pt,minimum size=20pt] (origin) at (0,0) {\bf \large O};

\node [circle,draw=black, fill=red!20, inner sep=0pt,minimum size=20pt] (destination) at (5,0) {\bf \large D};

\node [circle,draw=black, fill=black!25, inner sep=0pt,minimum size=5pt] (vertex1) at (2.5,1.5) {};

\node [circle,draw=black, fill=black!25, inner sep=0pt,minimum size=5pt] (vertex2) at (2.5,-1.5) {};

\draw[thick,color=black, ->] (origin) --  (vertex1) node [pos=.5] {link $1$};

\draw[thick,color=black, ->] (vertex1) --  (destination) node [pos=.5] {link $2$};

\draw[thick,color=black, ->] (vertex1) --  (vertex2) node [pos=.5] {link $3$};

\draw[thick,color=black, ->] (origin) --  (vertex2) node [pos=.5, below left] {link $4$};

\draw[thick,color=black, ->] (vertex2) --  (destination) node [pos=.5, below right] {link $5$};

\end{tikzpicture} \label{Fig:CongestionGameExample}}
% \caption{Congestion game example with one origin/destination pair.}

% \end{figure}

% \begin{figure} [t]
% \centering

\subfigure[]{
\begin{tikzpicture}[auto, node distance=7em,>=latex']

\node [circle,draw=black, fill=green!20, inner sep=0pt,minimum size=10pt] (origin) at (0,0) {};

\node [circle,draw=black, fill=red!20, inner sep=0pt,minimum size=10pt] (destination) at (2.5,0) {};

\node [circle,draw=black, fill=black!25, inner sep=0pt,minimum size=2.5pt] (vertex1) at (1.25,0.75) {};

\node [circle,draw=black, fill=black!25, inner sep=0pt,minimum size=2.5pt] (vertex2) at (1.25,-0.75) {};

\node at (1.25,-1.1) {strategy 1};

\draw[thick,color=black, ->] (origin) --  (vertex1);

\draw[thick,color=black, ->] (vertex1) --  (destination);

\draw[thin,color=black!25,dashed] (vertex1) --  (vertex2);

\draw[thin,color=black!25,dashed] (origin) --  (vertex2);

\draw[thin,color=black!25,dashed] (vertex2) --  (destination);

\node [circle,draw=black, fill=green!20, inner sep=0pt,minimum size=10pt] (origin-a) at (2.9,0) {};

\node [circle,draw=black, fill=red!20, inner sep=0pt,minimum size=10pt] (destination-a) at (5.4,0) {};

\node [circle,draw=black, fill=black!25, inner sep=0pt,minimum size=2.5pt] (vertex1-a) at (4.15,0.75) {};

\node [circle,draw=black, fill=black!25, inner sep=0pt,minimum size=2.5pt] (vertex2-a) at (4.15,-0.75) {};

\node at (4.15,-1.1) {strategy 2};

\draw[thin,color=black!25,dashed] (origin-a) --  (vertex1-a);

\draw[thin,color=black!25,dashed] (vertex1-a) --  (destination-a);

\draw[thin,color=black!25,dashed] (vertex1-a) --  (vertex2-a);

\draw[thick,color=black, ->] (origin-a) --  (vertex2-a);

\draw[thick,color=black, ->] (vertex2-a) --  (destination-a);

\node [circle,draw=black, fill=green!20, inner sep=0pt,minimum size=10pt] (origin-b) at (5.8,0) {};

\node [circle,draw=black, fill=red!20, inner sep=0pt,minimum size=10pt] (destination-b) at (8.3,0) {};

\node [circle,draw=black, fill=black!25, inner sep=0pt,minimum size=2.5pt] (vertex1-b) at (7.05,0.75) {};

\node [circle,draw=black, fill=black!25, inner sep=0pt,minimum size=2.5pt] (vertex2-b) at (7.05,-0.75) {};

\node at (7.05,-1.1) {strategy 3};

\draw[thick,color=black, ->] (origin-b) --  (vertex1-b);

\draw[thin,color=black!25,dashed] (vertex1-b) --  (destination-b);

\draw[thick,color=black, ->] (vertex1-b) --  (vertex2-b);

\draw[thin,color=black!25,dashed] (origin-b) --  (vertex2-b);

\draw[thick,color=black, ->] (vertex2-b) --  (destination-b);
\end{tikzpicture} \label{Fig:StrategiesCongestionGameExample}}
\caption{Congestion game example with one origin/destination pair: (a) depicts the topology of the links and (b) illustrates the $3$ strategies representing all possible routes from the origin to the destination.}
\end{figure}
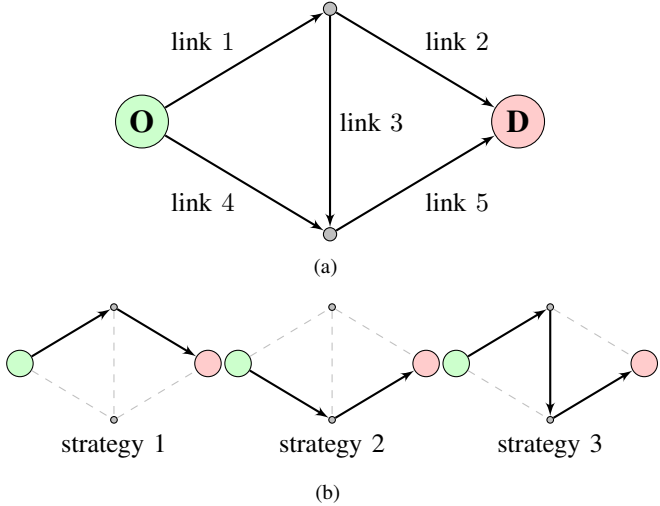

\begin{example}{\bf (Congestion game with 3 routing strategies)}
% \begin{example}{\bf (Congestion game with with 5 links and 3 strategies)}
\label{example:conjection5link}
Consider the congestion game associated with the graph in Fig.~\ref{Fig:CongestionGameExample}, in which there is a single origin-destination pair. We consider that there are three strategies represented by the routes in~Fig.~\ref{Fig:StrategiesCongestionGameExample}. The delay caused by utilization at each link~$i$ is quantified by a nondecreasing continuously differentiable function $\mathscr{D}_i:[0,1] \rightarrow \mathbb{R}_+$. The following population game quantifies the effect of the delays across the links of every strategy for a given strategy profile $z$ in $\mathbb{X}$:
\begin{equation}
    \mathcal{F}^{Ex.\ref{example:conjection5link}}(z)= - \begin{bmatrix} \mathscr{D}_1(z_1+z_3)+\mathscr{D}_2(z_1) \\ \mathscr{D}_4(z_2)+\mathscr{D}_5(z_2+z_3) \\ \mathscr{D}_1(z_1+z_3)+\mathscr{D}_3(z_3)+\mathscr{D}_5(z_2+z_3) \end{bmatrix}
\end{equation} The fact that the Jacobian $\mathcal{DF}^{Ex.\ref{example:conjection5link}}(z)$ is symmetric and negative definite for all $z$ in $\mathbb{X}$ implies that $\mathcal{F}^{Ex.\ref{example:conjection5link}}$ is (i) a strictly contractive game and (ii) it is also a strictly concave potential game.
\end{example}

\begin{figure} [t]
\centering
\begin{tikzpicture}[auto, node distance=7em,>=latex']

\tikzstyle{dra}=[draw, fill=orange!20, text width=5em, text centered, minimum height=2.5em]
\tikzstyle{ann} = [above, text width=5em, text centered]
\tikzstyle{epu} = [dra, text width=10em, fill=blue!10, minimum height=5em, rounded corners]

\node (epu) [epu]  {Electric Power Utility};
\path (epu.south)+(-3.5,-2.5) node (dra1) [dra] {\small Demand Response \\ Agent 1};
\path (epu.south)+(-1.0,-2.5) node (dra2)[dra] {\small Demand Response \\ Agent 2};
\path (epu.south)+(1.2,-2.7) node (dots)[ann] {$ \bullet\bullet\bullet$};
\path (epu.south)+(3.4,-2.5) node (dra3)[dra] {\small Demand Response \\ Agent M};   

\draw[ thick, blue, <-]([xshift=-0.1cm]dra1.north)-- node [above right = -.5cm and -.6cm, rotate=38] {\footnotesize cost signal} ([xshift=-0.1cm]epu.220);
\draw[ thick, red!80, ->]([xshift=0.1cm]dra1.north)-- node [below left = -.8cm and -1.0cm, rotate=38] {\footnotesize demand response} ([xshift=0.1cm]epu.220);
    
\draw[ thick, blue, <-]([xshift=-0.1cm]dra2.north)-- node [above right = -.5cm and -.5cm, rotate=40] {} ([xshift=-0.1cm]epu.270);
\draw[ thick, red!80, ->]([xshift=0.1cm]dra2.north)-- node [below left = -.8cm and -.9cm, rotate=40] {} ([xshift=0.1cm]epu.270);

\draw[ thick, blue, <-]([xshift=-0.1cm]dra3.north)-- node [above right = -.5cm and -.5cm, rotate=40] {} ([xshift=-0.1cm]epu.320);
\draw[ thick, red!80, ->]([xshift=0.1cm]dra3.north)-- node [below left = -.8cm and -.9cm, rotate=40] {} ([xshift=0.1cm]epu.320);
\end{tikzpicture}
\caption{Electricity demand response game example consisting of one electric power utility and multiple demand response agents.}
\label{Fig:DemandResponse}
\end{figure}
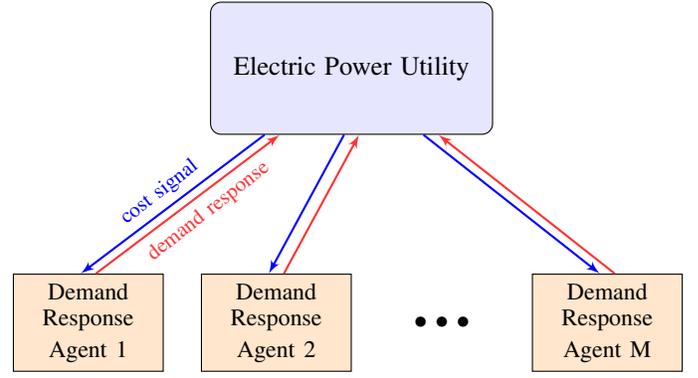

\begin{example} {\bf (Electricity demand response game with 3 strategies)} \label{ex:DemandResponse}
Consider an electricity demand response game where an Electric Power Utility (EPU) aims at reducing  power consumption of its consumers. The EPU employs Demand Response (DR) agents at consumer premises to control power consumption of participating consumers who, in return, receive incentives depending on the amount of reduction in power consumption. See Fig.~\ref{Fig:DemandResponse} for an illustration and also \cite{7422154} for more details on the demand response game.

In this example, we consider that $3$ demand response strategies are available to the DR agents, where each strategy~$i$ corresponds to the consumer's commitment to reduce his/her power consumption by $y_i~kW$ and that, for each strategy profile $z \in \mathbb X$, a strictly convex function $f: \mathbb X \to \mathbb R$ specifies the total incentive $f(z)$ the EPU provides to the consumers. The EPU primarily wants to achieve aggregate load reduction by $C~kW$ while minimizing the total incentive $f(z)$. This can be concisely formulated as optimization seeking $z^\ast \in \mathbb X$ that minimizes the cost function $\bar f: \mathbb X \to \mathbb R$ given by
\begin{align} \label{eq:CostFunction}
    \bar f(z) = f(z) + \nu \left( z^T y - C \right)
\end{align}
where the positive constant $\nu$ weighs the importance of meeting the aggregate load reduction $C$ with respect to the total incentive $f(z)$.

Each DR agent revises its strategy based on the following cost signal broadcast by the EPU:
\begin{align} \label{eq:DemandResponseGame}
    \mathcal F^{Ex.\ref{ex:DemandResponse}} \left( z \right) = \begin{bmatrix}
    -\nabla_{z_1} f(z) - \nu y_1 \\
    -\nabla_{z_2} f(z) - \nu y_2 \\
    -\nabla_{z_3} f(z) - \nu y_3 \\
    \end{bmatrix}
\end{align}
Note that $\mathcal F^{Ex.\ref{ex:DemandResponse}}$ is essentially (the negative of) the gradient of the cost function $\bar f$ and the DR agents would perform strategy revision along the gradient of $\bar f$. By construction $\mathcal F^{Ex.\ref{ex:DemandResponse}}$ is a strictly concave potential game and has a unique Nash equilibrium which is also the minimizer of \eqref{eq:CostFunction}.
\end{example}

\begin{figure} [t]
\centering
\includegraphics[trim={0.6in 1.5in 0 0},width=2.5in]{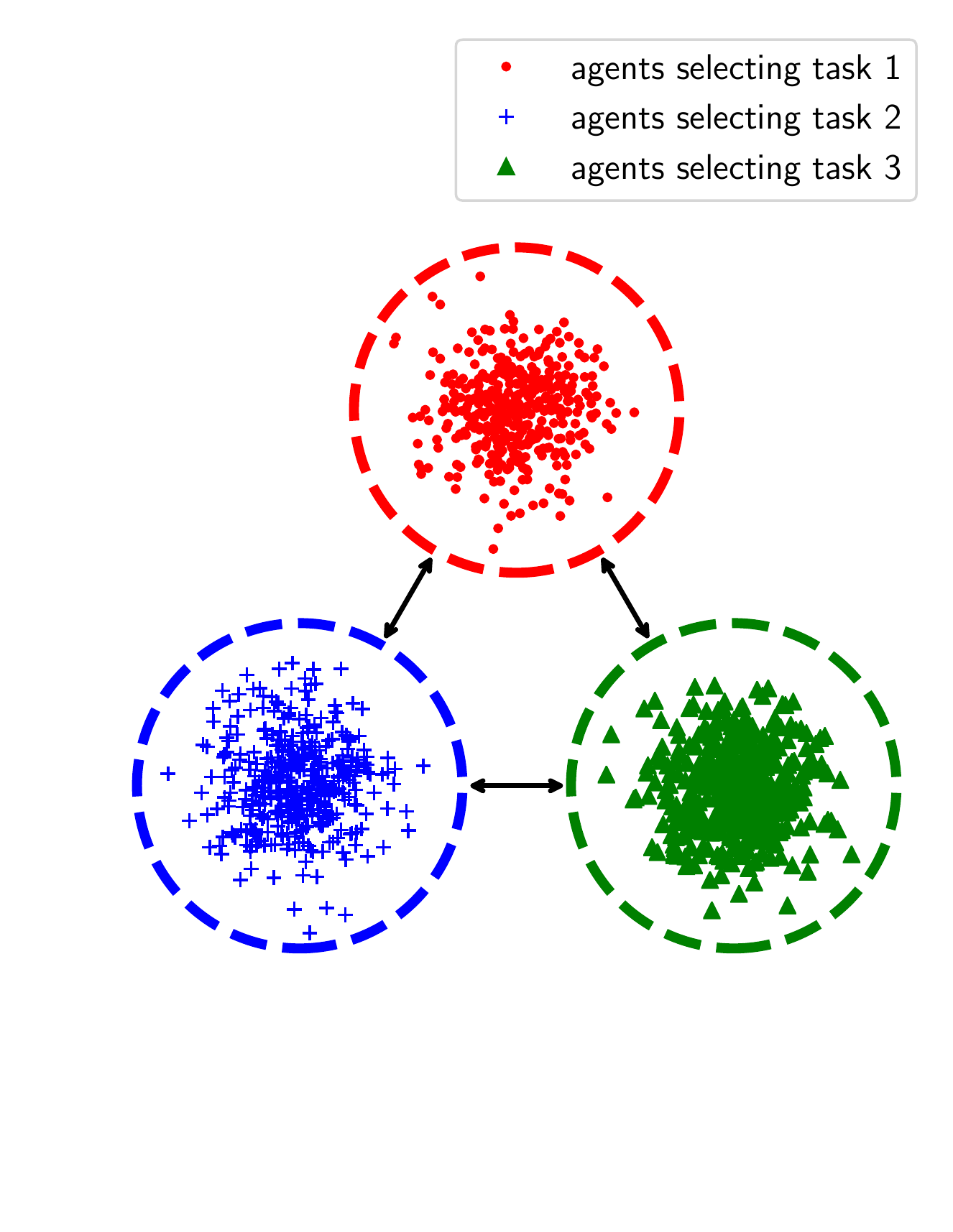}
\caption{Task allocation game example with $3$ tasks.}
\label{fig:TaskAllocationFigure}
\end{figure}

\begin{example}{\bf (Task allocation game with 3 tasks)} \label{ex:TaskAllocation}
Consider a task allocation game, as illustrated in Fig.~\ref{fig:TaskAllocationFigure}, in which 
% the population game associated with a multi-agent task allocation problem, which we refer to as the task allocation game (see Fig.~\ref{fig:TaskAllocationFigure} for an illustration). 
the agents select one among $3$ tasks and receive the reward specified by quasi-concave functions $\mathscr{R}_i: [0,1] \to \mathbb R_+$.
% for selecting task $i$. 
A payoff function for the game can be defined as:
\begin{align}
    \mathcal F^{Ex.\ref{ex:TaskAllocation}} \left( z \right) = \begin{bmatrix}
    \mathscr{R}_1 \left( z_1 \right) \\
    \mathscr{R}_2 \left( z_2 \right) \\
    \mathscr{R}_3 \left( z_3 \right)
    \end{bmatrix}
\end{align}
where $\mathscr{R}_i \left( z_i \right)$ represents, given the population state $z \in \mathbb X$, the reward assigned to the agents selecting task~$i$. Note that because $\mathscr{R}_i$ is quasi-concave, the population game defined by $\mathcal F^{Ex.\ref{ex:TaskAllocation}}$ may not be contractive.

% Note that different from Example~\ref{example:conjection5link},  $\mathcal F^{Ex.\ref{ex:TaskAllocation}}$ may not be a contractive game.
\end{example}

% \begin{example} [Consensus] \label{ex:coordination}
% Consider the payoff map given below for a coordination population game:
% \begin{align} \label{eq:coordination_game}
% \mathcal F(z) = \begin{pmatrix} 1 & 0 & 0 \\ 0 & 1 & 0 \\ 0 & 0 & 1
% \end{pmatrix} \begin{pmatrix} z_1 \\ z_2 \\ z_3 \end{pmatrix}
% \end{align}
% \end{example}

\subsection{Additional Examples of Application to Engineering and Optimization}

We now proceed to briefly surveying additional existing applications of population and evolutionary games to engineering problems, resource allocation and optimization. In these examples the PDM is either memoryless (population game) or has a fixed delay but no other internal dynamics. Still, they describe scenarios in which our results would be relevant, such as if the underlying game were to be modified to form a PDM that includes dynamics or if one would consider protocols more general than those of the replicator type. Recall that, as we discussed in~\S\ref{subsec:OutlineOfContrib}, by including additional dynamics in the PDM one can model important effects, such as inertia and anticipation in the agents' perception of the payoffs. Using numerical examples for the population games discussed in Examples~\ref{example:conjection5link}-\ref{ex:TaskAllocation}, we will examine such cases in \S\ref{sec:exmaples}.
%  We will discuss such cases using Examples~\ref{example:conjection5link}-\ref{ex:TaskAllocation} in \S\ref{sec:exmaples}.

% Say that our examples are for memoryless but with our approach we can also
% consider dynamic

\subsubsection{Traffic assignment}
\label{subsubsec:TrafficAssignmentExample} The seminal work reported in~\cite{Smith1984The-stability-o} investigates a traffic assignment problem in which the strategies are the possible routes each agent may traverse to travel from one location to another. The main result in~\cite{Smith1984The-stability-o} shows that certain simple protocols, which lead to what was latter called Smith dynamics~\cite[Example~5.6.1]{Sandholm2010Population-Game} and in our article is referred to as Smith EDM (see Example~\ref{ex:smith}), are guaranteed to steer the mean population state towards the Wardrop~\cite{Wardrop1952Some-theoretica} equilibrium of a given network cost-flow function. In the context of population games, such a cost is a congestion population game, of which Example~\ref{example:conjection5link} is an instance of, and the Wardrop equilibrium coincides with the befitting version of Nash's specified in Definition~\ref{def:NashForPDM}. 

\subsubsection{Distributed control strategies for resource allocation and optimization} The distributed control strategies discussed in~\cite{Quijano2017The-role-of-pop} hinge on simple protocols to regulate flow distribution in urban drainage systems, voltage split for lighting systems and economic power dispatch in microgrids. The underlying population games determining the payoff for these applications have a strictly concave potential. In this setting, the same Smith EDM mentioned in $\S$\ref{subsubsec:TrafficAssignmentExample}, or appropriate modifications thereof, guarantees that the mean population state will always converge to the maximum of the potential, which is also the unique Nash equilibrium of the population game. In fact, the potential can be used as a Lyapunov function to establish that its maximum is globally asymptotically stable. As is explained in~\cite{Sandholm2001Potential-games} and~\cite[\S3.1]{Sandholm2010Population-Game} the maximum of a strictly concave potential is globally asymptotically stable under any protocol satisfying positive correlation and Nash stationarity. More generally, the set of Nash equilibria of a potential game~-~as population games having a potential are called~\cite{Monderer1996Potential-games}~-~can be determined by the Karush-Kuhn-Tucker (KKT) conditions of local optimality of the potential. These results highlight the relevance of our formulation for distributed optimization and engineering applications in which a potential game is constructed and implemented by a coordinator, as is done in~\cite{Quijano2017The-role-of-pop} and references therein, to promote distributed resource allocation. Notably, as explained in~\cite[Example~3.1.2]{Sandholm2010Population-Game}, congestion population games in general admit a potential. In fact, the potential has been used since the early work in~\cite{Beckmann1956Studies-in-the-} to prove the existence of and also compute the Nash equilibria of congestion games. Population games were also used in~\cite{doi:10.1080/00207179.2016.1231422} to model and analyze dynamical resource allocation problems.

\subsubsection{Distributed control strategies for wireless networks} The approach in~\cite{Tembine2010Evolutionary-ga} uses the formalism of population games to design and analyze the performance of distributed algorithms that either trigger transmission or regulate the transmission power of a large number of agents in the scenarios of slotted ALOHA-based access network or W-CDMA power control, respectively. In these scenarios, the strategy sets, or equivalently the state of each agent, are $\{T,N\}$ (transmit or not) and $\{P_H,P_L\}$ (high or low power), respectively. The payoff for the ALOHA-based access network scenario captures the net reward of transmitting successfully subject to a cost quantifying the power spent. A population game captures the effect of collisions that may arise from simultaneous transmissions. In the W-CDMA power control scenario, the payoff captures the ability to broadcast to other nodes and the underlying population game also models the effect of interference that occurs when neighboring agents transmit concurrently. 

\subsubsection{Demand response management in smart grids}
The work of \cite{7422154, 7379949} applies the population game framework to demand response problems in smart grids. In particular, \cite{7422154}, which  motivates Example~\ref{ex:DemandResponse}, explains the formalism of a potential game to address an electricity demand reduction problem. Each strategy in the game describes the level of reduction in electricity consumption that a demand response agent would accept, and the underlying payoff mechanism incentivizes each agent proportional to the reduction in the agent's electricity consumption. The game attains a concave potential with which the authors establish the existence and stability of an evolutionarily stable strategy~\cite{Maynard-Smith1973The-logic-of-an} under the BNN and Smith EDMs, which are explained in Examples~\ref{ex:bnn} and \ref{ex:smith}, respectively, and the replicator dynamics~\cite{Taylor1978Evolutionarily-}.

On the other hand, \cite{7379949} describes two different population game formalisms: a participation game and consumption scheduling game. The participation game is designed to investigate, based on a given electricity pricing mechanism, how consumers adjust electricity consumption levels. Whereas the consumption scheduling game is devised to analyze consumers' time-of-use patterns on electricity consumption. The replicator protocol is used to define a decision-making model of the consumers and to analyze transient and asymptotic behavior of the consumers' electricity consumption patterns in each game.

% [Lee2015] Demand response problem in two different population game formulations: participation game and consumption scheduling game. Based on a pricing mechanism, participation game to predict how consumers adjust electricity consumption, and consumption scheduling game to understand  how consumers determine their time-of-use of electricity. Use replicator dynamics to analyze transient and asymptotic behavior of electricity consumption.

% [Srikantha2017] Demand response problem: discusses demand reduction of consumers by pricing mechanism Replicator dynamics, Smith dynamics, BNN dynamics

\subsubsection{User association in cellular networks}
The work in \cite{7244342} applies a potential game approach to an user association problem. In this game, a group of mobile stations (MSs) forms a population and a strategy adopted by each MS is defined as which base station (BS) it selects to be associated with. A payoff mechanism of the game captures how the associations between BSs and MSs affect flow-level efficiency and energy consumption at each BS. Leveraging the fact that the game admits a concave potential, the authors establish the existence of a unique Nash equilibrium, which is also socially optimal, and its stability under a distributed association control algorithm stemmed from the best response protocol~\cite[\S6.1]{Sandholm2010Population-Game}.

\subsubsection{Computation offloading over networks}
In \cite{8647522}, the population game framework is applied in a computation offloading problem over cellular networks. In this game, given a computation task, each mobile user (MU) makes a binary decision whether to perform the task locally using its own computation capability or remotely at a cloud server, where the remote server has more powerful computation capability than does the MU but the remote computation requires data transmission over the network. The population game formalism describes the trade-off between the time required to complete given tasks and data transmission cost incurred during the remote computation. Through numerical studies, the authors show that the Nash equilibrium of the game attains stability under the BNN and logit protocols, which are explained in Examples~\ref{ex:bnn} and \ref{ex:logit}, respectively.

\section{Population Games, Evolutionary Games and Mean Field Games: A Brief Comparison}
\label{sec:ComparionPopulationMeanF}

% Include other distributed coordination methods in final version

Our approach builds on the established framework of evolutionary dynamics and population games to take on the more general case in which the payoff is governed by a PDM. Hence, in order to precisely situate our work in light of existing related methods and approaches, we proceed to expounding the commonalities between our population-game-inspired framework and that of evolutionary games, and we will also discuss the most salient differences relative to mean field games. 

\subsection{Population games and evolutionary games: background and origins}
\label{sec:origins}

Starting with the foundations laid by Maynard-Smith and Price in~\cite{Maynard-Smith1973The-logic-of-an}, evolutionary game theory has been providing a tractable and mathematically rigorous framework to model and analyze the dynamics of natural selection according to which phenotypes prosper or flounder based on their fitness, as determined by their competitive advantages or drawbacks relative to each other and the environment. An evolutionary game that determines the fitness in terms of the traits of the intervening organisms, and the evolutionary mechanism by which reproduction is modulated based on the fitness form the basic tenets of the theory~\cite{Hofbauer2003Evolutionary-ga}. 

Subsequent work, mostly by mathematicians and economists~\cite{Sandholm2010Population-Game}, recognized that this theory could be adapted and extended in the following ways to model the evolution of strategy choices by a population of self-governing agents that are coupled only through a payoff structure that influences their decisions:
\begin{itemize}
\item The mathematical framework that is subjacent to evolutionary game theory can also be used to analyze the choice dynamics of populations of agents that are given a set of strategies to choose from and which they can repeatedly reconsider. According to this approach, in its simplest form, the mechanism governing the reproduction or depletion rate of a phenotype in terms of its fitness is analogous to the one determining the rate of increase or reduction of the prevalence of each strategy in one or more populations. More specifically, the phenotype selection mechanism and the concept of fitness for the former are analogous to the so-called strategy revision protocols and payoff used by the latter.
\item Payoff mechanisms that are relevant for economics and engineering, such as congestion games, can be adopted and handled as natural generalizations of those that determine the fitness in the evolutionary biology context. In the context of large populations, in which deterministic approximations are suitable, the payoff mechanism is denoted as a population game when it is a function of the so-called social state comprising the portions of every population choosing each strategy.
\item Equilibrium selection notions for evolutionary games such as ESS (evolutionarily stable strategy~\cite{Maynard-Smith1973The-logic-of-an}) that have important implications for stability analysis can be further generalized when considering contractive (or stable) population games~\cite{Hofbauer2007Stable-games,Hofbauer2009Stable-games-an}. As is thoroughly described in~\cite{Vincent1988The-evolution-o}, the concept of ESS can be further refined to account for strategies represented in uncountable subsets of coordinate spaces that are suitable to describe long-term adaptation of the strategy set. A model in which both the mean population state and the strategy are dynamic is described by the coupled equations in~\cite[(16)]{Vincent2000Evolution-and-C}~-~a coupled dynamic that is reminiscent of the mean closed loop model considered in \S\ref{sec:closedloop}.
\end{itemize}

After comparing~\cite[Chapters~3-5]{Weibull1995Evolutionary-ga} and~\cite[Chapter~7]{Hofbauer1998Evolutionary-ga} with~\cite[Chapter~4-8]{Sandholm2010Population-Game}, as well as references therein, one concludes that although work associated with evolutionary game theory is often focused on replicator dynamics~\cite{Taylor1978Evolutionarily-} or extensions of it, there is a significant overlap of the goals and the tools used for the analysis of stability relative to what is adopted in the study of population games and evolutionary dynamics.

\subsection{Comparing Population Games and Mean Field Games}
\label{sec:MFGComparison}

The description in~\cite[chapter~2]{Gomes2015Economic-models} suggests that the class of \underline{mean field games (MFG)} that is most similar to our framework is that of first order MFG. Starting with the pioneering work in~\cite{Huang2007Large-populatio,Huang2010The-NCE-mean-fi,Lasry2007Mean-field-game}, most work on first order MFG in an engineering context seeks to control a collective of structurally dissociated dynamical systems (agents) in the limit when the number of agents tends to infinity. Each agent in such a setting is steered by a causal control policy that depends on its own state and the so-called (deterministic) {\it mean state}, which is analogous to the (deterministic) mean population state adopted in our approach, or mean social state when there are two or more populations, in that both specify at every instant the portions of the population that attain every possible state or set of states. The control law is the same across an entire class (subgroup of identical agents), which is analogous to imposing the same protocol for all members of each population in the multi-population setting. In light of these similarities it is important that we list the following fundamental differences between the approach adopted here, and that of first order MFG. Examples of application for our framework and that of MFG to engineering problems are weaved in our discussion, which will also clarify the most salient differences between the two approaches.

\subsubsection{Main differences between MFG and our approach} In first order MFG, the set of {\it strategies} available to each agent class comprises control \underline{policies} that govern the control action as a causal function of the mean state and the state of each agent. The so-called Nash Certainty Equivalence principle determines the control policy for each agent class from the solution of a Hamilton-Jacobi-Bellman equation. We proceed to describing the most significant differences between our approach and that of MFG.

\paragraph{ States and strategies} In our approach, the state of each agent takes values in the strategy set. As a consequence, the strategy (or state) of each agent is revised repeatedly, which is in sharp contrast with the MFG approach in which the control policy of each agent class that is computed once off-line and used throughout the life of the MFG is a strategy. 

\paragraph{Cost and payoff mechanism} While a cost is used to compute the control policies for each agent class of a MFG, in our approach a payoff mechanism is itself the feedback component providing the payoff vector used by the agents to revise their strategies. In the particular case in which the payoff mechanism is a (memoryless) population game and the revision protocols are of the {\it best response} type, our approach is somewhat {\it similar} to a repeated game~\cite{Fudenberg1991Game-theory} in which the stage game would be the population game. In this analogy, we emphasize the qualifier ``similar" because unlike a standard repeated game, the techniques, models, concepts and objectives of our approach are tailored to deal with large populations.

% ADD HERE THAT IN OUR CONTEXT, THE PDM CAN ALSO ADD DYNAMICS TO THIS.

\paragraph{Equilibrium concepts} Characterization of relevant mean social state equilibria for our approach is inspired on Nash's and defined with respect to the payoff mechanism. When the payoff mechanism is a population game, the set of Nash equilibria, or perturbed versions of it, is functionally characterized in the usual way, and may be interpreted in the mass-action sense proposed in~\cite{Jr.1951Non-Cooperative} and further analyzed in \cite{Weibull1995The-mass-action}. In contrast, the so-called decentralized $\epsilon$-Nash concept~\cite{Huang2007Large-populatio}, which can be viewed as a relaxation of the person-to-person optimality notion~\cite{Mahajan2012Information-str} applied to the underlying cost, characterizes the relevant equilibria set for a MFG.

\paragraph{Convergence to equilibria} In contrast to MFG, the main goal of our framework is to establish concepts and methods to determine whether the populations will react in a way that the mean population state converges towards equilibria of interest. Convergence to an equilibria set, such as Nash's, gives credence to both the ability of the populations to self-organize using simple (bounded rationality) protocols in a distributed fashion and the value of the equilibria concept to predict long-term behavior.

%Moreover, depending on the revision protocol, even pairwise information exchange between randomly chosen agents may suffice. 

\paragraph{An MFG approach to power control for CDMA/CNO} 

The work in~\cite{Aziz2017A-mean-field-ga} puts forth a MFG approach to the design of power control for code division multiple access (CDMA) cellular network optimization (CNO). The state of each agent is the transmission power governed by the control policy, which is the pre-optimized strategy assigned to the class the agent belongs to. When the agents are mobile the state also includes their locations in the Cartesian plane, in which case the control policy not only governs the transmission power of each agent but also steers it. A careful comparison between this example with those of \S\ref{sec:EngineeringExamplesEvDyn} will certainly help clarify the differences between our approach and that of MFG.

\subsubsection{Advantages of our approach relative to MFG} Relative to the MFG approach, whose benefits are self-evident from the fact that the control policies are optimal for a given cost, our approach has the following advantages:
\paragraph {Simplicity of protocols} In our formulation, the protocols that the agents use to revise their strategies are simple to implement and amenable to analysis supported on conventional evolutionary principles. In addition, the agents may not be synchronized because independent Poisson clocks determine when each agent is allowed to revise its strategy. In contrast, the policies followed by the agents in the MFG approach are optimal, but may be intricate and difficult to implement in practice.
\paragraph{Information structure} As is discussed in~\cite{Sandholm2010Pairwise-compar}, the information structures required to implement the various types of revision protocols vary significantly. In contrast with MFG, in general, the agents do not need to access the entire mean social state.

\section{Relating The Mean Closed Loop Model With The Population State and Payoff Of Large Populations}
\label{sec:LargePopulationLimit}

The framework in~\cite{Sandholm2003Evolution-and-e} uses classical results~\cite{Kurtz1970Solutions-of-Or} to show that, as the size of the population tends to infinity, the solutions of the mean dynamics approximate with arbitrary accuracy, in the sense of \cite[Theorem~4.1]{Sandholm2003Evolution-and-e}, the realizations of the population state. Furthermore, such an analysis, which can also be found in~\cite[Chapter~10]{Sandholm2010Population-Game}, is further refined in~\cite[Lemma~1]{Benaim2003Deterministic-a}. 

\subsection{Finite population framework}
Likewise, in this section, we proceed to outlining the construction of a finite-population framework whose population state and payoff can be approximated with arbitrary accuracy uniformly over any given finite time interval by the solution of the mean closed loop model~(\ref{eq:ClosedLoop}) with probability approaching one as the population size tends to infinity. Our approach is to modify the framework in~\cite{Sandholm2003Evolution-and-e,Benaim2003Deterministic-a} and~\cite[Chapter~10]{Sandholm2010Population-Game} to comply with ours in which a PDM governs the deterministic payoff. 

Inspired by the construction in~\cite{Sandholm2003Evolution-and-e}, we consider that the population state of a single population with $N$ agents is represented by a right-continuous jump-process $X^N$ taking values in $\mathbb{X}^N$ defined as: 
$$ \mathbb{X}^N \overset{\text{def}}{=} \Big( \text{\footnotesize $\tfrac{1}{N}$} \mathbb{N}^n \Big ) \cap \mathbb{X}$$ where without loss of generality we assume that the population has a unit mass, i.e., $m=1$.
A Poisson process with positive rate $\varrho$ governs the revision times of each agent. The $N$ Poisson processes governing the revision times for the agents are independent. Hence, a Poisson process with rate $N\varrho$ governs the jump times of $X^N$. 

Given a pair $(\mathcal{G},\mathcal{H})$ satisfying the conditions of Definition~\ref{def:PDM}, the payoff vector at time $\tau$ is represented by $P^N(\tau)$, and is obtained in terms of $X^N$ and a pre-determined initial condition $Q^N(0)$ in $\mathbb{R}^n$ as the unique solution of:
\begin{equation}
\label{eq:FinitePopulationPDM}
\quad \begin{matrix}
\dot{Q}^N(t) = & \mathcal{G} \big (Q^N(t),X^N(t) \big ) \\
P^N(t)  = &\mathcal{H} \big ( Q^N(t),X^N(t) \big ) 
\end{matrix}, \quad t \geq 0 
\end{equation}

In contrast to~\cite{Sandholm2003Evolution-and-e,Benaim2003Deterministic-a}, here we assume that $\mathcal{F}(X^N(\tau))$ is replaced with $P^N(\tau)$ in the implementation\footnote{See also~\cite[Observation~10.1.2]{Sandholm2010Population-Game}.} of the protocol that models the strategy revision process. 

Following the approach in~\cite[\S4.1.2]{Sandholm2010Population-Game}, without loss of generality, we assume that the following holds:
$$\varrho \geq \sum_{j \neq i}^n \mathcal{T}_{ij}(r,z). \quad r \in \mathbb{R}^n, \ z \in \mathbb{X}, \ i \in \{1,\ldots,n\}$$

For each agent whose current strategy is $i$, the quantity $\tfrac{\mathcal{T}_{ij}(P^N(\tau),X^N(\tau))}{\varrho}$ determines the probability that it will switch to strategy $j$, conditioned on the event that it is allowed to switch at time~$\tau$.

More specifically,  $X^N$ is governed by the following probability transition law: 
%\begin{multline}
%\mathbb{P}(X^N(\tau) - z = \tilde{z} \ | \ \mathfrak{T}_{\tau^-,\tau,r,z} )  = \\ \begin{cases} z_i \mathcal{T}_{ij}\big ( r,z \big ) & \text{ if $\frac{\tilde{z}_j - \tilde{z}_i}{2} = 1, \ 1 \leq i,j \leq n$}  \\ 1 - \sum_{i=1}^n \sum_{j \neq i}  z_i \mathcal{T}_{ij}\big ( r,z ) & \text{ if $\tilde{z}=0$} \\ 0 & \text{otherwise}  \end{cases}
%\end{multline}
\begin{subequations}
\label{eq:XFinitePopTransProb}
\begin{align} \nonumber
& \mathbb{P}\Big (X_j^N(\tau)-z_j = \tfrac{1}{N} \ , \ X_i^N(\tau)-z_i = - \tfrac{1}{N}  \ \Big | \ \mathfrak{T} \Big )  \\  & \qquad \qquad \qquad \qquad \qquad = z_i\frac{ \mathcal{T}_{ij}\big ( r,z )}{\varrho} , \ i,j \in \{1,\ldots,n\} \\
& \mathbb{P}\Big (X^N(\tau)=z  \ \Big | \ \mathfrak{T} \Big ) = 1 - \sum_{i=1}^n \sum_{j \neq i}  z_i \frac{ \mathcal{T}_{ij}\big ( r,z )}{\varrho}
\end{align}
\end{subequations} where $\mathfrak{T}$ is the event that there are consecutive jumps at times $\tau^-$ and $\tau$, and ${\big( P^N(\tau^-),X^N(\tau^-) \big)=(r,z)}$ holds for pre-specified $z$ and $r$ in $\mathbb{X}^N$ and $\mathbb{R}^n$, respectively.

Given that, for any consecutive jump times $\tau^-$ and $\tau$,  $X^N(t)$ is constant for $t$ in $[\tau^-,\tau)$, from~(\ref{eq:FinitePopulationPDM}) we can deduce the following update rule:
\begin{equation}
\label{eq:QJumpUpdate}
Q^N(\tau) - Q^N(\tau^-)= \int_{\tau^-}^{\tau} \mathcal{G} \big (Q^N(\gamma),X^N(\tau^-) \big ) \, \mathrm d \gamma
\end{equation} where $Q^N:[\tau^-,\tau) \rightarrow \mathbb{R}^n$ is the solution of~(\ref{eq:FinitePopulationPDM}) starting with the initial condition $Q^N(\tau^-)$.

From~(\ref{eq:XFinitePopTransProb}), we also infer that $X^N$ is unchanged if we replace $P^N$ and $Q^N$ with the right-continuous jump processes $\check{P}^N$ and $\check{Q}^N$ that get updated at every jump time $\tau$ according to $\check{P}^N(\tau)=P^N(\tau)$ and $\check{Q}^N(\tau)=Q^N(\tau)$. Consequently, the update rules specified in~(\ref{eq:XFinitePopTransProb}) and~(\ref{eq:QJumpUpdate}) imply that the pair $(X^N,\check{Q}^N)$ is a right-continuous Markov jump process that satisfies the conditions of the framework in~\cite{Kurtz1970Solutions-of-Or}. In addition, for every $(z,s)$ in $\mathbb{X}^N \times \mathbb{R}^n$ and $t \geq 0$, the following holds:
\begin{subequations}
\label{eq:PrelimAsNTendsToInfty}
\begin{multline}
 \lim_{\delta \downarrow 0} \mathbf{E} \Bigg [  \frac{1}{\delta} \Big( X^N(t+\delta) -z \Big) \ \Bigg |  \\  \ \Big( X^N(t),\check{Q}^N(t) \Big)=(z,s)  \Bigg ] \\  = \mathcal{V}^{\mathcal{H}}(s,z)
\end{multline}
\begin{multline}
 \lim_{\delta \downarrow 0} \mathbf{E} \Bigg [  \frac{1}{\delta} \Big( \check{Q}^N(t+\delta) -s \Big) \ \Bigg |  \\  \ \Big( X^N(t),\check{Q}^N(t) \Big)=(z,s)  \Bigg ] \\  = \mathcal{G}(s,z)
\end{multline}
\end{subequations}

\subsection{Approximation in the limit of large populations}
Hence, using~\cite[Theorem~2.11]{Kurtz1970Solutions-of-Or}, we conclude from~(\ref{eq:PrelimAsNTendsToInfty}) that given any positive $\bar{t}$ and $\delta$, and initial condition $(x(0),q(0))$ in $\mathbb{X} \times \mathbb{R}^n$, the following holds for every sequence of initial states $\{(x^N(0),q^N(0))\}_{N=1}^{\infty}$ that converges to $(x(0),q(0))$:
\begin{equation}
\label{eq:MainApproxLimitNTendsInfty}
\lim_{N \rightarrow \infty} \mathbb{P} \Bigg ( \sup_{0 \leq t \leq \bar{t}} \Big \| \big (X^N(t),\check{Q}^N(t) \big ) - \big (x(t),q(t) \big ) \Big \| > \delta \Bigg ) = 0
\end{equation}
where $(x(t),q(t))$ is the solution of~(\ref{eq:ClosedLoop}) and we assume that for each $N$ the process $(X^N,\check{Q}^N)$ is initialized with $X^N(0)=x^N(0)$ and  {${\check{Q}^N(0)=q^N(0)}$}. Since $\mathcal{H}$ is Lipschitz continuous, from~(\ref{eq:MainApproxLimitNTendsInfty}), we also conclude that the following holds:
\begin{equation}
\label{eq:MainApproxLimitNTendsInftyWithP}
\lim_{N \rightarrow \infty} \mathbb{P} \Bigg ( \sup_{0 \leq t \leq \bar{t}} \Big \| \big (X^N(t),\check{P}^N(t) \big ) - \big (x(t),p(t) \big ) \Big \| > \delta \Bigg ) = 0
\end{equation}

As we discussed in \S\ref{sec:MeanDynamicsPopulationGames} for the particular case in which the PDM is a population game (memoryless), Theorems~\cite[12.B.3]{Sandholm2010Population-Game} and~\cite[12.B.5]{Sandholm2010Population-Game}, which are derived from work in~\cite{Benaim1998Recursive-algor} and~\cite{Benaim1999Stochastic-appr}, ascertain under unrestrictive conditions that as the population grows, the stationary distribution of the population state tends to concentrate around the smallest globally asymptotically stable set. Although it is beyond the scope of this article, we believe that immediate extensions of Theorems~\cite[12.B.3]{Sandholm2010Population-Game} and~\cite[12.B.5]{Sandholm2010Population-Game} to our context would show that the stationary distribution of $\big (X^N(t),\check{Q}^N(t) \big )$, when it exists, will tend to concentrate around a globally asymptotic stable set of the mean closed loop model as $N$ tends to infinity.

\section{EDM $\delta$-passivity, PDM $\delta$-antipassivity, and Main Supporting Lemma}
\label{sec:PassivityConcepts}

%Intro to the Section 
We start this section by defining key EDM and PDM properties, which we will use later on to state the conditions under which our convergence results hold. Subsequently, in \S\ref{subsec:Main Lemma} we state a key supporting lemma that we will use to establish the convergence results presented in \S\ref{sec:NashStationaryEDMGlobalConvergence}~through~\S\ref{sec:PBR}.

\subsection{EDM $\delta$-passivity and Informative Storage Functions}

Given an EDM with input $w$ and state $x$, the following inequality is central for the definition of $\delta$-passivity \cite{Fox2013Population-Game}:
\begin{multline}
\label{eq:DeltaPassive}
\mathcal{S} \big ( x(t), w(t) \big ) - \mathcal{S} \big ( x(t_0), w(t_0) \big ) \leq \\  \int_{t_0}^t \left[ \dot{x}^T(\tau) \dot{w}(\tau) - \eta \dot{x}^T(\tau) \dot{x}(\tau) \right] \, \mathrm d \tau, \\ t \geq t_0, \ t,t_0 \in \mathbb{T}, \ x(t_0) \in \mathbb{X}, \ w \in \mathfrak{P}  
\end{multline} where $\eta$ and $\mathcal{S}:\mathbb{X} \times \mathbb{R}^n \rightarrow \mathbb{R}_+$ are a nonnegative real constant and a map, respectively.

\begin{definition} {\bf (EDM $\delta$-passivity)} \label{def:passivity}  Given an EDM, we adopt the following $\delta$-passivity concepts:
\begin{itemize}
	\item The EDM is said to be \underline{$\delta$-passive} if there is a continuously differentiable $\mathcal{S}$ for which~(\ref{eq:DeltaPassive}) is satisfied with $\eta=0$.
	\item If the EDM is \underline{$\delta$-passive}, let $\eta^*$ be the supremum of all $\eta$ for which there is a continuously differentiable $\mathcal{S}$ satisfying~(\ref{eq:DeltaPassive}). If $\eta^*$ is positive then the EDM is qualified as \underline{$\delta$-passive} with \underline{surplus $\eta^*$}. 
\end{itemize} For either case, the map $\mathcal{S}$ is referred to as a \underline{$\delta$-storage function}. We refer to the EDM generally as \underline{strictly} \underline{output} $\delta$-passive when it is $\delta$-passive with some positive surplus~$\eta^*$.
\end{definition} 

Notice that the larger $\eta^*$ the more stringent the requirement for strict output $\delta$-passivity. When it is positive, we view such $\eta^*$ as a measure of $\delta$-passivity ``surplus".

As is discussed in \cite{Fox2013Population-Game}, an EDM is $\delta$-passive when the following augmented system with input $w^\delta$ and output $x^\delta$ is passive according to its standard definition \cite{Sepulchre1997Constructive-no}:
\begin{subequations}
\label{eq:InterpEDMdPAssive}
\begin{align}
\dot{w}(t) & = w^\delta(t), \quad \quad \quad \quad w(0) \in \mathbb{R}^n, \ w^\delta \in \mathfrak{P}^\delta \\
\dot{x}(t) &= \mathcal{V} \big ( x(t), w(t) \big ), \ \ x(0) \in \mathbb{X} \\
x^\delta(t) &= \mathcal{V} \big ( x(t), w(t) \big ) 
\end{align} 
\end{subequations} where $\mathfrak{P}^\delta \overset{\mathrm{def}}{=} \{ \dot{w} | w \in \mathfrak{P} \}$. Notably, $(x(t),w(t))$ is the state of the augmented system and $\mathcal{S}$ is a storage function for it.  

As we will see later in \S\ref{subsec:Main Lemma}, if a $\delta$-storage function is informative, according to the following definition, then we can use it to establish convergence results for the mean closed loop model.

\begin{definition}{\bf (Informative $\mathcal{S}$) } Let ${\mathcal{S}:\mathbb{X} \times \mathbb{R}^n \rightarrow \mathbb{R}_+}$ be a $\delta$-storage function for a given EDM specified by $\mathcal{V}$. We say that $\mathcal{S}$ is \underline{informative} if it satisfies the following two conditions:
  % The following was in the original manuscript, but I dont think it is needed in the current 
  % context because $\dot{x}$ is already guaranteed to be bounded for a closed loop in which
  % the pdm is bounded. 
  % \begin{equation}
  %   \limsup_{l \rightarrow \infty} \mathcal{S}(z^{(l)},z^{(l)}) < \infty \implies  \limsup_{l \rightarrow \infty} \|\mathcal{V} (z^{(l)},z^{(l)})\| < \infty 
  % \end{equation}
  \begin{subequations}
    \begin{equation}\label{eq:informative1}
      \mathcal{V}(z^*,r^*) = 0 \implies \mathcal{S}(z^*,r^*) = 0, 
    \end{equation}
    \begin{equation}\label{eq:informative2}
      \nabla_z^T\mathcal{S}(z^*,r^*)\mathcal{V}(z^*,r^*) = 0 \implies \mathcal{V}(z^*,r^*) = 0
    \end{equation} 
  \end{subequations} for every $z^*$ and $r^*$ in $\mathbb{X}$ and $ \mathbb{R}^n$, respectively. 
\end{definition}

The implication in~(\ref{eq:informative1}) suggests that, for a constant deterministic payoff, every equilibrium point of the EDM minimizes $\mathcal{S}$. In addition, from ${\frac{\mathrm d}{\mathrm dt}\mathcal{S}(x(t),r) = \nabla_x^T\mathcal{S}(x(t),r)\mathcal{V}(x(t),r)}$, we could conclude from (\ref{eq:informative2}) that, for a constant deterministic payoff $r$, $\mathcal{S}(x(t),r)$ is constant only if $x(t)$ remains at an equilibrium point of the~EDM. 

\subsection{PDM $\delta$-antipassivity and weak $\delta$-antipassivity}

The following conditions will be used in the definition of $\delta$-antipassivity for a given PDM with input $u$, state $q$, and output~$p$:
\begin{subequations}
\label{eq:DAPassivePDM}
\begin{equation}
\label{eq:1stDAPAssivePDM}
\mathcal{L}(z,s) = 0 \Leftrightarrow \mathcal H(s, z) = \bar{\mathcal{F}}(z), \quad z \in \mathbb{X}, \ s \in \mathbb{R}^n
\end{equation}
\begin{multline}
\label{eq:2ndDAPAssivePDM}
\mathcal{L}(u(0),q(0)) - \mathcal{L}(u(t),q(t)) \geq \\ \int_0^t \left[ \dot{p}^T(\tau) \dot{u}(\tau) - \nu \dot{u}^T(\tau) \dot{u}(\tau) \right] \, \mathrm d\tau, \\ t \geq 0, \ q(0) \in \mathbb{R}^n, u \in \mathfrak{X} 
\end{multline}
\end{subequations} where $\nu$ is a nonnegative constant, $\bar{\mathcal{F}}$ is the stationary population game of the PDM, and ${\mathcal{L}:\mathbb{X} \times \mathbb{R}^n \rightarrow \mathbb{R}_{+}}$ is a map. 

\begin{definition} {\bf (PDM $\delta$-antipassivity)}
\label{def:EquilibriumStability} Given a PDM, we consider the following cases:
\begin{itemize}
	\item The PDM is said to be \underline{$\delta$-antipassive} if there is a continuously differentiable $\mathcal{L}$ for which (\ref{eq:DAPassivePDM}) is satisfied for $\nu=0$.
	\item The PDM is \underline{$\delta$-antipassive with deficit $\nu^* > 0$} if there is a continuously differentiable $\mathcal{L}$ for which (\ref{eq:DAPassivePDM}) is satisfied for every~${\nu > \nu^*}$.
\end{itemize} A map $\mathcal{L}$ that satisfies either case is referred to as a {$\delta$-antistorage} function.
\end{definition}

 Notice that there is an antisymmetry between~(\ref{eq:DeltaPassive}) and~(\ref{eq:2ndDAPAssivePDM}) that is obtained from changing signs of certain terms and swapping the output with the input. This correspondence could be further strengthened by viewing $\mathcal{L}$ as a $\delta$-antistorage function that would be the antisymmetric equivalent of $\mathcal{S}$.  An analogy similar to~(\ref{eq:InterpEDMdPAssive}) is done in~\cite{Fox2013Population-Game} to compare~$\delta$-antipassivity with standard notions of passivity. 

Given a PDM with input $u$, state $q$, and output $p$, the following inequality is central to characterizing a weaker notion of $\delta$-antipassivity:
\begin{multline}
\label{eq:WDAntipPDM}
% \mathcal{A} \big ( q(0), \|\dot{u}\| \big ) \geq \\ \sup_{t \in \mathbb{T} } \int_0^t \left[ \dot{p}^T(\tau) \dot{u}(\tau) - \nu \dot{u}^T(\tau)\dot{u}(\tau) \right] \, \mathrm d \tau,  \\  \ q(0) \in \mathbb{R}^n,  \ u \in \mathfrak{X}
 \mathcal{A} \big ( q(0), \|\dot{u}\| \big ) \geq \\ \int_0^t \left[ \dot{p}^T(\tau) \dot{u}(\tau) - \nu \dot{u}^T(\tau)\dot{u}(\tau) \right] \, \mathrm d \tau,  \\  t \geq 0, \ q(0) \in \mathbb{R}^n,  \ u \in \mathfrak{X}
\end{multline} where $\nu$ and $\mathcal{A}:\mathbb{R}^{n} \times \mathbb{R}_{+} \rightarrow \mathbb{R}_+$ are a nonnegative real constant and a map, respectively.

\begin{definition}
\label{def:WDAntipPDM}
{\bf (PDM Weak $\delta$-antipassivity)}  Given a PDM, we consider the following cases:
\begin{itemize}
	\item The PDM is said to be \underline{weak $\delta$-antipassive} if there is $\mathcal{A}$ for which (\ref{eq:WDAntipPDM}) is satisfied for $\nu=0$.
	\item The PDM is \underline{weak $\delta$-antipassive with deficit $\nu^* > 0$} if there is $\mathcal{A}$ for which (\ref{eq:WDAntipPDM}) is satisfied for every~${\nu > \nu^*}$.
\end{itemize} 
\end{definition}

Unlike $\delta$-antipassivity, which requires the existence of a continuously differentiable $\delta$-antistorage function $\mathcal{L}$, weak $\delta$-antipassivity only requires the existence of a map $\mathcal{A}$ satisfying~(\ref{eq:WDAntipPDM}). Although well-known results~\cite{Moylan1974Implications-of,Hill1976The-stability-o,Willems1972Dissipative-dyn} indicate that a so-called ``available" storage function can be constructed when (\ref{eq:WDAntipPDM}) is satisfied, there are no guarantees that it will be continuously differentiable or satisfy~(\ref{eq:1stDAPAssivePDM}). The following remark outlines an argument to establish that $\delta$-antipassivity indeed implies weak $\delta$-antipassivity.

\begin{remark}{\bf ($\delta$-antipassivity implies weak $\delta$-antipassivity)} Given a PDM, the following holds:
\begin{itemize}
\item If the PDM is $\delta$-antipassive then it is weak $\delta$-antipassive and we can select $\mathcal{A}$ as:
\begin{align}
\mathcal{A}(q(0),\cdot)= &\max_{z \in \mathbb{X}} \mathcal{L}(z, q(0))
\end{align} where $\mathcal{L}$ is the $\delta$-antistorage function of the PDM.
% \begin{align}
% \mathcal{A}(q(0),\cdot)= &\max_{z \in \mathbb{X},s \in \mathbb{R}^n} \mathcal{L}(z,s), \quad q(0) \in \mathbb{R}^n \\
% & s.t.  \  \|s\| \leq q_{\text{max}}(q(0)) \nonumber
% \end{align} where $q_{\text{max}}(q(0))$ is an upper-bound for $\|s\|$ subject to~(\ref{eq:PDM}) and the initial condition $q(0)$. 
\item This choice for $\mathcal{A}$ also allows us to conclude that if the PDM is $\delta$-antipassive with deficit $\nu^*$ then it is also weak $\delta$-antipassive with deficit $\nu^*$.
\end{itemize} 
\end{remark}

As we will see later in \S\ref{sec:NashStationaryEDMGlobalConvergence} and~\S\ref{sec:PBR} where we study convergence of the mean population state to $\mathbb{NE}(\bar{\mathcal{F}})$, as well as perturbed versions of it, $\delta$-antipassivity and weak $\delta$-antipassivity are prerequisites for global asymptotic stability and global attractiveness, respectively. This confirms, as one should expect, that because the $\delta$-antipassivity condition is stricter than its weak version, it leads to stronger stability guarantees.

\subsubsection{Contractivity and $\delta$-antipassivity for memoryless PDM}
\label{sec:MemorylessPDM}

 We proceed to stating in the following proposition, and subsequently proving that a memoryless PDM is $\delta$-antipassive if and only if its stationary population game is contractive in the sense of Definition~\ref{def:StableGame}. Alternatively, one may assert that $\delta$-antipassivity gets to be contractivity in the particular case when the PDM is a population game.

\begin{proposition} 
\label{prop:Memoryless} Let $\mathcal{F}: \mathbb{X} \rightarrow \mathbb{R}^n$ be a given continuously differentiable map satisfying the following inequality:
\begin{equation}
\label{eq:PassivityForMemoryless}
\tilde{z}^T \mathcal{DF}(z) \tilde{z} \leq \nu^* \tilde{z}^T \tilde{z}, \quad z \in \mathbb{X}, \ \tilde{z} \in \mathbb{TX}
\end{equation} Here, $\nu^*$ is the least nonnegative real constant for which the inequality holds. The following holds for the memoryless PDM corresponding to the population game $\mathcal{F}$.
\begin{itemize}
	\item If $\nu^*$ is zero then the PDM is $\delta$-antipassive.
	\item If $\nu^*$ is positive then the PDM is $\delta$-antipassive with deficit~$\nu^*$.
\end{itemize}
\end{proposition} 

\begin{proof} Recall that in this case $\bar{\mathcal{F}} = \mathcal{F}$ and choose $\mathcal{L}(z,s) = 0$. Clearly, since $\mathcal H(s,z) = \mathcal F(z)$, Assumption \ref{assumption:stationary_game} and \eqref{eq:1stDAPAssivePDM} hold. Also the fact that $\dot{p}^T(t) \dot{u}(t) = \dot{u}^T(t) \mathcal{DF}(u(t)) \dot{u}(t)$ and (\ref{eq:PassivityForMemoryless}) imply that~(\ref{eq:2ndDAPAssivePDM}) holds.
\end{proof}

% ADD Example for $n=3$ in the Appendix. Here, use the fact that condition iii) of the proposition
% allows for the difference $\alphaF-F^T\alpha^T$ to have constant columns or rows.

%The procedure in the example also works for when $F$ is a tri-diagonal matrix and $\alpha$ is
%positive definite diagonal.

% In fact, we can setup a QP to search the vector $v$ such that the modified payoff
% $F' = F + \mathbf{1}v^T$ satisfies the conditions and minimized $\nu$.

\subsection{Main Supporting Lemma and Outline of Main Convergence Results}
\label{subsec:Main Lemma}

We proceed with stating a lemma that ascertains conditions on the EDM and PDM under which key stability properties for the mean closed loop are guaranteed. The lemma will be used as an important building block of the stability results in \S\ref{sec:NashStationaryEDMGlobalConvergence} and~\S\ref{sec:PBR} where we analyze well-known EDM classes. 
\noindent \rule{\columnwidth}{1pt}
\begin{lemma} 
\label{lem:MainLemma}
{\bf (Main Supporting Lemma)} \\  Let a PDM and an EDM be given. We consider the following two cases.
\begin{itemize}
	\item {\bf (Case I)} The PDM is weak $\delta$-antipassive ($\nu=0$) and the EDM is {$\delta$-passive} ($\eta=0$) with respect to an informative $\delta$-storage function~$\mathcal{S}$.
	\item {\bf (Case II)} The PDM is weak $\delta$-antipassive with deficit $\nu^*>0$ and the EDM is $\delta$-passive with surplus $\eta^* > \nu^*$ with respect to an informative $\delta$-storage function~$\mathcal{S}$.
\end{itemize}
If either Case I or Case II is true then the following holds:
\begin{equation} 
\label{eq:MainLemma}
\lim_{t \rightarrow \infty} \mathcal{S}(x(t),p(t)) = 0, \quad (x(0),q(0)) \in \mathbb{X} \times \mathbb{R}^n
\end{equation}
where the trajectory $(x,p)$ is determined from the unique solution of the initial value problem for~(\ref{eq:ClosedLoop}).
\end{lemma} 
\noindent \rule{\columnwidth}{1pt}

 A proof of Lemma \ref{lem:MainLemma} is given in Appendix \ref{appendix_a}. This lemma will enable us to use $\mathcal{S}$ to proceed in a manner that is analogous to how Lyapunov functions were used in~\cite{Hofbauer2009Stable-games-an} to establish convergence of the mean population state to meaningful equilibria of contractive population games, for various classes of EDM. 

 \begin{remark}{\bf (Trade-off in Case~II)} Notice that Case~II of the lemma allows for a PDM that is weak $\delta$-antipassive with deficit $\nu^*>0$ at the expense of requiring that the EDM is $\delta$-passive with surplus $\eta^* > \nu^*$. That is to say that a less stringent $\delta$-antipassivity requirement on the PDM can be counterbalanced by an appropriately stricter $\delta$-passivity condition on the EDM.
 \end{remark}

\subsection{Comparison with related notions of passivity}
\label{sec:PassivityComparison}
In dynamical system theory, there are other notions of passivity, namely, incremental passivity \cite{Pavlov2008Incremental-pas}, differential passivity \cite{Forni2014A-differential-,Forni2013On-differential} and equilibrium-independent passivity \cite{Hines2011Equilibrium-ind,Arcak2016Networks-of-Dis}. For a certain class of dynamical system models, e.g., linear system models, incremental passivity and differential passivity are equivalent to $\delta$-passivity; however, the equivalence would not hold for nonlinear system models such as EDM \eqref{PopulationDynamic} considered in this paper. On the other hand, as we briefly explain below, the qualification for equilibrium-independent passivity is basically different from that of $\delta$-passivity even for linear system models.

% In dynamical system theory, there are other notions of passivity, namely, incremental passivity \cite{Pavlov2008Incremental-pas}, differential passivity \cite{Forni2014A-differential-,Forni2013On-differential} and equilibrium-independent passivity \cite{Arcak2016Networks-of-Dis}. In this section, we compare these passivity notions with $\delta$-passivity in terms of their roles in establishing the stability of dynamical systems, and discuss the motivation for our selection of $\delta$-passivity.

% We start by describing the relation among the $4$ different notions of passivity. For a certain class of dynamical system models, e.g., linear system models, incremental passivity and differential passivity are equivalent to $\delta$-passivity; however, the equivalence would not hold for nonlinear system models such as EDM \eqref{PopulationDynamic} considered in this paper. On the other hand, as we briefly explain below, the qualification for equilibrium-independent passivity is basically different from that of $\delta$-passivity even for linear system models.

% In this section,
In what follows,
we compare these passivity notions with $\delta$-passivity in terms of their roles in establishing the stability of dynamical systems, and discuss the motivation for our selection of $\delta$-passivity.
Essentially, incremental passivity and differential passivity are used to analyze the so-called \textit{incremental stability} defined as the pairwise contraction of the state trajectories of a dynamical system -- a concept relevant in synchronization and consensus problems. On the other hand, inextricably related to Lyapunov stability, $\delta$-passivity is used to ascertain the convergence towards certain equilibria, which in our context would typically be Nash or perturbed equilibria of the stationary population game of a PDM. Hence, in comparison with incremental passivity and differential passivity, $\delta$-passivity is more adequate for the stability analysis provided in this article.

In some cases, equilibrium-independent passivity could alternatively be applied to an EDM \eqref{PopulationDynamic} by considering, as we do with $\delta$-passivity,  that $w(t)$ and $x(t)$ are the input and output, respectively. Notably, the work in~\cite{8619157} used equilibrium-independent passivity to investigate the convergence of the mixed strategy for a finite population and the mean population state for an infinite population towards a perturbed equilibrium in a game defined by a higher-order learning rule and negative-monotone payoff.

According to the definition in~\cite{Hines2011Equilibrium-ind}, equilibrium-independent passivity would hold in our context if for each $w^\ast$ in $\mathbb R^n$, there would be a continuously differentiable function $\mathcal S_{w^\ast}: \mathbb X \times \mathbb R^n \to \mathbb R_+$ for which the following condition is satisfied:
  \begin{multline} \label{eq:eip}
    \mathcal S_{w^\ast} (x(t), w(t)) - \mathcal S_{w^\ast} (x(t_0), w(t_0)) \leq \\ \int_{t_0}^t \left( x(\tau) - x^\ast \right)^T \left( w(\tau) - w^\ast \right) \,\mathrm d\tau, \\ t \geq t_0, \ t,t_0 \in \mathbb{T}, \ x(t_0) \in \mathbb{X}, \ w \in \mathfrak{P}  
  \end{multline}
  where $x^\ast$ would be a mean population state satisfying ${\mathcal V(x^\ast, w^\ast) = 0}$. Furthermore, the definition of equilibrium-independent passivity in~\cite{Hines2011Equilibrium-ind} would require that, for every $w^\ast$ in $\mathbb R^n$, there is a continuously differentiable $\mathcal S_{w^\ast}$ satisfying \eqref{eq:eip} for which $x^\ast$ is the unique state satisfying $\mathcal V(x^\ast, w^\ast) = 0$. However, as can be inferred by analyzing Examples~\ref{ex:bnn} and~\ref{ex:smith} in \S\ref{sec:IEPT} and \S\ref{sec:IPC} describing the BNN and Smith EDM, respectively, this uniqueness requirement is not satisfied by important EDM classes, which further justifies our adoption of $\delta$-passivity to develop the approach reported in this article. These examples will also not satisfy the assumptions of the modified version of equilibrium-independent passivity discussed in~\cite{Arcak2016Networks-of-Dis}. In particular, in these examples, if the pair $(x^\ast,w^\ast)$ satisfies ${\mathcal V(x^\ast, w^\ast) = 0}$ then so will the pair $(x^\ast,\chi w^\ast)$ for any positive scalar $\chi$, which violates a key uniqueness requirement of~\cite[\S3.1]{Arcak2016Networks-of-Dis}.
  
  It is also interesting to note that when a memoryless PDM is specified by a contractive game then not only it will be $\delta$-antipassive, as guaranteed by Proposition~\ref{prop:Memoryless}, but it will also satisfy the antisymmetric of the condition required for incremental passivity.

\subsection{$\delta$-dissipativity and Weighted-Contractivity}

The concepts of $\delta$-passivity, and $\delta$-antipassivity, are generalized in~\cite{Arcak2020Dissipativity-T} to notions based on  $\delta$-dissipativity inequalities that afford additional flexibility by allowing weighting matrices to be introduced in~(\ref{eq:DeltaPassive}) and~(\ref{eq:2ndDAPAssivePDM}). An analysis of stability of $\mathbb{NE}(\mathcal F)$ is then developed in~\cite{Arcak2020Dissipativity-T} that is based on this article and extends some of its results\footnote{We refer the reader to~\cite{Arcak2020Dissipativity-T} for a thorough comparison}. When the PDM is memoryless and specified by a population game, the conditions for stability in~\cite{Arcak2020Dissipativity-T} remain valid for {\it weighted-contractive} games, of which contractive ones are a particular case.  This is especially beneficial in multipopulation settings in which the contractive properties of the game vary from one population to another, such as in~\cite{MehHor19,LazCooPed17}. 

\section{Nash Stationarity and Convergence to $\mathbb{NE}(\bar{\mathcal{F}})$}
\label{sec:NashStationaryEDMGlobalConvergence}

We start by defining a key property called {\it Nash Stationarity}, which will allow us to associate equilibria of an EDM with the set of best responses to a deterministic payoff vector.

\begin{definition} {\bf (Nash Stationarity)} A given EDM specified by $\mathcal{V}:\mathbb{X} \times \mathbb{R}^n \rightarrow \mathbb{R}^n$ satisfies Nash stationarity if the following equivalence holds:
\begin{equation}
\label{eq:NashStationarity}
\mathcal{V}(z,r) = 0 \Leftrightarrow z \in \arg \max_{\bar{z} \in \mathbb{X}} \bar{z}^T r, \quad z \in \mathbb{X}, \ r \in \mathbb{R}^n
\end{equation}
\end{definition}

In the framework of \cite{Hofbauer2009Stable-games-an}, Nash stationarity of the mean dynamics is crucial to establishing that the mean population state converges to the set of Nash equilibria of an underlying contractive population game. Not surprisingly, it will also be essential in our analysis, as is evidenced by the following lemma. Fortunately, as we discuss in Remarks~\ref{rem:EPTNashStationary} and~\ref{rem:IPCNashStation}, the EDM classes considered throughout this section satisfy Nash stationarity.

\noindent \rule{\columnwidth}{1pt}
\begin{lemma} \label{lem:NashStationMainLemma} Consider that a mean closed loop model is formed by a given PDM and an EDM that is Nash stationary, $\delta$-passive, and has an informative $\delta$-storage function $\mathcal{S}$. If the PDM is weak $\delta$-antipassive then $\mathbb{NE}(\bar{\mathcal{F}})$ is globally attractive. If the PDM is $\delta$-antipassive then $\mathbb{NE}(\bar{\mathcal{F}})$ is globally asymptotically stable.
\end{lemma}
\noindent \rule{\columnwidth}{1pt}

A proof of Lemma \ref{lem:NashStationMainLemma} is given in Appendix \ref{appendix_a}.
Notice that Lemma~\ref{lem:NashStationMainLemma} is restricted to the case in which the EDM is $\delta$-passive and the PDM is either weak $\delta$-antipassive or $\delta$-antipassive. This stands in contrast with Case~II of Lemma~\ref{lem:MainLemma}, which allows a PDM to be weak $\delta$-antipassive with positive deficit $\nu^*$ at the expense of restricting the EDM to be $\delta$-passive with surplus $\eta^* > \nu^*$. This level of generality is not viable for the EPT and IPC EDM defined below because, as we show in \cite[Corollary~IV.3]{Park2018Passivity-and-e}, they are not strictly output $\delta$-passive.

% 	IT IS IMPORTANT TO HAVE A PROPOSITION IN THE APPENDIX PROVING THAT EPT AND IPC
%   ARE NOT STRICTLY OUTPUT $\DELTA$-PASSIVE. 

\subsection{Integrable Excess Payoff Target (EPT) EDM}
\label{sec:IEPT}
%Definitions and intro

We start by defining Excess Payoff Target (EPT) EDM by specifying the properties that the associated protocol must satisfy. 

\begin{definition} 
\label{def:EPTEDM}
{\bf (Excess Payoff Target (EPT) EDM)} 
\\A given protocol $\mathcal{T}$ yields an EPT EDM if it can be written as:
\begin{equation}
\begin{array}{rl}
\mathcal{T}_{ij}(r,z) & = \mathcal{T}^{\text{EPT}}_j(\hat{r}) \\ 
\hat{r}_i & \overset{\mathrm{def}}{=} r_i - \frac{1}{m} \sum_{i=1}^n r_i z_i
\end{array},   \quad r \in \mathbb{R}^n, \ z \in \mathbb{X}
\end{equation} where $\mathcal{T}^{\text{EPT}}:\mathbb{R}^n_* \rightarrow \mathbb{R}^n_{+}$ is a Lipschitz continuous map, $\hat{r}$ is the vector of \underline{excess payoff} relative to the population average and $\mathbb{R}^n_{*} = \mathbb{R}^n - \mathrm{int}(\mathbb{R}^n_{-})$ is the set of possible excess payoff vectors. In addition, $\mathcal{T}^{\text{EPT}}$ must satisfy the following {\it acuteness} condition:
\begin{equation}
\label{eq:EPTAcuteness}
\hat{r}^T \mathcal{T}^{\text{EPT}}(\hat{r}) > 0, \quad \hat{r} \in \mathrm{int}(\mathbb{R}^n_{*})
\end{equation}
\end{definition}

 A comprehensive analysis and motivation for this protocol class is provided in~\cite{Sandholm2005Excess-payoff-d}.

\begin{remark}
\label{rem:EPTNashStationary} 
{\bf (EPT EDM is Nash Stationary)} \\
In order to simplify the structure of our article, and given that there is no significant disadvantage in doing so, we adopt the convention that every EPT EDM satisfies acuteness~(\ref{eq:EPTAcuteness}). 
A trivial adaptation to our formulation of~\cite[Theorem~5.5.2 and~Exercise~5.5.7]{Sandholm2010Population-Game} for excess payoff target dynamics shows that our acuteness assumption guarantees that every EPT EDM is Nash stationary, which is crucial for the results in this section.
%In comprehensive discussions of the excess payoff target dynamics, such as in~\cite{Sandholm2010Population-Game}, acuteness may not be the part of the definition, and is, instead, viewed as a property that identifies a subclass. 
\end{remark}

\begin{definition}{\bf (Integrable EPT EDM)} A given EPT protocol $\mathcal{T}^{\text{EPT}}:\mathbb{R}^n_* \rightarrow \mathbb{R}^n_{+}$ is integrable if there is a continuously differentiable function $\mathcal{I}^{\text{EPT}}:\mathbb{R}^n \rightarrow \mathbb{R}$ such that the following holds:
\begin{equation} \label{eq:revision_potential}
\mathcal{T}^{\text{EPT}}(\hat{r}) = \nabla \mathcal{I}^{\text{EPT}}(\hat{r}), \quad \hat{r} \in \mathbb{R}^n_{*}
\end{equation} We refer to $\mathcal{I}^{\text{EPT}}$ as the revision potential of $\mathcal{T}^{\text{EPT}}$.
\end{definition}

We can now proceed with establishing that any given EPT EDM with integrable protocol is $\delta$-passive and has an informative $\delta$-storage function. This key step will allow us to use Lemma~\ref{lem:NashStationMainLemma} to ascertain in Theorem~\ref{thm:TheoremIntegrableEPT} that the mean population state of a mean closed loop model converges globally to $\mathbb{NE}(\bar{\mathcal{F}})$.

\begin{proposition} 
\label{prop:IEPT-PAssive}
If a given EPT EDM is integrable with revision potential $\mathcal{I}^{\text{EPT}}$ then it is $\delta$-passive and there is a constant $\gamma$ for which $\mathcal{S}^{\text{EPT}}$ given below is an informative $\delta$-storage function:
\begin{equation}
\mathcal{S}^{\text{EPT}}(z,r) = m \, \mathcal{I}^{\text{EPT}}(\hat{r}) - \gamma, \quad z \in \mathbb{X}, \ r \in \mathbb{R}^n
\end{equation} In addition, the following equivalence holds:
\begin{equation}
\label{eq:InverseIEPTStorage}
\mathcal{S}^{\text{EPT}}(z,r) = 0 \Leftrightarrow z \in \arg \max_{\bar{z} \in \mathbb{X}} \bar{z}^Tr, \quad z \in \mathbb{X},\ r \in \mathbb{R}^n
\end{equation}
\end{proposition}

A proof of Proposition \ref{prop:IEPT-PAssive} is given in Appendix \ref{appendix_b}. Notice that Proposition~\ref{prop:IEPT-PAssive} extends~\cite[Theorem~4.4]{Fox2013Population-Game} in two ways.
\begin{itemize} 
	\item Unlike Proposition~\ref{prop:IEPT-PAssive}, each $\delta$-storage function in~\cite[Theorem~4.4]{Fox2013Population-Game} is constructed for a given upper-bound on $\|r\|$, which must be known a priori (see also~\cite[Eq.~(62)]{Fox2013Population-Game}). For each integrable EPT EDM, our construction provides a unique $\delta$-storage function without any such assumptions.
	\item More importantly, Proposition~\ref{prop:IEPT-PAssive} guarantees that a constant $\gamma$ exists for which~(\ref{eq:InverseIEPTStorage}) holds. This is a key fact in proving Theorem~\ref{thm:TheoremIntegrableEPT}, and consequently Corollary~\ref{cor:StableGameIEPT}, at the level of generality we have here.
\end{itemize}

The following specifies an important subclass of integrable EPT EDM for which a $\delta$-storage function can be readily constructed.

\begin{definition}{\bf (Separable EPT EDM)} A given EPT EDM is separable if its protocol $\mathcal{T}^{\text{EPT}}$ can be written as:
\begin{equation} \mathcal{T}^{\text{EPT}}_j(\hat{r}) = \mathcal{T}^{\text{SEPT}}_j(\hat{r}_j), \quad \hat{r} \in \mathbb{R}^n_{*}, \ j \in \{ 1,\ldots,n \}
\end{equation} where $\mathcal{T}_j^{\text{SEPT}}: \mathbb{R} \rightarrow \mathbb{R}_{+}$ is a Lipschitz continuous map for each $j$ in $\{ 1,\ldots,n \}$.
\end{definition}

The following corollary follows from Proposition~\ref{prop:IEPT-PAssive} and the fact that a separable EPT protocol is also integrable.

\begin{corollary} If a given EPT EDM is separable with protocol $\mathcal{T}^{\text{SEPT}}$ then it is $\delta$-passive and $\mathcal{S}^{\text{SEPT}}$ given below is an informative $\delta$-storage function:
\begin{equation}
\mathcal{S}^{\text{SEPT}}(z,r) = \sum_{i=1}^n \int_0^{\hat{r}_i} \mathcal{T}_i^{\text{SEPT}} (\tau) \, \mathrm d\tau, \quad z \in \mathbb{X}, \ r \in \mathbb{R}^n
\end{equation}
\end{corollary}

The following is a widely used example of EPT protocol, which was originally introduced in~\cite{Brown1950Solutions-of-ga} to prove key properties of two-player zero-sum games.

\begin{example} \label{ex:bnn}{\bf (BNN EDM)} The Brown-von Neumann-Nash (BNN) EDM, as named in~\cite{Hofbauer2000From-Nash-and-B}, is specified by the following separable EPT protocol\footnote{See also \cite[Examples~5.5.1 and Exercise~5.5.1]{Sandholm2010Population-Game}.}:
\begin{equation}
  \mathcal{T}_j^{\text{BNN}}(\hat{r}) \overset{\mathrm{def}}{=} [\hat{r}_j]_+, \quad \hat{r} \in \mathbb{R}^n_{*}
\end{equation} The following is the associated $\delta$-storage function, which is informative:
\begin{equation}
  \mathcal{S}^{\text{BNN}}(z,r) = \frac{1}{2} \sum_{i=1}^n [\hat{r}_i]_+^2, \quad z \in \mathbb{X}, \ r \in \mathbb{R}^n
\end{equation}
\end{example}
% Here state the theorem guaranteing passivity and the LaSalle property of the
% storage function.

When agents follow the BNN protocol they are likely to switch to strategies whose payoff is higher than the average payoff for the population. The higher the excess payoff for a given strategy, relative to the average, the more likely an agent will select it.

We can now state our main theorem establishing an important convergence theorem for integrable EPT EDM.

\noindent \rule{\columnwidth}{1pt}
\begin{theorem}
\label{thm:TheoremIntegrableEPT} Consider a mean closed loop model formed by an integrable EPT EDM and a PDM. If the PDM is weak $\delta$-antipassive then $\mathbb{NE}(\bar{\mathcal{F}})$  is globally attractive. If the PDM is $\delta$-antipassive then $\mathbb{NE}(\bar{\mathcal{F}})$ is globally asymptotically stable.
\end{theorem}
\noindent \rule{\columnwidth}{1pt}
\begin{proof}
The proof follows immediately from Proposition~\ref{prop:IEPT-PAssive}, Lemma~\ref{lem:NashStationMainLemma}, and Remark~\ref{rem:EPTNashStationary}.
\end{proof}

The following corollary is an immediate consequence of Proposition~\ref{prop:Memoryless} and Theorem~\ref{thm:TheoremIntegrableEPT}.

\begin{corollary}
\label{cor:StableGameIEPT} Consider that a memoryless PDM is specified by a given continuously differentiable contractive population game $\mathcal{F}$. Let a mean closed loop model be formed by an integrable EPT EDM and the memoryless PDM. The set $\mathbb{NE}(\mathcal{F})$ is globally asymptotically stable.
\end{corollary}

\begin{remark}{\bf (Corollary~\ref{cor:StableGameIEPT} extends~\cite[Theorem~5.1]{Hofbauer2009Stable-games-an})} 
Notice that, in what regards the trajectory of the mean population state, when the PDM is memoryless the mean closed loop model is equivalent to the formulation in~\cite[Eq.~(E) of \S4.3]{Hofbauer2009Stable-games-an}. Consequently, Corollary~\ref{cor:StableGameIEPT} extends~\cite[Theorem~5.1]{Hofbauer2009Stable-games-an} because the latter guarantees global asymptotic stability only when the protocol is separable or $\mathcal{F}$ has a unique Nash equilibrium. Our more general result is possible, in part, because Proposition~\ref{prop:IEPT-PAssive} asserts the important fact that a $\delta$-storage function satisfying~(\ref{eq:InverseIEPTStorage}) exists for any integrable EPT protocol.
\end{remark}

% \subsection{Impartial Pairwise Comparison (IPC) EDM: Convergence Properties}
\subsection{Impartial Pairwise Comparison (IPC) EDM}
\label{sec:IPC}

% Here state the theorem guaranteing passivity and the LaSalle property of the
% storage function.

We now proceed with defining and characterizing global convergence properties for impartial pairwise comparison~(IPC)~EDM, whose designation was proposed in~\cite{Sandholm2010Pairwise-compar} for a more general context.

\begin{definition}{\bf (Impartial Pairwise Comparison (IPC) EDM)} A given protocol $\mathcal{T}$ yields an impartial pairwise comparison~(IPC)~EDM if it can be written as:
\begin{equation}
\mathcal{T}_{ij}(r,z) = \mathcal{T}^{\text{IPC}}_{j}(r_j-r_i), \quad r \in \mathbb{R}^n, \ z \in \mathbb{X}
\end{equation} where $\mathcal{T}^{\text{IPC}}:\mathbb{R}^n \rightarrow \mathbb{R}^n_{+}$ is a Lipschitz continuous map, which also satisfies the following sign preservation condition:
\begin{equation}
\label{eq:acutenessIPC}
\begin{cases}
\mathcal{T}^{\text{IPC}}_{j}(r_j-r_i) > 0, & \text{if $\ r_j > r_i$} \\
\mathcal{T}^{\text{IPC}}_{j}(r_j-r_i) = 0, & \text{if $\ r_j \leq r_i$}
\end{cases}, \quad r \in \mathbb{R}^n
\end{equation}
\end{definition}

According to an IPC protocol, an agent is likely to switch to strategies offering a higher payoff. Typically, the likelihood of switching to a given strategy increases with its payoff. 

\begin{remark} 
\label{rem:IPCNashStation}
{\bf (IPC EDM is Nash Stationary)}
 It follows from a trivial modification of~\cite[Theorem~5.6.2]{Sandholm2010Population-Game} that the IPC EDM also satisfies Nash stationarity.
\end{remark}

The following proposition shows that the method to construct the Lyapunov function in \cite[Theorem~7.1]{Hofbauer2009Stable-games-an} can be adapted to obtain an informative $\delta$-storage function for an IPC EDM.

\begin{proposition} 
\label{prop:IPCDPassive}
Let a Lipschitz continuous map ${\mathcal{T}^{\text{IPC}}:\mathbb{R}^n \rightarrow \mathbb{R}^n_{+}}$ specify the protocol of an IPC EDM. The IPC EDM is $\delta$-passive and the map $\mathcal{S}^{\text{IPC}}: \mathbb X \times \mathbb{R}^n \rightarrow \mathbb{R}_+$ defined below is an informative $\delta$-storage function:
\begin{multline}
\label{eq:StorageFuncIPC}
\mathcal{S}^{\text{IPC}}(z,r) \overset{\mathrm{def}}{=} \sum_{i=1}^n \sum_{j=1}^n z_i \int_0^{r_j-r_i} \mathcal{T}^{\text{IPC}}_j (\tau) \ \mathrm d \tau, \\ \quad z \in \mathbb{X}, \ r \in \mathbb{R}^n 
\end{multline}
\end{proposition}

A simple example is the so-called Smith EDM defined below, which was originally proposed in~\cite{Smith1984The-stability-o} to investigate a traffic assignment problem.

\begin{example}\label{ex:smith}{\bf (Smith EDM)} The Smith EDM is specified by the following IPC protocol\footnote{See also~\cite[Example~5.6.1 and Exercise~5.6.1]{Sandholm2010Population-Game}.}:
\begin{equation} \label{eq:SmithProtocol}
\mathcal{T}_j^{\text{Smith}}(r_j-r_i) \overset{\mathrm{def}}{=} [r_j-r_i]_+, \quad r \in \mathbb{R}^n
\end{equation} The following is the associated $\delta$-storage function, which is informative:
\begin{equation}
\mathcal{S}^{\text{Smith}}(z,r) \overset{\mathrm{def}}{=} \sum_{i=1}^n \sum_{j=1}^n z_i [r_j-r_i]_+^2
\end{equation}
\end{example}

We can finally make use of Lemma~\ref{lem:NashStationMainLemma} to state our main stability theorem for IPC EDM.

\noindent \rule{\columnwidth}{1pt}
\begin{theorem} 
\label{thm:IPCConvergence}
Consider a mean closed loop model formed by an IPC EDM and a PDM. If the PDM is weak $\delta$-antipassive then $\mathbb{NE}(\bar{\mathcal{F}})$ is globally attractive. If the PDM is $\delta$-antipassive then $\mathbb{NE}(\bar{\mathcal{F}})$ is globally asymptotically stable.
\end{theorem}
\noindent \rule{\columnwidth}{1pt}

\begin{proof}
The proof follows immediately from Proposition~\ref{prop:IPCDPassive}, Lemma~\ref{lem:NashStationMainLemma}, and Remark~\ref{rem:IPCNashStation}.
\end{proof}

\section{Perturbed Best Response (PBR) EDM: Convergence to $\mathbb{PE}(\bar{\mathcal{F}},\mathcal{Q})$}
\label{sec:PBR}

In this section, we consider a class of protocols according to which the mean population state is steered towards its best response to a perturbed payoff. The perturbation models imperfections in the perception of the payoff by the agents. In a prescriptive scenario or engineering application, the perturbation could account for sensor noise or limitations of the network disseminating payoff and population state information. 

\begin{definition}{\bf (Payoff Perturbation)} Let $\mathcal{Q}: \mathrm{int}(\mathbb{X}) \rightarrow \mathbb{R}$ be a given map. We deem $\mathcal{Q}$ an admissible payoff perturbation if it is twice continuously differentiable and satisfies the following conditions:
\begin{align}
\tilde{z}^T \nabla^2 \mathcal{Q}(z) \tilde{z} & >0, \quad z \in \mathbb{X}, \ \tilde{z} \in \mathbb{TX} - \{0\} \\
\lim_{z_{\text{min}} \rightarrow 0} \| \nabla \mathcal{Q}(z) \| & = \infty, \quad \text{where}\quad z_{\text{min}} \overset{\mathrm{def}}{=} \min_{1 \leq i \leq n} z_i \label{eq:boundary}
\end{align} The subset of $\mathbb{X}$ for which $z_{\text{min}}$ is zero is often referred to as the boundary of $\mathbb{X}$ and is denoted as $\mathrm{bd}(\mathbb{X})$. For every admissible perturbation $\mathcal{Q}$, below we also define the associated perturbed maximizer:
\begin{equation}
\mathcal{M}^{\mathcal{Q}}(r) \overset{\mathrm{def}}{=} \argmax_{z \in \mathrm{int}(\mathbb{X})} \left( z^T r - \mathcal{Q}(z) \right)
\end{equation} The so-called \underline{choice function} can be computed as ${\mathcal{C}^{\mathcal{Q}}(r) = \frac{1}{m} \mathcal{M}^{\mathcal{Q}}(r)}$ for $r$ in $\mathbb{R}^n$.
\end{definition}

\S6.2 of~\cite{Sandholm2010Population-Game} includes a comprehensive discussion of the properties of $\mathcal{M}^{\mathcal{Q}}$. Notably, it explains why, for each $r$ in $\mathbb{R}^n$, $\mathcal{M}^{\mathcal{Q}}(r)$ takes a single value in $\mathrm{int}(\mathbb{X})$, in contrast with best response maps that are in general set valued, and it also discusses analogs of most of the notions we will define below.  The seminal article~\cite{Hofbauer2007Evolution-in-ga} offers a well-documented justification for the model adopted here and~\cite{Hofbauer2002On-the-global-c} provides an important discrete choice theorem relating $\mathcal{Q}$ with the distribution of the additive noise that characterizes a probabilistic formulation of~$\mathcal{C}^{\mathcal{Q}}$.

\paragraph*{Convention} We should note that in most published work, the domain of the payoff perturbation is a {\it normalized } version of $\mathbb{X}$ denoted by $\Delta=\{\frac{1}{m} z | z \in \mathbb{X} \}$. However, we find that, in our context, stating the definitions and results consistently in terms of $\mathbb{X}$ simplifies our notation. 

The following is the general form of the PBR EDM, which was originally proposed in an slightly different but equivalent form in~\cite{Fudenberg1993Learning-mixed-}.

\begin{definition}{\bf (Perturbed Best Response (PBR) EDM)} Consider that an admissible payoff perturbation $\mathcal{Q}$ is given. A given protocol $\mathcal{T}$ yields a perturbed best response (PBR) EDM associated with $\mathcal{Q}$ if it can be written as:
\begin{equation}
\mathcal{T}_{ij}(r,z) = \mathcal{C}_j^{\mathcal{Q}}(r), \quad r \in \mathbb{R}^n, \ z \in \mathbb{X}
\end{equation} The following expression is determined according to~(\ref{eq:MeanDynamic}):
\begin{equation}
\label{eq:EDMPBR}
\tilde{\mathcal{V}}^{\mathcal{Q}}(z,r) = \mathcal{M}^{\mathcal{Q}}(r)-z, \quad z \in \mathbb{X}, \ r \in \mathbb{R}^n
\end{equation}
\end{definition}

\subsection{Perturbed Stationarity for PBR EDM}

Unlike the EPT EDM and IPC EDM described in \S\ref{sec:NashStationaryEDMGlobalConvergence}, the PBR EDM does not satisfy Nash stationarity. However, as pointed out in~\cite[Observation~6.2.7]{Sandholm2010Population-Game}, it does satisfy {\it Perturbed Stationarity} as defined below.

\begin{definition} {\bf (Perturbed Equilibrium Set and Virtual Payoff)} \\
Given a PDM with continuously differentiable $\bar{\mathcal{F}}$ and an admissible payoff perturbation $\mathcal{Q}$, the associated perturbed equilibrium is defined as:
\begin{equation}
\mathbb{PE}(\bar{\mathcal{F}},\mathcal{Q}) \overset{\mathrm{def}}{=} \big \{z \in \mathbb{X} \ \big| \ z=\mathcal{M}^{\mathcal{Q}} \big ( \bar{\mathcal{F}}(z) \big ) \big \}
\end{equation} 
\end{definition}

An immediate adaptation of~\cite[Theorem~6.2.8]{Sandholm2010Population-Game} to our context leads to the conclusion that the perturbed equilibrium can also be specified as the Nash equilibrium of the so-called \underline{virtual payoff} ${\tilde{\mathcal{F}}^{\mathcal{Q}}: \mathrm{int}(\mathbb{X}) \rightarrow \mathbb{R}^n}$ defined as:

\begin{equation}
\tilde{\mathcal{F}}^{\mathcal{Q}}(z) \overset{\mathrm{def}}{=}  \bar{\mathcal{F}}(z) - \nabla \mathcal{Q}(z), \quad z \in \mathrm{int}(\mathbb{X})
\end{equation} In summary, we can state the following:
\begin{equation}
\mathbb{NE}(\tilde{\mathcal{F}}^{\mathcal{Q}}) =  \mathbb{PE}(\bar{\mathcal{F}},\mathcal{Q})
\end{equation} 

\begin{remark} \label{remark:perturbed_stationarity} {\bf (Perturbed Stationarity)} It is an immediate consequence of (\ref{eq:EDMPBR}) that the PBR EDM satisfies the following equivalence also referred to as \underline{perturbed stationarity}:
\begin{equation}
\tilde{\mathcal{V}}^{\mathcal{Q}}(z,r) = 0 \Leftrightarrow z = \mathcal{M}^{\mathcal{Q}}(r) , \quad z \in \mathbb{X}, \ r \in \mathbb{R}^n
\end{equation}
\end{remark}

In addition, it follows from~\cite[Theorem~3.1]{Hofbauer2007Evolution-in-ga} that if $\bar{\mathcal{F}}$ is continuously differentiable and contractive then $\mathbb{PE}(\bar{\mathcal{F}},\mathcal{Q})$ is a singleton.

\subsection{$\delta$-passivity Characterization for PBR EDM}
The following proposition establishes $\delta$-passivity properties for a given PBR EDM, which will allow us to use Lemma~\ref{lem:MainLemma} to assert in Theorem~\ref{thm:PBRConvergence} sufficient conditions under which $\mathbb{PE}(\bar{\mathcal{F}},\mathcal{Q})$ is globally attractive and globally asymptotically stable.

\begin{proposition} 
\label{prop:DeltaPassivityPBR}
Consider that an admissible payoff perturbation $\mathcal{Q}$ is given, for which we define the following candidate $\delta$-storage function:
\begin{multline} \label{eq:s_pbr}
\mathcal{S}^{\text{PBR}}(z,r) \overset{\mathrm{def}}{=}  \max_{\bar{z} \in \mathrm{int}(\mathbb{X})} \big (\bar{z}^Tr-\mathcal{Q}(\bar{z}) \big ) - \big (z^Tr - \mathcal{Q}(z) \big ), \\  z \in \mathbb{X} , \ r \in \mathbb{R}^n
 \end{multline}
Let $\eta^*$ be the infimum of all nonnegative constants $\eta$ for which the following holds:
\begin{equation}
\label{eq:PositivityQetaPBREDM}
\tilde{z}^T \nabla^2 \mathcal{Q}(z) \tilde{z} \geq \eta\tilde{z}^T \tilde{z} , \quad z \in \mathrm{int}(\mathbb{X}), \ \tilde{z} \in \mathbb{TX}
\end{equation} One of the two cases holds:
\begin{itemize}
\item {\bf (Case~I)} If $\eta^* \geq 0$ then the PBR EDM is $\delta$-passive and $\mathcal{S}^{\text{PBR}}$ is an informative $\delta$-storage function. 
\item {\bf (Case~II)} If $\eta^* > 0 $ then the PBR EDM is $\delta$-passive with surplus $\eta^*$ and $\mathcal{S}^{\text{PBR}}$ is an informative $\delta$-storage function. 
\end{itemize}
\end{proposition}

A proof of Proposition \ref{prop:DeltaPassivityPBR}  is given in Appendix~\ref{appendix_b}. Notice that~\cite[Theorem~3.1]{Hofbauer2007Evolution-in-ga} uses a Lyapunov function that is analogous to $\mathcal{S}^{\textit{PBR}}$ to establish convergence results in the framework of contractive population games. 

\begin{example} \label{ex:logit}{\bf (Logit EDM)} The Logit EDM is specified by the following protocol\footnote{See also \cite{Hofbauer2007Evolution-in-ga, Hofbauer2002On-the-global-c}.}:
\begin{equation}
  \mathcal{C}_i^{\mathcal{Q}}(r) \overset{\mathrm{def}}{=} \frac{e^{\eta^{-1}r_i}}{\sum_{j=1}^n e^{\eta^{-1}r_j}}, \quad r \in \mathbb R^n
\end{equation}
where $\eta$ is a positive constant. The following is the associated $\delta$-storage function, which is informative:
\begin{multline} \label{eq:s_logit}
\mathcal{S}^{\text{Logit}}(z,r) = \max_{\bar{z} \in \mathrm{int}(\mathbb{X})} \big (\bar{z}^Tr-\mathcal{Q}(\bar{z}) \big ) - \big (z^Tr - \mathcal{Q}(z) \big ), \\  z \in \mathbb{X} , \ r \in \mathbb{R}^n
\end{multline}
where the payoff perturbation $\mathcal Q$ is given by $\mathcal Q(z) = \eta\sum_{i=1}^n z_i \ln z_i$. The parameter $\eta$ is referred to as the noise level \cite{Hofbauer2007Evolution-in-ga}.
\end{example}

\subsection{Global Convergence to Perturbed Equilibria for PBR EDM}

At this point, we have defined all the key concepts and presented the preliminary results required to state our main theorem establishing the conditions under which we can guarantee global convergence to $\mathbb{PE}(\bar{\mathcal{F}},\mathcal{Q})$.

\noindent \rule{\columnwidth}{1pt}
\begin{theorem} 
\label{thm:PBRConvergence}
Consider a mean closed loop model formed by a $\delta$-passive PBR EDM characterized by an admissible payoff perturbation $\mathcal{Q}$ and a PDM with a continuously differentiable $\bar{\mathcal{F}}$. One of the following two cases holds:
\begin{itemize}
	\item {\bf (Case~I)} If the PDM is weak $\delta$-antipassive then $\mathbb{PE}(\bar{\mathcal{F}},\mathcal{Q})$ is globally attractive. If the PDM is $\delta$-antipassive then $\mathbb{PE}(\bar{\mathcal{F}},\mathcal{Q})$ is globally asymptotically stable.
	\item {\bf (Case~II)} If the PDM is weak $\delta$-antipassive with positive deficit $\nu^*$ and the PBR EDM is $\delta$-passive with surplus $\eta^*>\nu^*$ then $\mathbb{PE}(\bar{\mathcal{F}},\mathcal{Q})$ is globally attractive. If the PDM is $\delta$-antipassive with positive deficit $\nu^*$ and the PBR EDM is $\delta$-passive with surplus $\eta^*>\nu^*$ then $\mathbb{PE}(\bar{\mathcal{F}}, \mathcal Q)$ is globally asymptotically stable.
\end{itemize}

\end{theorem} 
\noindent \rule{\columnwidth}{1pt}

A proof of Theorem \ref{thm:PBRConvergence}  is given in Appendix \ref{appendix_b}.

\section{Smoothing-Anticipatory PDM: \\Definition and $\delta$-antipassivity properties}
\label{sec:AnticipatoryPDM}

In this section, we study the class of so-called smoothing-anticipatory PDM, which extends both the anticipatory and smoothing modified payoff dynamics, which were considered in~\cite{Shamma2005Dynamic-fictiti,Fox2013Population-Game,Arslan2006Anticipatory-le,Chasparis2012Distributed-dyn} to account for learning dynamics~\cite{Shamma2019Game-theory-lea}. 

We start with defining the smoothing-anticipatory PDM class in a way that is consistent with our formulation. Subsequently, in \S\ref{subsec:AnticipPDMAffine} and \S\ref{subsec:AnticipPDMPotential},  we establish sufficient conditions under which a smoothing-anticipatory PDM is $\delta$-antipassive for the case when the stationary population game is potential.

% and affine. In Section~\ref{subsec:AnticipPDMPotential} we will provide sufficient conditions to $\delta$-antipassivity for the case when the stationary game is potential but not necessarily affine.

% ADD INTRO: LEARNING DYNAMICS

\begin{definition} {\bf (Smoothing-Anticipatory PDM)} Consider that ${\mathcal{F}:\mathbb{X} \rightarrow \mathbb{R}^n}$ is a given continuously differentiable map defining a population game. Given a positive constant $\alpha$ and nonnegative parameters $\mu_0$, $\mu_1$, and $\mu_2$ satisfying $\mu_0+\mu_1=1$, the associated \underline{smoothing-anticipatory} PDM is defined as follows:
\begin{subequations}
\label{eq:SmoothingPDM}
\begin{align} \label{eq:SmoothingPDMa}
\dot{q}(t) =  & \alpha \Big( \mathcal{F} \big ( u(t) \big) - q(t) \Big) \\ \label{eq:SmoothingPDMb}
p(t) = &  \mu_0 \mathcal{F}\big ( u(t) \big) + \mu_1 q(t) + \mu_2 \dot{q}(t)
\end{align}
\end{subequations} for $t \geq 0$,  $q(0) \in \mathbb{R}^n$, and  $u \in \mathfrak{X}$. 
\end{definition} 

In order to show that~(\ref{eq:SmoothingPDMb}) complies with~(\ref{eq:PDM}), it suffices to notice that we can substitute the expression~(\ref{eq:SmoothingPDMa}) for $\dot{q}(t)$ to get the following alternative formula for $p(t)$:
$$ p(t) = (\mu_0 + \alpha\mu_2) \mathcal{F} \big ( u(t) \big) + (\mu_1 - \alpha \mu_2)q(t) $$

\begin{remark} \label{remark:smoothing_anticipatory_pdm} The following are parameter choices leading to existing PDM types:
\begin{itemize}
	\item When $\mu_0=1$ and $\mu_1=\mu_2=0$ the PDM is \underline{memoryless}.
	\item When $\mu_0=1$, $\mu_1=0$, and $\mu_2 >0$ we obtain an \underline{anticipatory PDM}, as considered in~\cite{Shamma2005Dynamic-fictiti,Fox2013Population-Game,Arslan2006Anticipatory-le,Chasparis2012Distributed-dyn}.
	\item The \underline{smoothing PDM} considered in~\cite{Fox2013Population-Game} is obtained when $\mu_0=0$, $\mu_1=1$, and $\mu_2=0$.
\end{itemize} 
\end{remark}

The following proposition guarantees that any smoothing-anticipatory PDM satisfies Assumptions~\ref{assump:BoundedPDM} and~\ref{assump:StationaryGame}. 

\begin{proposition}
\label{prop:SmoothingAnticipPDMSatisfyAssump} Let ${\mathcal{F}:\mathbb{X} \rightarrow \mathbb{R}^n}$ be a given continuously differentiable map defining a population game. Given positive $\alpha$ and nonnegative $\mu_0$, $\mu_1$, and $\mu_2$ satisfying $\mu_0+\mu_1=1$, the associated smoothing-anticipatory PDM satisfies Assumption~\ref{assump:BoundedPDM} (boundedness). It also satisfies Assumption~\ref{assump:StationaryGame}, and the stationary  population game is $\bar{\mathcal{F}} = \mathcal{F}$.
\end{proposition}
\begin{proof} We start by writing the explicit solution of (\ref{eq:SmoothingPDM}) for a given input $u$ in $\mathfrak{X}$ and $q(0)$ in $\mathbb{R}^n$:
\begin{equation}
\label{eq:smoothSoln}
q(t) = \alpha \int_0^t e^{-\alpha(t-\tau)} \mathcal{F}\big( u(\tau) \big ) \, \mathrm d \tau + e^{-\alpha t} q(0), \quad t \geq 0
\end{equation} Since $\mathcal{F}$ is continuous, we conclude from~(\ref{eq:smoothSoln}) that the following inequality holds and the right-hand side is finite:
\begin{equation}
\|q\| \leq \max_{z \in \mathbb{X}} \|\mathcal{F}(z) \| + \| q(0) \|,
\end{equation} which implies that the \underline{PDM is bounded} in the sense of Assumption~\ref{assump:BoundedPDM}. In order to prove that $\bar{\mathcal{F}} = \mathcal{F}$, we use a Lyapunov-like argument based on the following function:
\begin{equation}
\mathcal{L}(z,s) = \frac{1}{2\alpha} \big (\mathcal{F}(z)-s \big )^T \big (\mathcal{F}(z)-s \big ), \quad z \in \mathbb{X}, \ s \in \mathbb{R}^n
\end{equation} Now, we can use this to calculate the following derivative:
\begin{multline}
\frac{\mathrm d}{\mathrm d t} \mathcal{L}(u(t), q(t)) = - 2\alpha \mathcal{L}(u(t), q(t)) \\ + \frac{1}{\alpha} \Big (\mathcal{F} \big( u(t) \big )-q(t) \Big )^T D\mathcal{F}\big( u(t) \big) \dot{u}(t)
\end{multline}

Since $\mathcal{F}$ is continuously differentiable, $\|u\| \leq m$, and, as we proved above, $q$ is bounded, we conclude that \underline{when $\dot{u}(t)$ tends to zero} the second term on the right-hand side above vanishes, which implies that the following limit holds:
\begin{equation}
\label{eq:limitstationarygameforq}
\lim_{t \rightarrow \infty} \|q(t) - \mathcal{F} \big (u(t) \big ) \| = 0
\end{equation}
As a result, from~(\ref{eq:SmoothingPDMa}) we infer that $\dot{q}$ tends to zero and from~(\ref{eq:SmoothingPDMb}) we conclude that $\lim_{t \rightarrow \infty} \|p(t) - \mu_0\mathcal{F}\big (u(t) \big ) - \mu_1 q(t)\| =0$. Finally, using this fact, that $\mu_0+\mu_1=1$, and (\ref{eq:limitstationarygameforq}) we infer that $\lim_{t \rightarrow \infty} \|p(t) - \mathcal{F} \big (u(t) \big ) \| = 0$.

Also, we note that the set $\{(z,s) \in \mathbb X \times \mathbb R^n \,|\, \mathcal H(z,s) = \mathcal F(z) \}$ is equivalent to the compact subset $\{(z,s) \in \mathbb X \times \mathbb R^n \,|\, s = \mathcal F(z) \}$ if $\mu_1 \neq \alpha \mu_2$, or the entire set $\mathbb X \times \mathbb R^n$ otherwise. This means that Assumption \ref{assumption:stationary_game} is satisfied and the stationary population game is well-defined and is given by $\bar{\mathcal{F}} = \mathcal{F}$.
%  the entire set $\mathbb X \times \mathbb R^n$  or 
% $\{(z,s) \in \mathbb X \times \mathbb R^n \,|\, s = \mathcal F(z) \}$, which is a compact subset of $\mathbb X \times \mathbb R^n$. 
\end{proof}

\subsection{Smoothing-Anticipatory PDM when $\mathcal{F}$ is Potential Affine}
\label{subsec:AnticipPDMAffine}

 Notice that $\delta$-antipassivity for the classes of anticipatory and smoothing PDM was studied separately in~\cite{Fox2013Population-Game} for the case when $\mathcal{F}$ is affine, potential, and strictly contractive. Below, in Proposition~\ref{prop:SmoothingAnticipatoryPDMStability}, we determine sufficient conditions for weak $\delta$-antipassivity for smoothing-anticipatory~PDM associated with affine potential $\mathcal{F}$ that is not required to be strictly contractive. 

\begin{definition}{\bf (Projection Matrix $\Phi$)}
Henceforth, we will use a projection matrix $\Phi$ in $\mathbb{R}^{n \times n}$ defined as follows:
\begin{equation}
\Phi_{ij} = \begin{cases} \frac{n-1}{n} & \text{if $i=j$} \\ - \frac{1}{n} & \text{otherwise} \end{cases}
\end{equation}
\end{definition} 

\begin{definition} Given a real symmetric matrix $M$, the largest eigenvalue of $M$ is represented as $\bar{\lambda}(M)$.
\end{definition}

\begin{proposition} 
\label{prop:SmoothingAnticipatoryPDMStability}
Let $\mathcal{F}$ be an affine population game specified as follows:
\begin{equation}
\label{eq:AffineGame}
\mathcal{F}(z) = Fz+\bar{r}, \quad z \in \mathbb{X}
\end{equation} where $F \in \mathbb{R}^{n \times n}$ is such that $\Phi F \Phi$ is symmetric\footnote{This implies that $\mathcal{F}$ is a potential population game. See~\cite{Sandholm2010Population-Game} for more details.},  and $\bar{r}$ is a constant vector in $\mathbb{R}^n$. Consider that $\mathcal{F}$, a given positive $\alpha$, and nonnegative $\mu_0$, $\mu_1$, and $\mu_2$ satisfying $\mu_0+\mu_1=1$ define a smoothing-anticipatory PDM. Let $\lambda^*$ be selected as: $$\lambda^* = \bar{\lambda}\big( \Phi F \Phi) $$ 
The PDM satisfies the following: 
\begin{itemize}
	\item[i)] If $\lambda^* = 0$  then the PDM is weak $\delta$-antipassive.
	\item[ii)] If $\lambda^* >0$ and $\mu_0+\alpha \mu_2 \leq 1$ then the PDM is weak $\delta$-antipassive with deficit $\lambda^*$.
	\item[iii)] If $\lambda^* >0$ and $\mu_0+\alpha \mu_2 > 1$ then the PDM is weak $\delta$-antipassive with deficit $(\mu_0+\alpha \mu_2) \lambda^*$.
\end{itemize}
\end{proposition}

A proof of Proposition \ref{prop:SmoothingAnticipatoryPDMStability}  is given in Appendix \ref{appendix_c}.

% The proof of i) mirrors that of~\cite[Theorem~4.5, Theorem~4.7]{Fox2013Population-Game} by showing that $\mathcal{L}$ given below is a $\delta$-storage function satisfying~(\ref{eq:DAPassivePDM}).
% \begin{equation}
% \mathcal{L}(z,s) = -(Fz +\bar{r} -s )^T F^{-1} (Fz+\bar{r}-s), \quad z \in \mathbb{X}, \ s \in \mathbb{R}^n
% \end{equation}

\begin{proposition} If, in addition to the conditions of Proposition~\ref{prop:SmoothingAnticipatoryPDMStability}, a PDM is specified by a symmetric and negative definite $F$ then it is $\delta$-antipassive. 
\end{proposition}
\begin{proof}

The proof mirrors that of~\cite[Theorem~4.5, Theorem~4.7]{Fox2013Population-Game} by showing that $\mathcal{L}$ given below is a $\delta$-antistorage function satisfying~(\ref{eq:DAPassivePDM}).
\begin{multline}
    \mathcal{L}(z,s) = -(Fz +\bar{r} -s )^T F^{-1} (Fz+\bar{r}-s),\\ z \in \mathbb{X}, \ s \in \mathbb{R}^n
\end{multline}
\end{proof}

\subsection{Smoothing PDM when $\mathcal F$ is Potential Nonlinear} \label{subsec:AnticipPDMPotential}
In the following proposition, we establish $\delta$-antipassivity of smoothing PDM.
% (see Remark \ref{remark:smoothing_anticipatory_pdm}). 
We note that unlike the case considered in Section~\ref{subsec:AnticipPDMAffine},
% Proposition~\ref{prop:SmoothingAnticipatoryPDMStability}, 
we allow the stationary population game $\mathcal F$ to be nonlinear.
% we consider that the associated stationary population game is defined by a nonlinear map $\mathcal F$.

\begin{proposition} 
\label{prop:SmoothingPDMStability_Potential}
Let $\mathcal{F}: \mathbb X \to \mathbb R^n$ admit a  strictly concave potential function $f: \mathbb R^n \to \mathbb R$ for which $\nabla f = \mathcal F$ holds on $\mathbb X$. Suppose that $f$ is twice continuously differentiable and $\mathrm{Im}(\mathcal F) \subset \mathrm{int}(\mathbb D^\ast)$ holds, where  $\mathrm{Im}(\mathcal F)$ and $\mathbb D^\ast$ are, respectively, defined as
\begin{align*}
  \mathrm{Im}(\mathcal F) &= \left\{\mathcal F(z) \,|\, z \in \mathbb X \right\} \\
  \mathbb D^\ast &= \left\{s \in \mathbb R^n \,\bigg|\, \sup_{y \in \mathbb R^n} (f(y) - s^Ty) < \infty \right\}
% \end{align*}
% \begin{align*}
\end{align*}
Consider that $\mathcal F$ and a given positive $\alpha$ define a smoothing PDM for which $q(0) \in \mathrm{int}(\mathbb D^\ast)$ holds. The PDM is $\delta$-antipassive and its $\delta$-antistorage function is given by
\begin{align} \label{eq:PDM_antistorage}
  \mathcal L(z,s) = \alpha \left[ \sup_{y \in \mathbb R^n} (f(y) - s^Ty) - (f(z) - s^Tz) \right]
\end{align}
\end{proposition}

A proof of Proposition \ref{prop:SmoothingPDMStability_Potential}  is given in Appendix \ref{appendix_c}.

% \begin{figure} [t]
% \begin{center}
% \includegraphics[width=2.0in]{bnn_coordination.pdf}
% \caption{Mean population state trajectories induced by the BNN EDM under a memoryless PDM defined by the coordination population game \eqref{eq:coordination_game}.}
% \label{fig:bnn_coordination}
% \end{center}
% % \end{figure}
% % \begin{figure} [t]
% \begin{center}
% \includegraphics[width=2.0in]{logit_coordination.pdf}
% \caption{Mean population state trajectories induced by the Logit EDM under a memoryless PDM defined by the coordination population game \eqref{eq:coordination_game}.}
% \label{fig:logit_coordination}
% \end{center}
% \end{figure}

% \begin{figure} [t]
% \begin{center}
% \includegraphics[width=2.0in]{bnn_rps.pdf}
% \caption{Mean population state trajectories induced by the BNN EDM under an anticipatory PDM defined by the RPS population game \eqref{eq:rps_game}.}
% \label{fig:bnn_rps}
% \end{center}
% % \end{figure}
% % \begin{figure} [t]
% \begin{center}
% \includegraphics[width=2.0in]{bnn_congestion.pdf}
% \caption{Mean population state trajectories induced by the BNN EDM under a smoothing PDM defined by the congestion population game \eqref{eq:congestion_game}.}
% \label{fig:bnn_congestion}
% \end{center}
% \end{figure}

\begin{remark}
In Proposition \ref{prop:SmoothingPDMStability_Potential}, we assume that the initial condition of the smoothing PDM satisfies $q(0) \in \mathrm{int}(\mathbb D^\ast)$. In fact, based on Lemma~\ref{lemma:convex_invariance} in Appendix, it can be verified that given any $q(0) \in \mathbb R^n$ and $u \in \mathfrak X$, the state      $q(t)$ of the PDM converges to a compact convex subset of $\mathrm{int}(\mathbb D^\ast)$ containing $\mathrm{Im}(\mathcal F)$. This implies that the state trajectory $q$ enters the set $\mathrm{int}(\mathbb D^\ast)$ in a finite time.
\end{remark}

% \subsection{Numerical Examples}
\section{Numerical Examples and Simulations} \label{sec:exmaples}
Based on the population games described in Examples~\ref{example:conjection5link}-\ref{ex:TaskAllocation}, we provide numerical examples, with simulation results, to demonstrate how our main results can be used to assess global convergence to equilibria of the mean closed loop model~\eqref{eq:ClosedLoop}. In particular, in Examples~\ref{ex:congestion_game_revisited} and \ref{ex:demand_response_game_revisited}, we consider two different cases in which the deterministic payoff is governed by the conventional population game (memoryless PDM) and the smoothing-anticipatory PDM. In Example~\ref{ex:task_allocation_game_revisited}, we examine how the $\delta$-passivity surplus for EDM affects the stability of Nash equilibria of an underlying stationary population game.

\begin{figure} [t]
\begin{center}
\subfigure[]{
\includegraphics[trim={.2in .5in 0 0},width=2.5in]{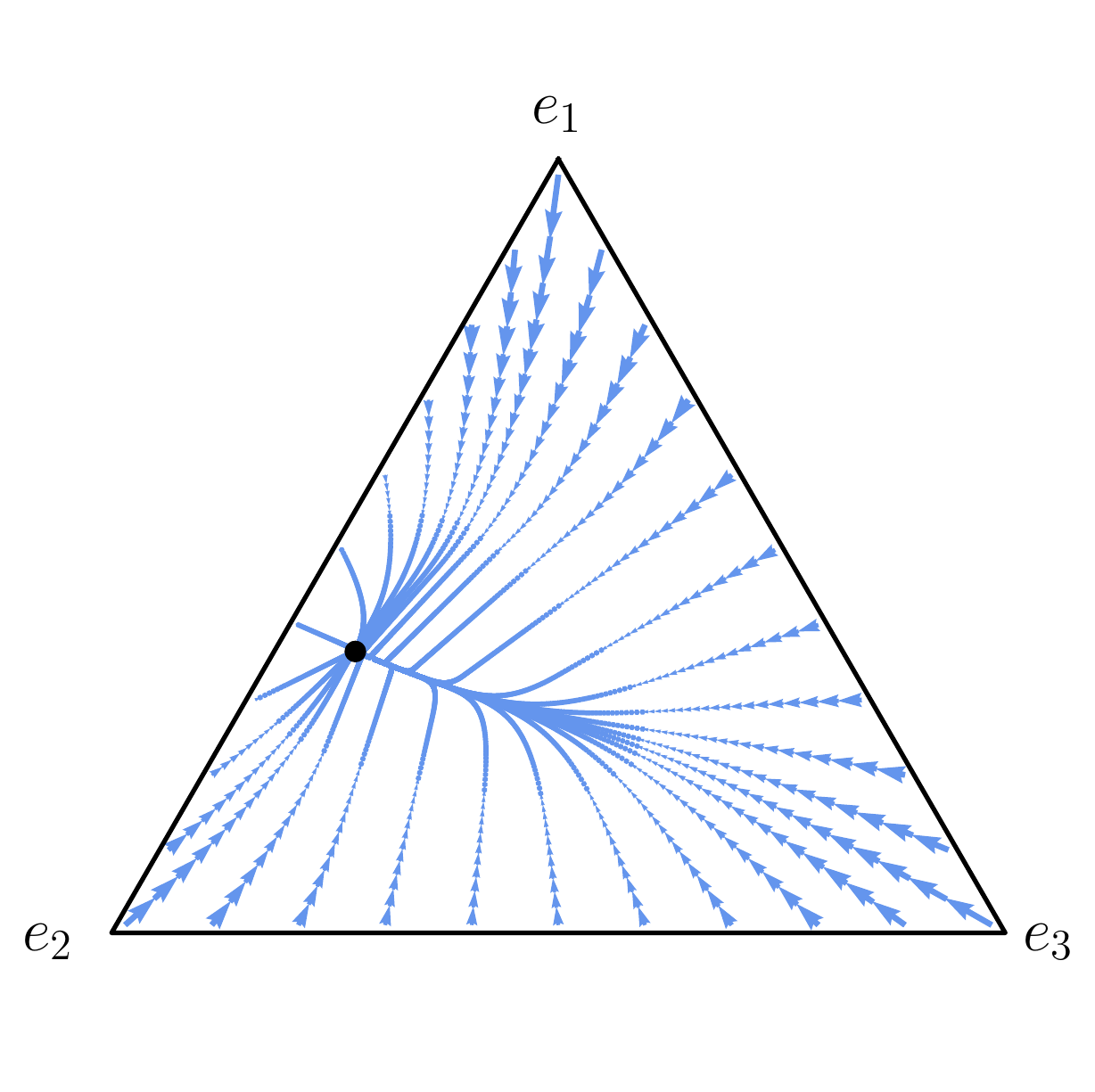}
\label{fig:BNNCongestion_a}}

\subfigure[]{
\includegraphics[trim={.2in .5in 0 0},width=2.5in]{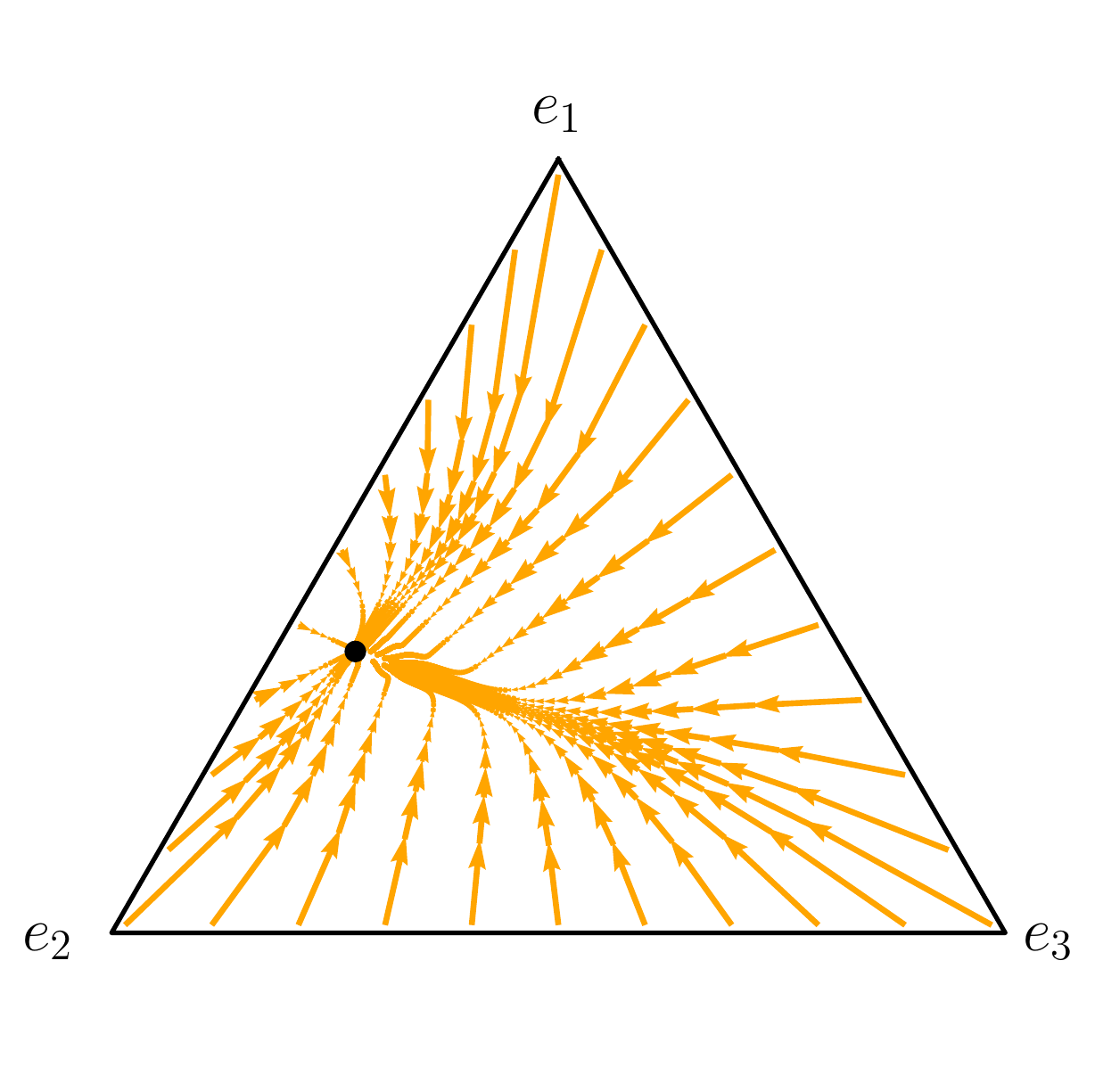}
\label{fig:BNNCongestion_b}}
\caption{Population state trajectories induced by the BNN EDM under (a) the memoryless PDM  and (b) the anticipatory PDM both defined by the congestion game  \eqref{eq:CongestionGame}. The black dots indicate the unique Nash equilibrium of the game.}
\label{fig:BNNCongestion}
\end{center}
\end{figure}

\begin{figure} [t]
\begin{center}
\includegraphics[trim={0.55in .1in 0.5in 0},width=2.0in]{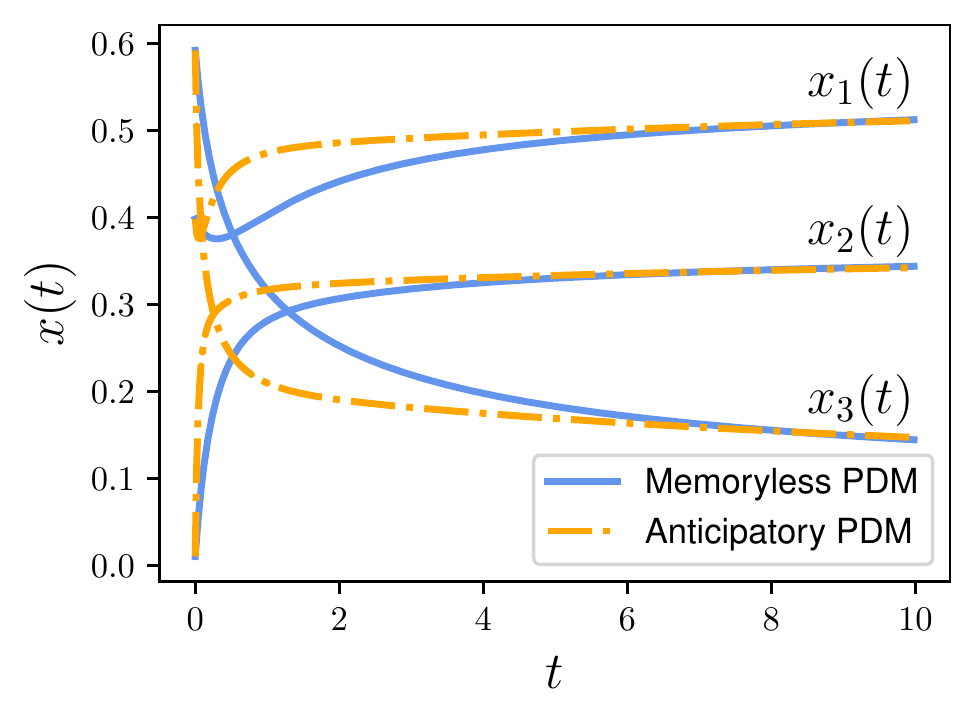}

% \subfigure[]{
% \includegraphics[trim={0.55in .1in 0 0},width=2.0in]{./figures/congestion_game_bnn_anticipation_traj.pdf}}

% \subfigure[]{
% \includegraphics[trim={0.55in .1in 0 0},width=2.0in]{./figures/congestion_game_bnn_traj.pdf}}

\caption{Comparison of two population state trajectories induced by the BNN EDM under the memoryless PDM (solid blue lines) and the anticipatory PDM (dotted orange lines) for \eqref{eq:CongestionGame}, where both the trajectories have the same initial condition.}
\label{fig:BNNCongestionTraj}
\end{center}
\end{figure}

% \subsubsection{Congestion Game}
\begin{example} \label{ex:congestion_game_revisited}
We revisit the congestion game described in Example~\ref{example:conjection5link}. Consider that the function $\mathscr{D}_i$ is defined as follows:
\begin{align} \label{eq:DelayFunction}
    \mathscr{D}_i(s) = \begin{cases} 2s & \text{if } i = 2 \\
    s & \text{otherwise}
    \end{cases}
\end{align}
Note that the definition of $\mathscr{D}_i$ suggests that the delay incurred on link~2 is twice that on other link. From \eqref{eq:DelayFunction}, we can explicitly express the payoff function $\mathcal F^{Ex.\ref{example:conjection5link}}$ as
\begin{align} \label{eq:CongestionGame}
    \mathcal F^{Ex.\ref{example:conjection5link}} \left( z \right) = \begin{bmatrix}
    -3 z_1 - z_3 \\
    -2 z_2 - z_3\\
    -z_1 - z_2 - 3z_3
    \end{bmatrix}
\end{align}

Using $\mathcal F^{Ex.\ref{example:conjection5link}}$, we define \textnormal{(a)} a memoryless PDM and \textnormal{(b)} an anticipatory PDM ($\mu_2 = 5$) for which both the PDMs are $\delta$-antipassive. In feedback interconnection with the BNN EDM, we construct two mean closed loop models with which we compare their asymptotic and transient behavior.

Fig.~\ref{fig:BNNCongestion} and Fig.~\ref{fig:BNNCongestionTraj} illustrate population state trajectories derived by the two mean closed loop models. As we observe from Fig.~\ref{fig:BNNCongestion}, all the trajectories converge to the unique Nash equilibrium of \eqref{eq:CongestionGame}, which we can also conclude from Theorem~\ref{thm:TheoremIntegrableEPT} and Proposition~\ref{prop:SmoothingAnticipatoryPDMStability}. On the other hand, the trajectories exhibit different transient behavior in that, as illustrated in  Fig.~\ref{fig:BNNCongestionTraj}, the trajectories induced under the anticipatory PDM converge at faster rate than those induced under the memoryless PDM. The simulation results suggest that the agents' anticipation of their future payoffs would improve the rate of convergence to the Nash equilibrium.
\end{example}

\begin{figure} [t]
\begin{center}
\subfigure[]{
\includegraphics[trim={.2in .5in 0 0},width=2.5in]{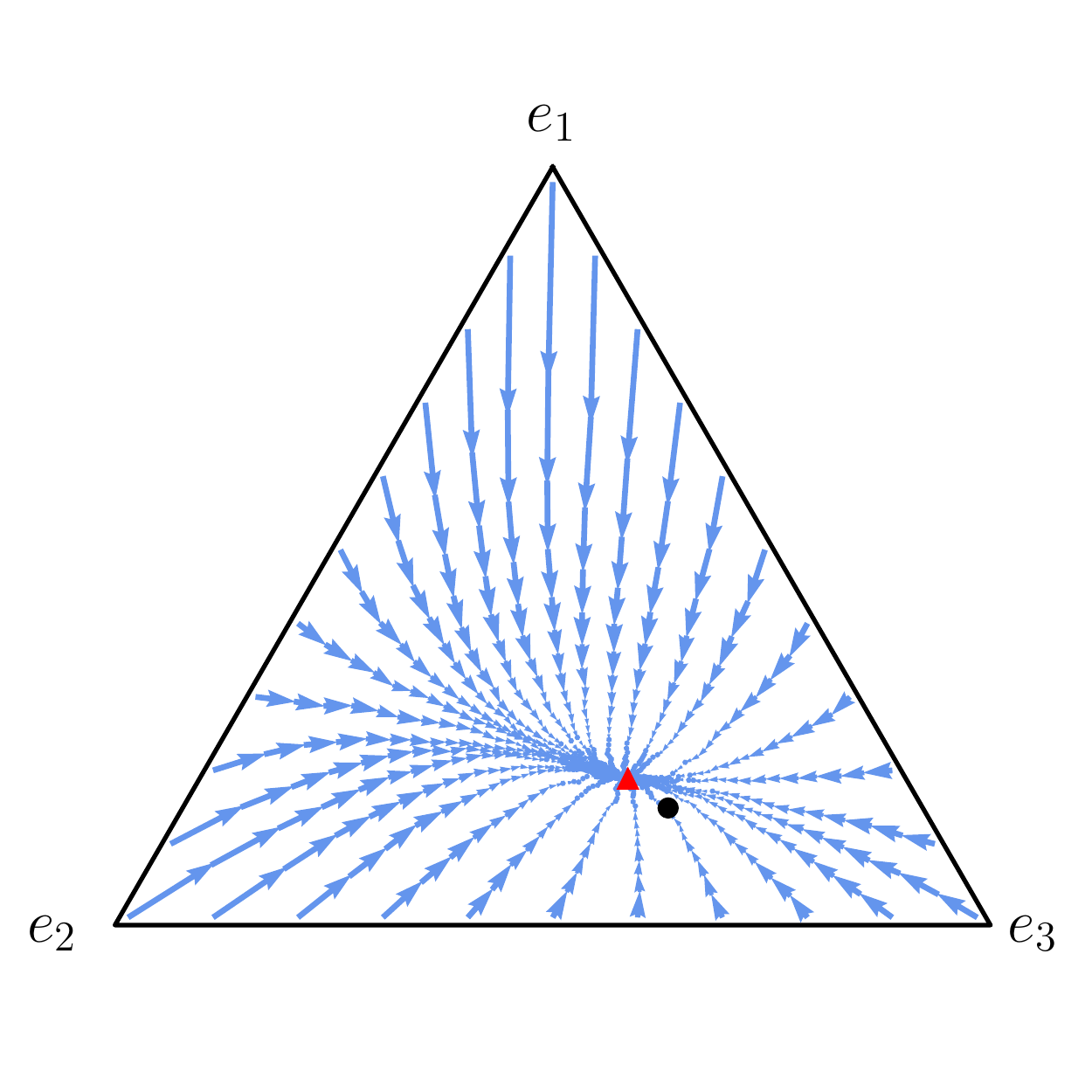}
\label{fig:SmithDemandResponse_a}}

\subfigure[]{
\includegraphics[trim={.2in .5in 0 0},width=2.5in]{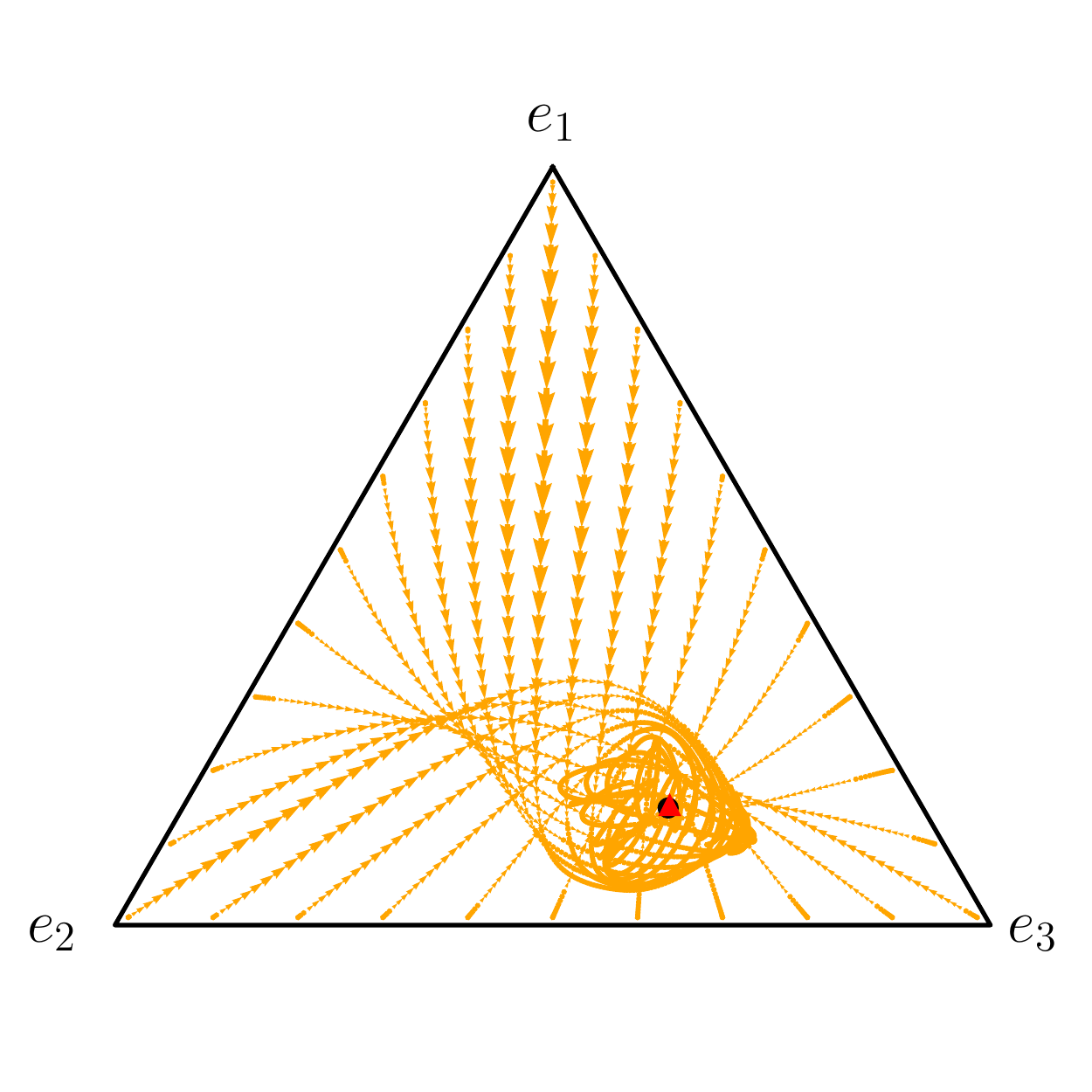}
\label{fig:SmithDemandResponse_b}}
\caption{Averaged population state trajectories induced by the Smith EDM under (a) the memoryless PDM  and (b) the smoothing PDM both defined by the electricity demand response game  \eqref{eq:DemandResponseGame} subject to an additive perturbation. The black dots indicates the unique Nash equilibrium of the game where as the red triangles represent the limit point of the averaged  trajectories.}
\label{fig:SmithDemandResponse}
\end{center}
\end{figure}

\begin{figure} [t]
\begin{center}
\includegraphics[trim={0.55in .1in 0.5in 0},width=2.4in]{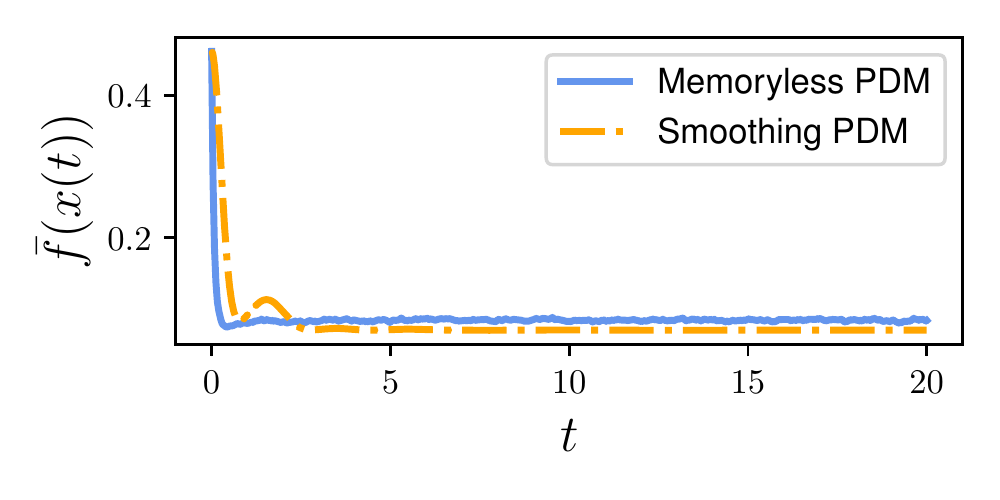}

\caption{The average of the cost $\bar f (x(t))$ evaluated along multiple population state trajectories induced by the Smith EDM under the memoryless PDM (solid blue line) and the smoothing PDM (dotted orange line) defined by \eqref{eq:CostSignalEx}, where all the trajectories start from a same initial condition.}
\label{fig:SmithDemandReponseTraj}
\end{center}
\end{figure}

% \subsubsection{Electricity Demand Response Game}
\begin{example} \label{ex:demand_response_game_revisited}
We revisit the electricity demand response game described in Example~\ref{ex:DemandResponse}. We select the parameter $y$ describing the reduction level of consumers' power consumption to be $y = \left( 10^{-2}, \, 10^{-1}, \, 1 \right)$ and define the cost function \eqref{eq:CostFunction} at the EPU as
\begin{align} \label{eq:CostFunctionEx}
    \bar f (z) = 5 z_1^2 + \frac{5}{2} z_2^2 + \frac{1}{2} z_3^2 + \left(y^T z- 1 \right)
\end{align}
in which case the cost signal~\eqref{eq:DemandResponseGame} is specified by
\begin{align} \label{eq:CostSignalEx}
    \mathcal F^{Ex.\ref{ex:DemandResponse}}(z) = \begin{bmatrix} -10 z_1 - 10^{-2} \\ -5 z_2 - 10^{-1} \\ - z_3 - 1 \end{bmatrix}
\end{align}
Using $\mathcal F^{Ex.\ref{ex:DemandResponse}}$, we define \textnormal{(a)} a memoryless PDM and \textnormal{(b)} a smoothing PDM ($\alpha=1$) for which both the PDMs are $\delta$-antipassive. We consider that the agents in the population revise their strategies according to the Smith protocol \eqref{eq:SmithProtocol}. Using Theorem~\ref{thm:IPCConvergence} and Proposition~\ref{prop:SmoothingAnticipatoryPDMStability}, we assert that population state trajectories derived by the Smith EDM for both the PDMs converge to the unique Nash equilibrium of \eqref{eq:CostSignalEx}.

In this example, we aim at investigating the effect of the payoff smoothing. For this purpose, we suppose that the cost signal \eqref{eq:CostSignalEx} is subject to an additive perturbation that has the standard normal distribution. We execute multiple rounds of simulations for each pre-selected initial condition and compute the average of the resulting trajectories. Also we assess the average of the cost function \eqref{eq:CostFunctionEx} evaluated along the multiple trajectories starting from one fixed initial condition.

Fig.~\ref{fig:SmithDemandResponse} illustrates averaged population state trajectories starting from different initial conditions. As we can observe from the figure, in the memoryless PDM case, the trajectories tend to approach the Nash equilibrium; however, because of the perturbation, the trajectories do not precisely converge to the Nash equilibrium. On the other hand, in the smoothing PDM case, the trajectories converge to a limit point that is closer to the Nash equilibrium than the memoryless PDM case. Fig.~\ref{fig:SmithDemandReponseTraj} depicts the average of the cost $\bar f(x(t))$ evaluated along the multiple population state trajectories in both of the memoryless PDM and smoothing PDM cases. We can observe that the smoothing PDM incurs a smaller average cost than the memoryless PDM case.
 
% in the smoothing PDM case, as smoothing PDM filters out short term fluctuations in the cost signal 
 
From these observations, we conclude that as the payoff smoothing removes short term fluctuations in the cost signal, it allows the agents to revise their strategies depending on longer term trends in the signal. Interestingly, in our demand response example, such payoff smoothing allows the population to select better strategy profiles that incur smaller costs at the EPU.
\end{example}

\begin{figure} [t]
\centering
\includegraphics[trim={0.55in .1in 0.5in 0},width=1.5in]{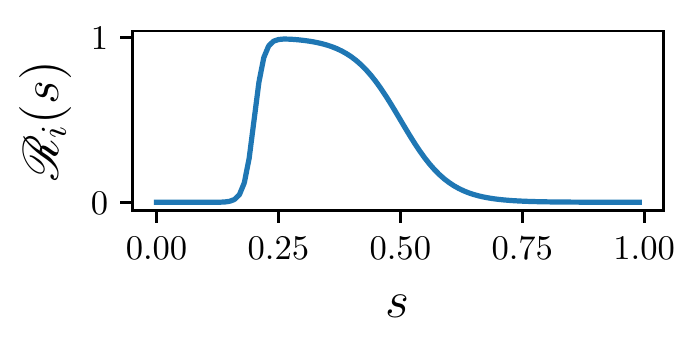}
\caption{The illustration of the reward function $\mathscr{R}_i$ in the task allocation game.}
\label{fig:TaskAllocationReward}
\end{figure}

\begin{figure} [t]
\begin{center}
\subfigure[]{
\includegraphics[trim={.2in .5in 0 0},width=2.5in]{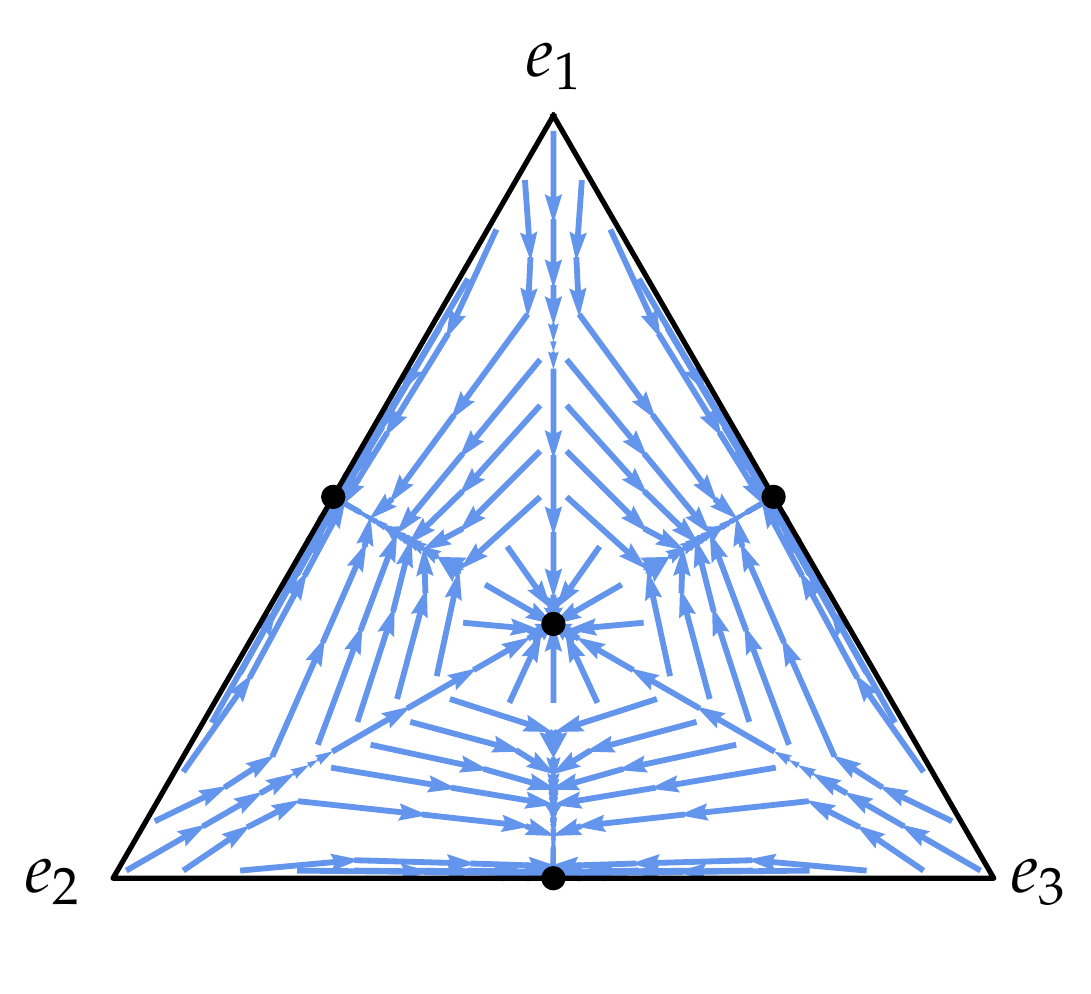}
\label{fig:LogitTaskAllocation_a}
}

\subfigure[]{
\includegraphics[trim={.2in .5in 0 0},width=2.5in]{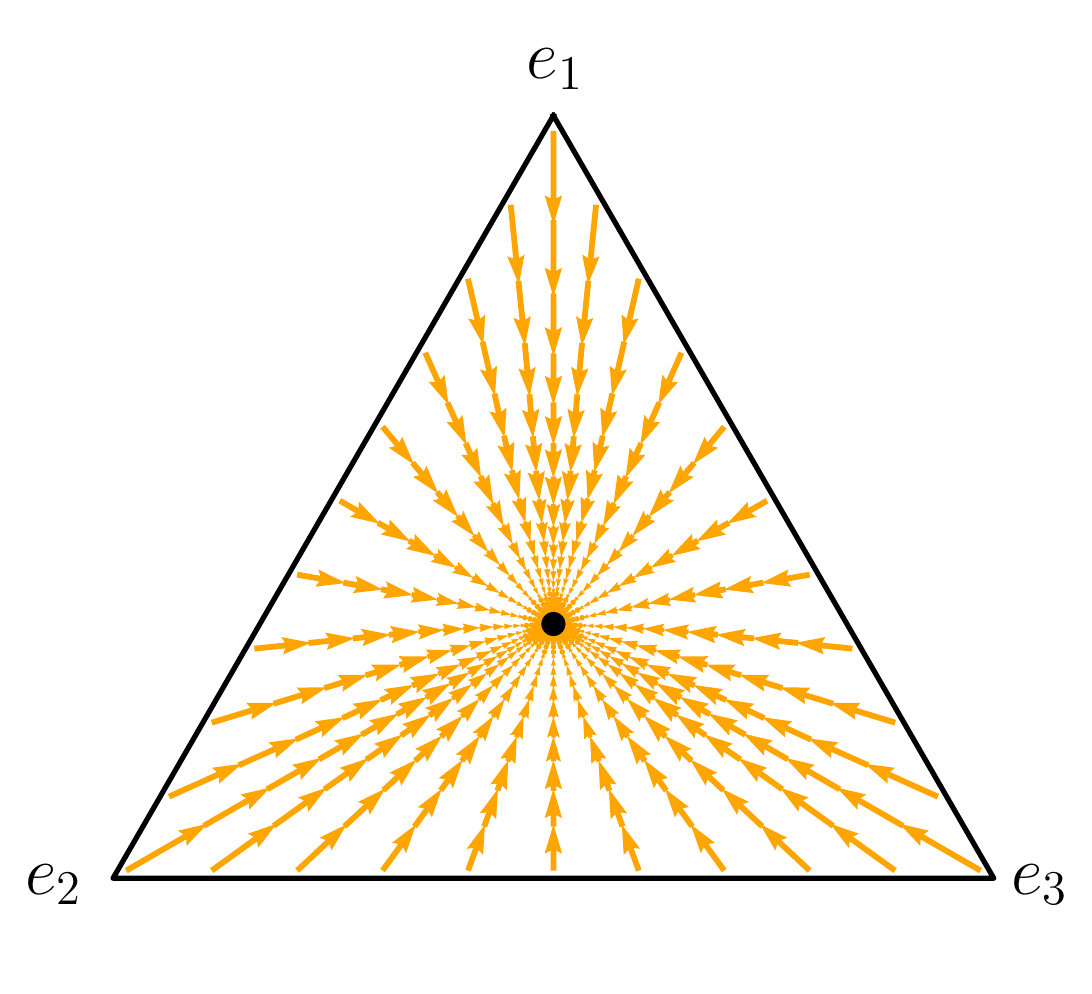}
\label{fig:LogitTaskAllocation_b}
}

\caption{Population state trajectories induced by the logit PDM with the noise level $\eta$ determined as (a) $\eta=0.01$ and (b) $\eta=25$ in the task allocation game  \eqref{eq:TaskAllocationReward}. The black dots indicate the Nash equilibria of the game in (a) and the perturbed equilibrium under the logit protocol in (b).}
\label{fig:LogitTaskAllocation}
% \caption{Mean population state trajectories induced by the BNN EDM under a memoryless PDM defined by the coordination population game \eqref{eq:coordination_game}.}
% \label{fig:bnn_coordination}
\end{center}
\end{figure}

% \subsubsection{Task Allocation Game}
\begin{example} \label{ex:task_allocation_game_revisited}
Let us revisit the task allocation game described in  Example~\ref{ex:TaskAllocation} in which the reward functions $\mathscr{R}_i$ are identical and defined as follows (also see Fig.~\ref{fig:TaskAllocationReward} for the illustration of $\mathscr{R}_i$):
\begin{multline} \label{eq:TaskAllocationReward}
    \mathscr{R}_i (s) = \frac{1}{1+\exp(-100(s-0.2))} \\ - \frac{1}{1+\exp(-20(s-0.5))}, ~ s \in [0,1]
\end{multline}
Here we consider the task allocation game defined by the memoryless PDM $\mathcal F^{Ex.\ref{ex:TaskAllocation}}$ which is $\delta$-antipassive with deficit $\nu^\ast (< 25)$.

Similar to the scenario discussed in \cite[\S V]{8439076}, the reward function \eqref{eq:TaskAllocationReward} can be used to define a task allocation game in which each task requires a minimum number of agents, and hence the reward $\mathscr{R}_i (x_i)$ is negligible when the portion $x_i$ of the population is below a certain threshold, specified in the requirement. On the other hand, once $x_i$ goes beyond the threshold, the reward tends to increase until it reaches its maximum and starts to decrease as the portion $x_i$ continues to increase.

% as the total reward given to the subpopulation $x_i$ reaches its maximum and the reward allocated to each agent is reduced as $x_i$ becomes larger, the value of $\mathscr{R}_i (x_i)$ decreases.

Using the task allocation game, in what follows, we examine asymptotic behavior of the logit EDM (described in  Example~\ref{ex:logit}) subject to two different selections of the noise level $\eta$. Fig.~\ref{fig:LogitTaskAllocation_a} depicts population state trajectories induced by the logit EDM with small $\eta$ ($\eta=0.01$). Notice that the associated mean closed loop model has $4$ stable stationary population states. The stationary state at the center of the simplex is the socially optimal Nash equilibrium of the game, maximizing its average reward, whereas the rests are suboptimal Nash  equilibrium points.

On the contrary, for the mean closed loop model defined by the task allocation game and the logit EDM with larger $\eta$ ($\eta=25$), there is one stationary population state, which is the socially optimal Nash equilibrium. As illustrated in Fig.~\ref{fig:LogitTaskAllocation_b}, the stationary population state is globally asymptotically stable which we can also conclude using  Theorem~\ref{thm:PBRConvergence}.

The results presented in this example suggest that the logit protocol with sufficiently large noise level $\eta$ prevents the population state trajectories from converging to the boundary of the state space $\mathbb X$, which is related to the property \eqref{eq:boundary} of the payoff perturbation. Consequently, in this task allocation game, the agents adopting the logit protocol could avoid selecting suboptimal strategy profiles that are all residing in the boundary of $\mathbb X$.
\end{example}

\appendix
%\section{} \label{sec:appendix_a}
%\renewcommand{\thesubsection}{\arabic{subsection}}

The following three lemmas are key to the proofs of the main results.

\begin{lemma} \label{lemma:passivity_edm}
Given an EDM \eqref{PopulationDynamic} specified by $\mathcal V$, consider the following two relations: For every $z$ and $r$ in $\mathbb X$ and $\mathbb R^n$, respectively,
\begin{subequations} \label{eq:passivity_characterization}
\begin{align}
  \nabla_r \mathcal S(z, r) &= \mathcal V(z, r) \label{eq:passivity_characterization_01} \\
  \nabla_z^T \mathcal S(z, r) \mathcal V(z, r) & \leq - \eta \mathcal V^T(z, r) \mathcal V(z, r) \label{eq:passivity_characterization_02}
\end{align}
\end{subequations}
where $\eta$ and $\mathcal S: \mathbb X \times \mathbb R^n \to \mathbb R_+$ are a nonnegative constant and a map, respectively. The following two statements are true:
\begin{enumerate}
\item  The EDM is \textit{$\delta$-passive} if and only if there is a continuously differentiable $\mathcal S$ satisfying \eqref{eq:passivity_characterization} with $\eta = 0$.
  
\item For positive $\eta$, the EDM is qualified as $\delta$-passive with surplus $\eta$ if and only if there is continuously differentiable $\mathcal S$ satisfying \eqref{eq:passivity_characterization}.
\end{enumerate}
\end{lemma}
\begin{IEEEproof}
  As described in \S\ref{sec:PassivityConcepts}, the EDM \eqref{PopulationDynamic} can be viewed as a control-affine nonlinear system \eqref{eq:InterpEDMdPAssive} with the input $w^\delta(t)$, state $(x(t),w(t))$, and output $x^\delta(t)$. Using similar arguments as for the passivity characterization theorem (see, for instance, \cite[Theorem~1]{Hill1976The-stability-o}) for control-affine systems, we can see that there is a continuously differentiable $\mathcal S$ satisfying \eqref{eq:passivity_characterization} with $\eta \geq 0$ if and only if $\mathcal S$ satisfies the inequality \eqref{eq:DeltaPassive} with the same $\eta$. The statements $1)$ and $2)$ immediately follow from this equivalence.
\end{IEEEproof}

\begin{lemma} \label{lemma:ns_minimizer}
  Given a $\delta$-passive EDM \eqref{PopulationDynamic} with its $\delta$-storage function $\mathcal S$, let $\mathbb S = \{ (z,r) \in \mathbb X \times \mathbb R^n \,|\, \mathcal V (z,r) = 0 \}$ be the stationary points of the EDM and $\mathcal S^{-1}(0) = \{ (z,r) \in \mathbb X \times \mathbb R^n \,|\, \mathcal S (z,r) = 0\}$ be the global minima of $\mathcal S$. It holds that $\mathcal S^{-1}(0) \subseteq \mathbb S$ and the equality holds if the EDM satisfies Nash stationarity, where we assume that the set $\mathcal S^{-1}(0)$ is nonempty.

  % . It holds that $\{ (z,r) \in \mathbb X \times \mathbb R^n \,|\, \mathcal S (z,r) = 0\} \subseteq \mathbb S$ and the equality holds if the EDM satisfies Nash stationarity, where we assume that the set $\{ (z,r) \in \mathbb X \times \mathbb R^n \,|\, \mathcal S (z,r) = 0\}$ is non-empty.
\end{lemma}
\begin{IEEEproof}
The first part of the statement directly follows from the condition \eqref{eq:passivity_characterization_01} and the fact that at a global minimizer $(z^\ast,r^\ast)$ of $\mathcal S$, it holds that $\nabla_r \mathcal S(z^\ast,r^\ast) = 0$. 

Now suppose that the EDM satisfies Nash stationarity. To prove the second statement, it is sufficient to show that at each equilibrium point $(z_o,r_o)$ of \eqref{PopulationDynamic}, it holds that $\mathcal S(z_o,r_o) = 0$. To this end, let us consider the payoff map given by $\mathcal F_{z_{o}} (z) = -(z - z_{o})$ for a fixed $z_o$ in $\mathbb{X}$. Notice that $z_o$ is the unique Nash equilibrium of the population game $\mathcal F_{z_{o}}$. In what follows, we show that $\mathcal S(z_o, r_o) = 0$ holds for any choice of $(z_o, r_o)$ in $\mathbb S$.

Let $(z^\ast, r^\ast)$ be a global minimizer of $\mathcal S$, i.e., $\mathcal S\left( z^\ast, r^\ast \right) = 0$. By the first part of the statement, \eqref{eq:NashStationarity}, and \eqref{eq:passivity_characterization_01}, we have that $\mathcal V(z^\ast, \sigma r^\ast) = 0$ for all $\sigma$ in $\mathbb R_+$, and hence it holds that
$$\mathcal S (z^\ast, 0) = \mathcal S (z^\ast, r^\ast) - \int_0^1 \left( r^\ast \right)^T  \mathcal V(z^\ast, \sigma  r^\ast) \,\mathrm d \sigma = 0$$ 
By the continuity of $\mathcal S$, for each $\epsilon > 0$, there exists $\delta > 0$ for which $\mathcal S \left(z^\ast, \delta \mathcal F_{z_o} (z^\ast) \right) < \epsilon$ holds.

Given input $w(t) = \delta \mathcal F_{z_o}(x(t))$, let $x$ be the mean population state trajectory derived by the EDM. Since the EDM is $\delta$-passive, according to \eqref{eq:passivity_characterization}, the following relation holds for every positive constant $\delta$:
\begin{align} \label{eq:prop_storage_function_global_minimum_02}
  &\frac{\mathrm d}{\mathrm d t} \mathcal S (x(t), \delta \mathcal F_{z_o}(x(t))) \nonumber \\
  &\leq \delta \mathcal V^T(x(t), \delta \mathcal F_{z_o}(x(t))) \mathcal{DF}_{z_0}(x(t)) \mathcal V(x(t), \delta \mathcal F_{z_o}(x(t))) \nonumber \\
  &= -\delta \left\| \mathcal V \left( x(t), \delta \mathcal F_{z_o}(x(t)) \right) \right\|^2
\end{align}
Suppose that the mean population state $x(t)$ satisfies the initial condition $x(0) = z^\ast$. By an application of LaSalle's theorem \cite{Khalil1995Nonlinear-syste} and by \eqref{eq:NashStationarity}, we can verify that $(x(t), \delta \mathcal F_{z_o}(x(t)))$ converges to $\left( z_{o}, 0 \right)$ as $t \to \infty$. In addition, due to \eqref{eq:prop_storage_function_global_minimum_02}, we have that 
$$\mathcal S \left( z_o, 0 \right) \leq \mathcal S \left( z^\ast, \delta \mathcal F_{z_0} (z^\ast) \right) < \epsilon$$
Since this holds for every $\epsilon > 0$, we conclude that $\mathcal S (z_o, 0 ) = 0$. By the fact that $\mathcal V (z_o, \sigma r_o) = 0$ for all $\sigma$ in $\mathbb R_+$ if $(z_o, r_o)$ belongs to $\mathbb S$, we can see that the following equality holds for every $r_o$ for which $(z_o, r_o)$ belongs to $\mathbb S$:
\begin{align} \label{eq:prop_storage_function_global_minimum_03}
  \mathcal S (z_o, r_o ) = \mathcal S(z_o, 0) + \int_0^1 r_o^T \mathcal V (z_o, \sigma r_o) \,\mathrm d \sigma = 0
\end{align}

Since we made an arbitrary choice of $z_o$ from $\mathbb X$ in constructing the payoff map $\mathcal F_{z_o}$, we conclude that \eqref{eq:prop_storage_function_global_minimum_03} holds for every $(z_o, r_o)$ in $\mathbb S$. This proves the lemma.
\end{IEEEproof}

\begin{lemma} \label{lemma:convex_invariance}
Consider a differential equation given by
\begin{align} \label{eq:convex_input}
  \dot q(t) = \alpha \left( v(t) - q(t) \right)
\end{align}
where $\alpha$ is a positive constant and $v$ is a continuous function that takes a value in a closed convex subset $\mathbb F$ of $\mathbb R^n$. The set $\mathbb F$ is positively invariant and it holds that $\lim_{t \to \infty} \left( \inf_{s \in \mathbb F} \|q(t) - s\| \right) = 0$ for any $q(0) \in \mathbb R^n$.
\end{lemma}
\begin{IEEEproof}
  We first proceed with the case where $q(0)$ is contained in $\mathbb F$ and show that $\mathbb F$ is a positively invariant set of \eqref{eq:convex_input}. By contradiction, suppose that there is time indices $t_0, t_1$ for which $q(t_0) \in \mathbb F$ and $q(t) \notin \mathbb F$ for all $t \in (t_0, t_1]$.

  Let us define a piecewise constant function by
  \begin{align*}
    v_K(t) \overset{\text{def}}{=} v\left(t_0 + \frac{k-1}{K} (t_1-t_0)\right)
  \end{align*}
  for $t \in \left (t_0 + \frac{k-1}{K} (t_1-t_0), t_0 + \frac{k}{K} (t_1-t_0) \right ]$ for each $k$ in $\{1, \cdots, K\}$. Using the function $v_K(t)$, 
  we define the following:
  \begin{align}
    q_K(t_1) &= e^{-\alpha (t_1-t_0)} b_0 \nonumber \\
           & \quad + \alpha \sum_{k=1}^{K} \int_{t_0 + \frac{k-1}{K} (t_1-t_0)}^{t_0 + \frac{k}{K} (t_1-t_0)} e^{-\alpha (t_1-\tau)}  \, \mathrm d \tau ~ b_k
  \end{align}
  where $b_0 = q(t_0) \in \mathbb F$ and $b_k = v\left(t_0 + \frac{k-1}{K} (t_1-t_0)\right) \in \mathbb F$ for $k$ in $\{1, \cdots, K\}$.  Note that since $q_K(t_1)$ is a convex combination of $\{ b_k \}_{k=0}^{K}$, it holds that $q_K(t_1) \in \mathbb F$. Using the fact that $\lim_{K \to \infty} \|q_K(t_1) - q(t_1) \| = 0$, we have that $q(t_1) \in \mathbb F$, which contradicts the hypothesis that $q(t) \notin \mathbb F$ for all $t \in (t_0, t_1]$.
  
Now consider that $q(0)$ is not necessarily contained in $\mathbb F$. By explicitly writing a solution to \eqref{eq:convex_input}, we can derive the following expression:
\begin{align} \label{eq:lemma:convex_invariance_01}
q(t) &= e^{-\alpha t} q(0) + \alpha \int_0^t e^{-\alpha (t-\tau)} v(\tau) \, \mathrm d \tau \nonumber \\
&= \bar q(t) + e^{-\alpha t} \left( q(0) - \bar q(0) \right)
\end{align}
where $\bar q(t) = e^{-\alpha t} \bar q(0) + \alpha \int_0^t e^{-\alpha (t-\tau)} v(\tau) \, \mathrm d \tau$ with $\bar q(0) \in \mathbb F$. Since the second term in \eqref{eq:lemma:convex_invariance_01} vanishes as $t \to \infty$, based on the positive invariance of $\mathbb F$, we conclude that 
\begin{align}
\lim_{t \to \infty} \left( \inf_{s \in \mathbb F} \|q(t) - s\| \right) \leq  \lim_{t \to \infty} \|q(t) - \bar q(t)\| = 0
\end{align}
\end{IEEEproof}

\subsection{Proofs of Lemmas \ref{lem:MainLemma} and \ref{lem:NashStationMainLemma}} \label{appendix_a}

\paragraph{Proof of Lemma \ref{lem:MainLemma}}
% \textit{\textbf{Proof of Lemma} \ref{lem:MainLemma}:} 
Recall that under  Case~I, since the $\delta$-storage function $\mathcal S$ is informative, we have that
\begin{align} \label{eq:lemma_stability_01}
  \nabla_z^T \mathcal S \left( z, r \right) \mathcal V \left( z, r \right) = 0  \implies \mathcal S \left(z, r \right) = 0
\end{align}
Also, using Lemma \ref{lemma:passivity_edm}, we note that under Case~II the following relation is true:
\begin{align}
  &\nabla_z^T \mathcal S \left(z, r \right) \mathcal V \left(z, r \right) + \nu^\ast \mathcal V^T \left(z, r \right) \mathcal V \left(z, r\right) \nonumber \\
  &\leq -(\eta^\ast - \nu^\ast) \mathcal V^T \left( z, r \right) \mathcal V \left( z, r \right) \leq 0 
\end{align}
Since $\mathcal S$ is informative, we have that
\begin{align} \label{eq:lemma_stability_02}
  &\nabla_z^T \mathcal S \left( z, r \right) \mathcal V \left( z, r \right) + \nu^\ast \mathcal V^T \left(x, p \right) \mathcal V \left(z, r \right) = 0 \nonumber \\
  & \implies \mathcal S \left(z, r \right) = 0
\end{align}

Hence, under either Case~I ($\nu^\ast = 0$) or Case~II ($\nu^\ast > 0$), according to \eqref{eq:lemma_stability_01} and \eqref{eq:lemma_stability_02}, we can see that
\begin{align} \label{eq:lemma_stability_03}
  &\nabla_z^T \mathcal S \left( z , r \right) \mathcal V \left(z, r \right) + \nu^\ast\mathcal V^T \left(z, r \right) \mathcal V \left(z, r \right) = 0 \nonumber \\
  &\implies \mathcal S \left(z, r \right) = 0
\end{align}
    
In what follows, using \eqref{eq:lemma_stability_03}, we prove the statement of the lemma. We proceed with defining an open set defined by $\mathbb O_\epsilon \overset{\mathrm{def}}{=} \{ t > 0 \,\big|\, \mathcal S (x(t), p(t)) > \frac{\epsilon}{2} \}$ for a given state trajectory $(x,p)$ and any constant $\epsilon>0$. According to \eqref{eq:lemma_stability_03} and Remark \ref{rem:bounded}, there exists $\delta_1 > 0$ for which the following holds for all $t$ in $\mathbb O_\epsilon$:
\begin{align} \label{eq:lemma_stability_04}
  &\nabla_x^T \mathcal S (x(t), p(t)) \mathcal V(x(t), p(t)) \nonumber \\
  &+ \nu^\ast \mathcal V^T(x(t), p(t)) \mathcal V(x(t), p(t)) \leq -\delta_1
\end{align}
Note that using \eqref{eq:WDAntipPDM}, we can derive the following relations:
\begin{align} \label{eq:lemma_stability_05}
  & \mathcal S (x(t), p(t)) - \mathcal S (x(0), p(0)) - \mathcal{A}(q(0), \|\dot x \|) \nonumber \\
  &\leq \int_0^t \bigg[ \frac{\mathrm d}{\mathrm d\tau} \mathcal S (x(\tau), p(\tau)) \nonumber \\
  & \qquad\qquad - \dot p^T(\tau) \dot x(\tau) + \nu^\ast \dot x^T(\tau) \dot x(\tau) \bigg]\,\mathrm d\tau \nonumber \\
  &= \int_0^t \Big[ \nabla_x^T \mathcal S (x(\tau), p(\tau)) \mathcal V(x(\tau), p(\tau)) \nonumber \\
  & \qquad\qquad + \nu^\ast \mathcal V^T(x(\tau), p(\tau)) \mathcal V(x(\tau), p(\tau)) \Big] \,\mathrm d\tau
\end{align}
where we use the fact that $\nabla_r \mathcal S(z,r) = \mathcal V(z,r)$ (see Lemma~\ref{lemma:passivity_edm}). Since $\mathcal S$ is a non-negative function, we can infer that \eqref{eq:lemma_stability_05} is lower-bounded by $-\mathcal S(x(0), p(0))-\mathcal{A}(q(0), \|\dot x \|)$ for $t \geq 0$. In conjunction with \eqref{eq:lemma_stability_04}, this yields that
\begin{align}
  &- \mathcal S (x(0), p(0))-\mathcal{A}(q(0), \|\dot x \|) \nonumber \\
  &\leq \int_0^\infty \Big[ \nabla_x^T \mathcal S (x(\tau), p(\tau)) \mathcal V(x(\tau), p(\tau)) \nonumber \\
  & \qquad\qquad + \nu^\ast \mathcal V^T(x(\tau), p(\tau)) \mathcal V(x(\tau), p(\tau)) \Big] \,\mathrm d\tau \nonumber \\
  &\leq \int_{\mathbb{O}_\epsilon} \Big[ \nabla_x^T \mathcal S (x(\tau), p(\tau)) \mathcal V(x(\tau), p(\tau)) \nonumber \\
  & \qquad\qquad + \nu^\ast \mathcal V^T(x(\tau), p(\tau)) \mathcal V(x(\tau), p(\tau)) \Big] \,\mathrm d\tau \nonumber \\
  & \leq -\delta_1 \cdot \mathfrak L \left(\mathbb O_\epsilon\right)
\end{align}
where $\mathfrak L\left(\mathbb O_\epsilon\right)$ is the Lebesgue measure of $\mathbb O_{\epsilon}$. Hence, we have that $\mathfrak L \left(\mathbb O_\epsilon \right) \leq \delta_1^{-1} (\mathcal S (x(0), p(0))+\mathcal{A}(q(0), \|\dot x \|))$. Note that since $x \in \mathfrak X$, $\|\dot x\|$ is bounded and, hence, so does $\mathfrak L(\mathbb O_\epsilon)$.

We can represent the open set $\mathbb O_\epsilon$ as a union of disjoint open intervals $\left\{ \mathbb I_i \right\}_{i=1}^\infty$, i.e., $\mathbb O_\epsilon = \bigcup_{i=1}^\infty \mathbb I_i$. Notice that by our construction of $\mathbb O_\epsilon$, by letting $\mathbb I_i = (a_i, b_i)$, we have that $\mathcal S (x(a_i), p(a_i)) \leq \frac{\epsilon}{2}$ and $\mathcal S (x(t), p(t)) > \frac{\epsilon}{2}$ for all $t$ in $\mathbb I_i$. Since $\mathbb O_\epsilon$ has finite Lebesgue measure, it holds that $\lim_{i \to \infty} \mathfrak L \left( \mathbb I_i \right) = 0$.

In what follows, we show that for each $\epsilon > 0$, there exists $T_\epsilon > 0$ for which $\mathcal S (x(t), p(t)) < \epsilon$ holds for all $t \geq T_\epsilon$, and we conclude that $\lim_{t \to \infty} \mathcal S (x(t), p(t)) = 0$. By contradiction, suppose that there exist a constant $\epsilon > 0$ and an infinite subsequence $\left\{ \mathbb I_j \right\}_{j \in \mathbb J}$ of $\left\{ \mathbb I_i \right\}_{i=1}^\infty$ for which the following holds: For every $j$ in $\mathbb J$,
$$\max_{t \in \mathrm{cl}\left(\mathbb I_j\right)} \mathcal S (x(t), p(t)) \geq \epsilon$$
where $\mathrm{cl}\left(\mathbb I_j\right)$ is the closure of $\mathbb I_j$. Let $\overline t_j \in \mathrm{cl}\left(\mathbb I_j\right)$ be for which the following holds:
$$\mathcal S (x(\overline t_j), p(\overline t_j)) = \max_{t \in \mathrm{cl}\left(\mathbb I_j\right)} \mathcal S (x(t), p(t))$$
By letting $\mathbb I_j = (a_j, b_j)$, we can derive the following:
\begin{align} \label{eq:lemma_stability_06}
  & \mathcal S (x(\overline t_j), p(\overline t_j)) - \mathcal S (x(a_j), p(a_j)) \nonumber \\
  &= \int_{a_j}^{\overline t_j} \frac{\mathrm d}{\mathrm d \tau} \mathcal S (x(\tau), p(\tau)) \,\mathrm d\tau \nonumber \\
  &\overset{(i)}{\leq} \int_{a_j}^{\overline t_j} \dot p^T(\tau) \mathcal V(x(\tau), p(\tau)) \,\mathrm d\tau \nonumber \\
  &\overset{(ii)}{\leq} \delta_2 \mathfrak L \left(\mathbb I_j\right)
\end{align}
The inequality $(i)$ can be derived using the facts that $\nabla_x^T \mathcal S (x(t), p(t)) \mathcal V(x(t), p(t)) \leq 0$ and $\nabla_p \mathcal S(x(t),p(t)) = \mathcal V(x(t),p(t))$ (see Lemma \ref{lemma:passivity_edm}). To see that $(ii)$ holds, recall that $p$ and $\dot p$ are both bounded (see Remark \ref{rem:bounded}), and hence by the Lipschitz continuity of $\mathcal V$ (see Definition \ref{def:edm}), there is $\delta_2>0$ for which ${\dot p^T(\tau) \mathcal V(x(\tau), p(\tau)) \leq \delta_2}$ holds for $\tau \geq 0$. This immediately yields $(ii)$.

Since $\mathcal S (x(a_j), p(a_j)) \leq \frac{\epsilon}{2}$ for every $j$ in $\mathbb J$ and $\lim_{j \to \infty} \mathfrak L \left( \mathbb I_j \right) = 0$, from \eqref{eq:lemma_stability_06}, we can see that $\mathcal S (x(\overline t_j), p(\overline t_j)) < \epsilon$ for sufficiently large $j$ in $\mathbb J$. This contradicts the hypothesis that $\mathcal S (x(\overline t_j), p(\overline t_j)) \geq \epsilon$ holds for all $j$ in $\mathbb J$. Hence, we can infer that, for each $\epsilon > 0$, there exists $T_\epsilon > 0$ for which $\mathcal S (x(t), p(t)) < \epsilon, ~ \forall t \geq T_{\epsilon}$ from which we conclude that $\lim_{t \to \infty} \mathcal S (x(t), p(t)) = 0$. \QED

% Hence we can infer that for each $\epsilon > 0$, there exists $T_\epsilon > 0$ for which $\mathcal S (x(t), p(t)) < \epsilon$ holds for all $t \geq T_{\epsilon}$ from which we conclude that $\lim_{t \to \infty} \mathcal S (x(t), p(t)) = 0$. \QED

\paragraph{Proof of Lemma \ref{lem:NashStationMainLemma}}
Since the EDM is Nash stationary and has an informative $\delta$-storage function $\mathcal S$, according to Lemma \ref{lemma:ns_minimizer}, the following relations hold:
\begin{align} \label{eq:lemma_stability_ne_01}
  \mathcal S(z,r) = 0 &\iff \mathcal V(z,r) = 0 \nonumber \\
                      & \iff z \in \argmax_{\bar z \in \mathbb X} \bar z^T r
\end{align}
Using Lemma \ref{lem:MainLemma}, we have that 
\begin{align} \label{eq:lemma_stability_ne_02}
  \lim_{t \to \infty} \mathcal S(x(t), p(t)) = 0
\end{align}

According to Remark \ref{rem:bounded}, without loss of generality, we may assume that there is a positive constant $\rho$ for which the deterministic payoff $p(t)$ satisfies $\|p(t)\| \leq \rho, ~ t \geq 0$, and we may redefine the set of stationary points of the EDM as ${\mathbb S = \{(z,r) \in \mathbb X \times \mathbb R^n \,|\, z \in \argmax_{\bar z \in \mathbb X} \bar z^T r \text{ and } \|r\| \leq \rho\}}$. Note that $\mathbb S$ is a closed set, and hence it is compact.

By \eqref{eq:lemma_stability_ne_01}, \eqref{eq:lemma_stability_ne_02}, and Remark \ref{rem:bounded}, it holds that 
\begin{align} \label{eq:lemma_stability_ne_03}
  &\lim_{t \to \infty} \left( \inf_{(z,r) \in \mathbb S} \left[ \|x(t) - z\| + \|p(t) - r\| \right] \right) = 0 \\
  &\lim_{t \to \infty} \|\dot x(t)\| = 0
\end{align}
and in conjunction with Assumption~\ref{assumption:stationary_game}, we have that
\begin{align} \label{eq:lemma_stability_ne_09}
  \lim_{t \to \infty} \|p(t) - \bar{\mathcal F}(x(t))\| = 0
\end{align}
Also note that in conjunction with \eqref{eq:lemma_stability_ne_03}-\eqref{eq:lemma_stability_ne_09}, using the fact that
\begin{align} \label{eq:lemma_stability_ne_04}
  (z, \bar{\mathcal F}(z)) \in \mathbb S \iff z \in \mathbb{NE}(\bar{\mathcal F})
\end{align}
we can derive the following inequality:
\begin{align} \label{eq:lemma_stability_ne_05}
  & \lim_{t \to \infty} \left( \inf_{(z,r) \in \mathbb S} \left[ \|x(t) - z\| + \|\bar{\mathcal F}(x(t)) - r \| \right] \right) \nonumber \\
  &\leq \lim_{t \to \infty} \left( \inf_{(z, r) \in \mathbb S} \left[ \|x(t) - z\| + \| p(t) - r\| \right] \right) \nonumber \\
  & \qquad + \lim_{t \to \infty} \|p(t) - \bar{\mathcal F}(x(t)) \|  = 0
\end{align}

To show global attractiveness of $\mathbb{NE}(\bar{\mathcal F})$,
in what follows, we prove that \eqref{eq:lemma_stability_ne_05} yields
\begin{align} \label{eq:lemma_stability_ne_06}
  \lim_{t \to \infty} \left( \inf_{(z,\bar{\mathcal F}(z)) \in \mathbb S} \left[ \|x(t) - z\| + \|\bar{\mathcal F}(x(t)) - \bar{\mathcal F}(z) \| \right] \right) = 0
\end{align}
and we conclude that
\begin{align}
  & \lim_{t \to \infty} \inf_{z \in \mathbb{NE} (\bar{\mathcal F})} \|x(t) - z\| \nonumber \\
  & = \lim_{t \to \infty} \inf_{(z,\bar{\mathcal F}(z)) \in \mathbb S} \|x(t) - z\| \nonumber \\
  & \leq \lim_{t \to \infty} \left( \inf_{(z,\bar{\mathcal F}(z)) \in \mathbb S} \left[ \|x(t) - z\| + \|\bar{\mathcal F}(x(t)) - \bar{\mathcal F}(z) \| \right] \right) = 0
\end{align}

By contradiction, suppose that there is a sequence of increasing time indices $\{t_n\}_{n=1}^\infty$ for which the sequence $\{x(t_n)\}_{n=1}^\infty$ converges and it holds that
\begin{align} \label{eq:lemma_stability_ne_07}
  \lim_{n \to \infty} \left( \inf_{(z,r) \in \mathbb S} \left[ \|x(t_n) - z\| + \|\bar{\mathcal F}(x(t_n)) - r \| \right] \right) = 0
\end{align}
but
\begin{align} \label{eq:lemma_stability_ne_08}
  \lim_{n \to \infty} \left( \inf_{(z,\bar{\mathcal F}(z)) \in \mathbb S} \left[ \|x(t_n) - z\| + \|\bar{\mathcal F}(x(t_n)) - \bar{\mathcal F}(z) \| \right] \right) > 0
\end{align}
Since $\mathbb S$ is compact, there is a converging sequence $\{(z_n, r_n)\}_{n=1}^\infty$ for which its limit point $(z^\ast, r^\ast)$ is contained in $\mathbb S$ and it holds that
\begin{align}
  \|x^\ast - z^\ast\| + \|\bar{\mathcal F}(x^\ast) - r^\ast\| = 0
\end{align}
where $x^\ast$ is the limit of $\{x(t_n)\}_{n=1}^\infty$. Hence, we have that $(x^\ast, \bar{\mathcal F}(x^\ast)) \in \mathbb S$. 

On the other hand, by a similar argument, from \eqref{eq:lemma_stability_ne_08}, we can show that $x^\ast$ is not a Nash equilibrium; otherwise \eqref{eq:lemma_stability_ne_08} should converge to zero. This is a contradiction and proves that $\mathbb{NE} (\bar{\mathcal F})$ is globally attractive.

In what follows, we prove Lyapunov stability of $\mathbb{NE} (\bar{\mathcal F})$. Recall that if the PDM is $\delta$-antipassive, there is a continuously differentiable map $\mathcal L: \mathbb X \times \mathbb R^n \to \mathbb R_+$ for which the following two relations are true:
\begin{align}
  &\mathcal L(z,s) = 0 \iff \mathcal H(s,z) = \bar{\mathcal F}(z) \label{eq:antipassivity_01} \\
  &\frac{\mathrm d}{\mathrm dt} \mathcal L(x(t), q(t)) \leq - \dot p^T(t) \dot x(t) \label{eq:antipassivity_02}
\end{align}
In conjunction with \eqref{eq:lemma_stability_ne_01} and Lemma~\ref{lemma:passivity_edm}, using \eqref{eq:antipassivity_01} and \eqref{eq:antipassivity_02}, we have that
\begin{align}
&\mathcal S(z, \mathcal H(s,z)) + \mathcal L(z, s) = 0 \nonumber \\
& \qquad \iff \mathcal H(s,z) = \bar{\mathcal F}(z) \text{ and } z \in \mathbb{NE}(\bar{\mathcal F}) \label{eq:lyapunv_stability_01}\\
&\frac{\mathrm d}{\mathrm dt} \left[ \mathcal S(x(t), \mathcal H(q(t), x(t))) + \mathcal L(x(t), q(t)) \right] \leq 0 \label{eq:lyapunv_stability_02}
\end{align}

Let $\mathbb A$ be a subset of $\mathbb X \times \mathbb R^n$ defined by
\begin{align*}
  \mathbb A = \left\{ (z,s) \in \mathbb X \times \mathbb R^n \, \big| \, \mathcal S(z, \mathcal H(s,z)) + \mathcal L(z, s) = 0 \right\}
\end{align*}
Note that according to Assumption \ref{assumption:stationary_game}, the set $\mathbb A$ is a compact subset or $\mathbb F \times \mathbb R^n$, where the set $\mathbb F$ is given by
\begin{align*}
  \mathbb F = \{z \in \mathbb X \, \big| \, \mathcal S(z, \bar{\mathcal F}(z)) = 0\}
\end{align*}

Let $\mathbb O$ be a given open set containing $\mathbb A$. For the former case, without loss of generality, suppose that $\mathbb O$ is bounded; otherwise, we can select a bounded open set $\mathbb V$ containing $\mathbb A$ and proceed with the intersection $\mathbb O \cap \mathbb V$. Define a constant $\alpha$ as $\alpha = \min_{(z,s) \in \textrm{bd} (\mathbb O)  \cap (\mathbb X \times \mathbb R^n) } \left[ \mathcal S(z, \mathcal H(s,z)) + \mathcal L(z, s) \right]$. With a constant $\beta$ satisfying $\alpha > \beta > 0$, we consider an open set defined by $\mathbb O_{\beta} = \left\{ (z,s) \in \mathbb O  \cap (\mathbb X \times \mathbb R^n) \, \big| \, \mathcal S(z, \mathcal H(s,z)) + \mathcal L(z, s) < \beta\right\}$. Note that since $\frac{\mathrm d}{\mathrm dt} \left[ \mathcal S(x(t), \mathcal H(q(t), x(t))) + \mathcal L(x(t), q(t)) \right] \leq 0$, any trajectory starting from $\mathbb O_{\beta}$ should remain in $\mathbb O$; otherwise, there should exist a time index $t_1$ for which $\mathcal S(x(t_1), \mathcal H(q(t_1), x(t_1))) + \mathcal L(x(t_1), q(t_1)) \geq \alpha > \beta$ holds, which is a contradiction.

For the later case, the set $\mathbb O$ can be written as $\mathbb O = \mathbb U \times \mathbb R^n$ where $\mathbb U$ is an open set containing $\mathbb F$. Without loss of generality, we assume that $\mathbb U$ is a bounded subset of $\mathbb R^n$; otherwise, we can select a bounded open set $\mathbb V$ containing $\mathbb F$ and proceed the proof with the intersection $\mathbb U \cap \mathbb V$. Define $\alpha = \min_{z \in \textrm{bd} (\mathbb U) \cap \mathbb X}  \mathcal S(z, \bar{\mathcal F}(z))$. With a constant $\beta$ satisfying $\alpha > \beta > 0$, we consider an open set defined by $\mathbb U_{\beta} = \left\{ z \in \mathbb U \, \big| \, \mathcal S(z, \bar{\mathcal F}(z)) < \beta\right\}$. Note that since $\frac{\mathrm d}{\mathrm dt} \mathcal S(x(t), \bar{\mathcal F}(x(t))) \leq 0$, any trajectory starting from $\mathbb U_\beta$ should remain in $\mathbb U$; otherwise, there should exist a time index $t_1 > 0$ for which $\mathcal S(x(t_1), \bar{\mathcal F}(x(t_1))) \geq \alpha > \beta$ holds, which is a contradiction. This completes the proof of lemma. \QED

\subsection{Proofs of Propositions \ref{prop:IEPT-PAssive} and \ref{prop:DeltaPassivityPBR}, and Theorem \ref{thm:PBRConvergence}} \label{appendix_b}

\paragraph{Proof of Proposition \ref{prop:IEPT-PAssive}}
We first note that the acuteness condition \eqref{eq:EPTAcuteness} implies the so-called \textit{Positive Correlation} \cite{Sandholm2005Excess-payoff-d} which is described as follows:
\begin{align}
  \mathcal V(z, r) \neq 0 \implies r^T \mathcal V(z, r) > 0 \label{eq:strict_positive_correlation}
\end{align}
where $\mathcal V$ is a map defined by the protocol $\mathcal T^\textit{EPT}$ as in Definition~\ref{def:protocol}.

Let $\mathcal I^\textit{EPT}$ be the revision potential of $\mathcal T^\textit{EPT}$ as defined in \eqref{eq:revision_potential}. It can be verified that the gradients of $\mathcal I^\textit{EPT}$ with respect to $r$ and $z$, respectively, satisfy
\begin{subequations}
  \begin{align}
    m \, \nabla_r \mathcal I^\textit{EPT}(\hat r) &=  \mathcal V(z, r) \label{eq:prop_ept_passivity_04} \\
    m \, \nabla_z^T \mathcal I^\textit{EPT}(\hat r) \mathcal V(z, r) &= -\left( r^T \mathcal V(z, r)\right) \sum_{i=1}^n \mathcal T_i^\textit{EPT}(\hat r)  \label{eq:prop_ept_passivity_01}
  \end{align}
\end{subequations}
Let us select a candidate $\delta$-storage function as $\mathcal S^\textit{EPT}(z, r) = m \, \mathcal I^\textit{EPT}(\hat r) - \gamma$ for some constant $\gamma$. Due to \eqref{eq:prop_ept_passivity_04}, the function $\mathcal S^\textit{EPT}$ satisfies \eqref{eq:passivity_characterization_01}. In conjunction with the fact that $\mathcal T^\textit{EPT}(\hat r) = 0$ implies $\mathcal V(z, r) = 0$, due to \eqref{eq:strict_positive_correlation} and \eqref{eq:prop_ept_passivity_01}, we can see that \eqref{eq:passivity_characterization_02} holds with $\eta = 0$ and the equality in \eqref{eq:passivity_characterization_02} holds only if $\mathcal V(z,r) = 0$.

Suppose that $\mathcal I^\textit{EPT}$ also satisfies the following inequality for every $\hat r$ in $\mathbb R_\ast^n$:
\begin{align} \label{eq:prop_ept_passivity_02}
  \mathcal I^\textit{EPT}(\hat r) \geq \mathcal I^\textit{EPT}(0)
\end{align}
Then by defining $\gamma = m \, \mathcal I^\textit{EPT}(0)$, we conclude that $ m \, \mathcal I^\textit{EPT}(\hat r) - \gamma$ is nonnegative.

Based on the aforementioned arguments and Lemma~\ref{lemma:passivity_edm}, the EPT EDM is $\delta$-passive with a $\delta$-storage function given by $\mathcal S^\textit{EPT}(z,r) = m \, \mathcal I^\textit{EPT}(\hat r) - m \, \mathcal I^\textit{EPT}(0)$. In addition, in conjunction with Remark~\ref{rem:EPTNashStationary} and Lemma~\ref{lemma:ns_minimizer}, we conclude that $S^\textit{EPT}$ is informative and \eqref{eq:InverseIEPTStorage} holds. In what follows, we show that \eqref{eq:prop_ept_passivity_02} is valid for every $\hat r$ in $\mathbb R_\ast^n$.

We first claim that \eqref{eq:prop_ept_passivity_02} holds for all $(z,r)$ in the set $\mathbb S^\textit{EPT}$ of stationary points of the EPT EDM.
% To see this, first note that due to \eqref{eq:strict_positive_correlation}, it holds that $\mathcal V(z, 0) = 0$, i.e., $(z,0) \in \mathbb S^\textit{EPT}$, for all $z$ in $\mathbb X$.
By \eqref{eq:prop_ept_passivity_04}, for fixed $z$ in $\mathbb X$, the following equality holds for all $r$ in $\mathbb R^n$:
\begin{align} \label{eq:prop_ept_passivity_05}
  \mathcal I^\textit{EPT} \left( \hat r\right) - \mathcal I^\textit{EPT} \left(  0 \right) = \frac{1}{m} \int_0^1  r^T \mathcal V(z, \sigma r) \,\mathrm d \sigma
\end{align}
where $\hat r_i = r_i-\frac{1}{m}\sum_{i=1}^nr_i z_i$. Since the EPT EDM is Nash stationary (see Remark \ref{rem:EPTNashStationary}), $(z,r) \in \mathbb S^\textit{EPT}$ implies $(z,\sigma r) \in \mathbb S^\textit{EPT}$ for all $\sigma$ in $[0,1)$, and by \eqref{eq:prop_ept_passivity_05}, for each $(z,r)$ in $\mathbb S^\textit{EPT}$, we have that 
\begin{align} \label{eq:prop_ept_passivity_03}
\mathcal I^\textit{EPT}(\hat r) = \mathcal I^\textit{EPT}(0)
\end{align}
Since \eqref{eq:prop_ept_passivity_03} holds for every $(z,r)$ in $\mathbb S^\textit{EPT}$, this proves the claim.

To see that \eqref{eq:prop_ept_passivity_02} extends to the entire domain $\mathbb X \times \mathbb R^n$, by contradiction, let us assume that there is $(z', r') \notin \mathbb S^\textit{EPT}$ for which $\mathcal S^\textit{EPT} (z', r') + \gamma = m \, \mathcal I^\textit{EPT}(\hat r') < m \, \mathcal I^\textit{EPT}(0)$ holds, where $\hat r_i' = r_i'-\frac{1}{m}\sum_{i=1}^nr_i' z_i'$. Let $x(t)$ be the mean population state induced by the EPT EDM with the initial condition $x(0) = z'$ and the constant deterministic payoff $p(t) = r', ~ t \geq 0$. By \eqref{eq:strict_positive_correlation} and \eqref{eq:prop_ept_passivity_01}, the value of $\mathcal S^\textit{EPT} (x(t), p(t))$ is strictly decreasing unless $\mathcal V(x(t), p(t)) = 0$. By the hypothesis that $\mathcal S^\textit{EPT} (z', r') + \gamma < m \, \mathcal I^\textit{EPT}(0)$ and by \eqref{eq:prop_ept_passivity_03}, for every $(z,r)$ in $\mathbb S^\textit{EPT}$, it holds that $\mathcal S^\textit{EPT} (z',r') < \mathcal S^\textit{EPT} (z,r)$ and the state $\left( x(t), p(t) \right)$ never converges to $\mathbb{S}^\textit{EPT}$. On the other hand, by LaSalle's Theorem \cite{Khalil1995Nonlinear-syste}, since $p(t)$ is constant and the mean population state $x(t)$ is contained in a compact set, $x(t)$ converges to an invariant subset of $\left\{ z \in \mathbb X \,\big|\, \nabla_z^T \mathcal S^\textit{EPT} (z, r') \mathcal V(z, r') = 0\right\}$. By \eqref{eq:strict_positive_correlation} and \eqref{eq:prop_ept_passivity_01}, the invariant subset is contained in $\mathbb S^\textit{EPT}$. This contradicts the previous argument that the state $(x(t), p(t))$ does not converge to $\mathbb S^\textit{EPT}$. This proves that $\mathcal I^\textit{EPT} (\hat r) \geq \mathcal I^\textit{EPT}(0)$ holds for all $(z,r)$ in $\mathbb X \times \mathbb R^n$. \QED

\paragraph{Proof of Proposition \ref{prop:DeltaPassivityPBR}}
The analysis used in \cite[Theorem 2.1]{Hofbauer2002On-the-global-c} suggests that the following hold: For all $r$ in $\mathbb R^n$, $z$ in $\mathbb X$, and $\tilde z$ in $\mathbb{TX}$,
\begin{subequations} \label{eq:prop_perturbed_dynamics_sop_01}
  \begin{align}
    \tilde z^T \nabla_r \left[ \max_{\bar z \in \mathrm{int}(\mathbb{X})} (r^T \bar z - \mathcal Q(\bar z)) \right] = \tilde z^T \mathcal M^{\mathcal Q}(r)
  \end{align}
  \begin{align}
    \tilde z^T \nabla \mathcal Q(z) = \tilde z^T r \text{ if and only if } z = \mathcal M^{\mathcal Q}(r) \label{eq:prop_perturbed_dynamics_sop_01_b}
  \end{align}
\end{subequations}
Using \eqref{eq:prop_perturbed_dynamics_sop_01}, we can see that
\begin{align}
  \nabla_r \mathcal S^{\textit{PBR}}(z,r) = \mathcal M^{\mathcal Q}(r)-z = \tilde{\mathcal V}^{\mathcal Q}(z,r)
\end{align}
and
\begin{align} \label{eq:prop_perturbed_dynamics_sop_02}
  \nabla_z^T \mathcal S^{\textit{PBR}}(z, r) \tilde{\mathcal V}^{\mathcal Q}(z, r) &= - \left( r - \nabla \mathcal Q(z) \right)^T \tilde{\mathcal V}^{\mathcal Q}(z, r) \nonumber \\
	&= -\left( \nabla \mathcal Q(w) - \nabla \mathcal Q (z)\right)^T \left( w - z \right)
\end{align}
where $w = \mathcal M^{\mathcal Q}(r)$. By the fact that $\mathcal Q$ is strictly convex, it holds that $\nabla_z^T \mathcal S^{\textit{PBR}} (z,r) \tilde{\mathcal V}^{\mathcal Q}(z,r) \leq 0$ where the equality holds only if $\tilde{\mathcal V}^{\mathcal Q}(z,r) = 0$, which, by the definition of $\mathcal S^\textit{PBR}$, is equivalent to $\mathcal S^\textit{PBR}(z,r) = 0$. By Lemma \ref{lemma:passivity_edm} and Remark \ref{remark:perturbed_stationarity}, the PBR EDM is $\delta$-passive and $\mathcal S^{\textit{PBR}}$ is its associated informative $\delta$-storage function. This proves Case I.

Furthermore, if the perturbation $\mathcal Q$ satisfies $\tilde z^T \nabla^2 \mathcal Q(z) \tilde z \geq \eta^\ast \tilde z^T \tilde z$ for all $z \in \mathbb X, \, \tilde z \in \mathbb{TX}$, i.e., $\mathcal Q$ is strongly convex, then it holds that $\left( \nabla \mathcal Q(w) - \nabla \mathcal Q(z)\right)^T \left( w - z \right) \geq \eta^\ast \left\| w-z\right\|^2$ for all $w,z$ in $\mathbb X$. Using \eqref{eq:prop_perturbed_dynamics_sop_02} we can derive $\nabla_z^T \mathcal S^{PBR} (z, r) \tilde{\mathcal V}^{\mathcal Q}(z, r) \leq -\eta^\ast \tilde{\mathcal V}^{\mathcal Q}(z, r)^T \tilde{\mathcal V}^{\mathcal Q}(z, r)$. By Lemma \ref{lemma:passivity_edm} and Remark \ref{remark:perturbed_stationarity}, the PBR EDM is $\delta$-passive with surplus $\eta^\ast$ and $\mathcal S^{\textit{PBR}}$ is its associated informative $\delta$-storage function. This proves Case II. \QED

\paragraph{Proof of Theorem \ref{thm:PBRConvergence}}
We proceed with recalling that according to Proposition \ref{prop:DeltaPassivityPBR}, $\mathcal S^{\textit{PBR}}$ given in \eqref{eq:s_pbr} is an informative $\delta$-storage function of the PBR EDM and satisfies
\begin{align}
\mathcal S^\textit{PBR}(z,r) = 0 &\iff \tilde{\mathcal V}^{\mathcal Q}(z,r) = 0 \nonumber \\
& \iff z = \argmax_{\bar z \in \mathrm{int}(\mathbb X)} (\bar z^T r - \mathcal Q(\bar z))
\end{align}

According to Remark \ref{rem:bounded}, without loss of generality, we may assume that there is a positive constant $\rho$ for which the deterministic payoff $p(t)$ satisfies $\|p(t)\| \leq \rho, ~ t \geq 0$, and we may redefine the set of stationary points of the PBR EDM as 
\begin{align*}
	\mathbb S^{\textit{PBR}} = \big\{ (z,r) \in \mathrm{int}(\mathbb X) \times \mathbb R^n \, \big|\, &z = \argmax_{\bar z \in \mathrm{int}(\mathbb X)} (\bar z^T r - \mathcal Q(\bar z)) \\
    & \qquad \qquad \text{ and } \|r\| \leq \rho \big\}
\end{align*}
% Based on the analysis used in \cite[Theorem 2.1]{Hofbauer2002On-the-global-c}, if 
% $(z,r) \in \mathbb S^{\textit{PBR}}$ then $\tilde z^T \nabla \mathcal Q(z) = \tilde z^Tr$ holds for all $\tilde z \in \mathbb {TX}$. In addition 
which asserts that $\mathbb S^{\textit{PBR}}$ is a bounded set. To show that $\mathbb S^{\textit{PBR}}$ is closed, and hence compact, consider any sequence $\{(z_n, r_n)\}_{n=1}^\infty$ in $\mathbb S^{\textit{PBR}}$. Since it is bounded, $\{(z_n, r_n)\}_{n=1}^\infty$ has a convergent subsequence, and by definition of $\mathbb S^{\textit{PBR}}$, the limit point of the subsequence belongs to $\mathbb S^{\textit{PBR}}$, which shows that $\mathbb S^{\textit{PBR}}$ is a compact set.

% Since $\mathbb S^{\textit{PBR}}$ is a bounded set, 
% from \eqref{eq:boundary}, 
% we conclude that any sequence $\{(z_n, r_n)\}_{n=1}^\infty \subset \mathbb S^{\textit{PBR}}$ has accumulation points that belong to $\mathbb S^{\textit{PBR}}$, which shows that $\mathbb S^{\textit{PBR}}$ is a compact set.

% Using Lemma \ref{lem:MainLemma}, 
The rest of the proof on establishing global attractiveness and Lyapunov stability of $\mathbb{PE}(\bar{\mathcal F}, \mathcal Q)$ is analogous to the arguments used in the proof of Lemma \ref{lem:NashStationMainLemma}. We omit the details for brevity \QED

\subsection{Proof of Propositions \ref{prop:SmoothingAnticipatoryPDMStability} and \ref{prop:SmoothingPDMStability_Potential}} \label{appendix_c}

\paragraph{Proof of Proposition \ref{prop:SmoothingAnticipatoryPDMStability}} In order to prove i), ii), and iii) we need to establish that the PDM satisfies the weak $\delta$-antipassivity condition~(\ref{eq:WDAntipPDM}).

We start by using the linear superposition principle to write the solution of (\ref{eq:SmoothingPDM}) as $p=p^h+p^f$, where $p^h$ is the solution assuming that $\dot{u}$ is zero and $p^f$ assumes that ${u(0)=0}$, ${p(0)=0}$, and ${\bar{r}=0}$. Namely, $p^f$ captures the effect of $\dot{u}$ and $p^h$ is the solution of the initial value problem with respect to given $u(0)$, $p(0)$, and $\bar{r}$. We proceed with analyzing the two terms on the right-hand side below:
\begin{equation}
\int_0^t \dot{u}^T (\tau) \dot{p} (\tau) d \tau = \int_0^t \dot{u}^T (\tau) \dot{p}^{f} (\tau) \, \mathrm d \tau + \int_0^t \dot{u}^T (\tau) \dot{p}^{h} (\tau) \, \mathrm d \tau
\end{equation}
The explicit solution for the homogeneous part gives ${\dot{p}^h(t) = (\alpha \mu_1 -  \alpha^2 \mu_2)  e^{-\alpha t} \big ( Fu(0) +\bar{r}  -  q(0) \big ) }$. In particular, the following inequality holds:
\begin{equation}
\int_0^t  \dot{u}^T (\tau) \dot{p}^h (\tau) \, \mathrm d \tau \leq \mathcal{A}(q(0),\|\dot{u}\|), \quad q(0) \in \mathbb{R}^n, u \in \mathfrak{X}
\end{equation} where the map $\mathcal{A}:\mathbb{R}^n \times \mathbb{R}_+ \rightarrow \mathbb{R}_+$ can be chosen as:
% \begin{equation}
% \mathcal{A}(q(0),\|\dot{u}\|) = |\mu_1-\mu_2\alpha| \|\dot{u}\| \left ( \| \bar{r}-q(0)\| + \max_{z \in \mathbb{X}} \| Fz \| \right)
% \end{equation}
\begin{equation}
\mathcal{A}(q(0),\|\dot{u}\|) = |\mu_1-\mu_2\alpha| \|\dot{u}\| \max_{z \in \mathbb X} \left\| Fz + \bar r - q(0)\right\|
\end{equation}

We now proceed with showing that \underline{i) and ii) hold} by determining when the following inequality is satisfied:
\begin{equation}
\label{eq:passivityPosRealL}
\int_0^t \dot{u}^T (\tau) \dot{p}^{f} (\tau) \, \mathrm d \tau \leq \lambda^* \int_0^{t} \dot{u}^T (\tau) \dot{u}(\tau) \, \mathrm d \tau, \quad t \geq 0, \ u \in \mathfrak{X}
\end{equation} Here, we can use the fact that the map from $\dot{u}$ to $\dot{p}^f$ is linear time invariant and characterized by the transfer function matrix $\frac{\alpha+\bar{\mu} s}{\alpha + s} F$, where $s$ is the complex Laplace variable and $\bar{\mu}$ is defined as $\mu_0+\alpha \mu_2$. At this point we can use Parseval's Theorem, which guarantees that (\ref{eq:passivityPosRealL}) holds if the following inequality is satisfied:
\begin{multline} 
 z^H \left ( \frac{\alpha + \bar{\mu} j \omega}{\alpha + j \omega} F + \frac{\alpha - \bar{\mu} j \omega}{\alpha - j \omega} F^T \right ) z \leq 2 \lambda^* z^Hz, \\ \omega \in \mathbb{R}, \ z \in \mathbb{TC}
\end{multline} where $H$ represents the Hermitian operator (transpose conjugate) and $\mathbb{TC} \overset{\mathrm{def}}{=} \{z \in \mathbb{C}^n | \sum_{i=1}^n z_i = 0 \}$. Since $(\alpha+j\omega)$ is nonzero for all real $\omega$, we can re-state the condition as follows:
\begin{multline} 
z^H (\alpha^2+\bar{\mu} \omega^2 + j \omega (\bar{\mu} \alpha -\alpha)) F  z + \\ z^H (\alpha^2+\bar{\mu} \omega^2 - j \omega (\bar{\mu} \alpha -\alpha)) F^T z \leq  \\ 2 \lambda^* (\alpha^2 + \omega^2) z^Hz, \\ \omega \in \mathbb{R}, \ z \in \mathbb{TC}
\end{multline} Rearranging terms leads to yet another equivalent inequality:
\begin{multline} 
 (\alpha^2+\bar{\mu} \omega^2) z^H (F+F^T) z + \\ j \omega (\bar{\mu} \alpha -\alpha) z^H (F - F^T )z \leq \\ 2 \lambda^* (\alpha^2 + \omega^2) z^Hz, \\ \omega \in \mathbb{R}, \ z \in \mathbb{TC}
\end{multline} 
Now, notice that because the statement of the proposition requires $\Phi F \Phi$ to be symmetric,  the second term of the left-hand side above vanishes and the inequality becomes:
\begin{equation} 
z^T F z  \leq  \lambda^* \frac{\alpha^2 + \omega^2}{\alpha^2+ \bar{\mu} \omega^2} z^Tz, \quad \omega \in \mathbb{R}, \ z \in \mathbb{TX}
\end{equation}  

Notice that above we no longer need to consider $z$ in $\mathbb{TC}$. More specifically, we can simply restrict $z$ to $\mathbb{TX}$ because otherwise the imaginary components of $z$ would cancel out. The proof for i) and ii) is concluded once we realize that if $\lambda^*=0$ or $\lambda^*>0$ and $ \bar{\mu} \leq 1$ then the inequality above is satisfied. In order to show that iii) is also true, we need to determine when the following inequality holds:
\begin{equation}
\label{eq:passivityPosRealLv2}
\int_0^t \dot{u}^T (\tau) \dot{p}^{f} (\tau) \, \mathrm d \tau \leq \bar{\mu} \lambda^* \int_0^{t} \dot{u}^T (\tau)\dot{u}(\tau) \, \mathrm d \tau, \quad t \geq 0, \ u \in \mathfrak{X}
\end{equation}
By proceeding in a way that is analogous to our proof of i) and ii), we can show that (\ref{eq:passivityPosRealLv2}) is equivalent to:
\begin{equation} 
z^T F z  \leq \bar{\mu}  \lambda^* \frac{\alpha^2 + \omega^2}{\alpha^2 + \bar{\mu} \omega^2} z^Tz, \quad \omega \in \mathbb{R}, \ z \in \mathbb{TX}
\end{equation} 
The proof of iii) is concluded once we realize that if $\lambda^*>0$ and $\bar{\mu} >1$ then the inequality above holds. \QED

\paragraph{Proof of Proposition \ref{prop:SmoothingPDMStability_Potential}}
Define $f^\ast$ as the Legendre conjugate of $f$ given by
$$f^\ast(s) \overset{\text{def}}{=} \sup_{y \in \mathbb R^n} (f(y) - s^Ty)$$
Note that $\mathbb D^\ast$ is the domain of $f^\ast$, which is a convex set, and its interior $\mathrm{int}(\mathbb D^\ast)$ is an open convex set.

We first show that, for any $u \in \mathfrak X$, the set $\mathrm{int}(\mathbb D^\ast)$ is positively invariant, i.e., $q(t) \in \mathrm{int}(\mathbb D^\ast),~ \forall t \geq 0$. Let $\mathbb F$ be a closed convex subset of $\mathrm{int}(\mathbb D^\ast)$ containing $\mathrm{Im}(\mathcal F)$ and $q(0)$.
Based on Lemma~\ref{lemma:convex_invariance}, we can infer that $q(t) \in \mathbb F \subset \mathrm{int}(\mathbb D^\ast)$ for all $t \geq 0$. This proves the positive invariance of $\mathrm{int}(\mathbb D^\ast)$.

Note that it follows from continuous differentiability of $f$ and \cite[Theorem 26.5]{Rockafellar1996Convex-analysis} that $\mathcal L$ is also continuously differentiable. Hence, to show $\delta$-antipassivity of the PDM, according to Definition~\ref{def:EquilibriumStability}, it suffices to show that 
\begin{align} \label{eq:smoothing_pdm_antipassivity_01}
\mathcal L(z,s) = 0 \iff \mathcal F(z) = s
\end{align}
holds for all $z$ and $s$ in $\mathbb X$ and $\mathbb R^n$, respectively, and 
\begin{align} \label{eq:smoothing_pdm_antipassivity_02}
\frac{\mathrm d}{\mathrm d t} \mathcal L(u(t),q(t)) \leq - \dot p^T(t) \dot u(t)
\end{align}
holds for all $q(0)$ in $\mathrm{int}(\mathbb D^\ast)$ and $u$ in $\mathfrak X$.

To see that \eqref{eq:smoothing_pdm_antipassivity_01} is valid, note that the function $\mathcal L$ is defined over $\mathbb R^n \times \mathbb R^n$ and $\mathcal L(z,s) = 0$ implies $\nabla_z \mathcal L(z,s) = 0$ which, according to \eqref{eq:PDM_antistorage}, ensures that $\mathcal F(z) = s$. Conversely, if $\mathcal F(z) = s$ then, by strict concavity of $f$, it holds that $\sup_{y \in \mathbb R^n} (f(y) - s^Ty) = f(z) - s^Tz$, which ensures that  $\mathcal L(z,s) = 0$.

To show \eqref{eq:smoothing_pdm_antipassivity_02}, we compute the time-derivative of \eqref{eq:PDM_antistorage} as follows:
\begin{align*}
  &\frac{\mathrm d}{\mathrm d t} \mathcal L(u(t), q(t)) \nonumber \\
  &= \nabla_u^T \mathcal L(u(t),q(t)) \dot u(t) + \nabla_q^T \mathcal L(u(t), q(t)) \dot q(t) \\
  &= -\dot p^T(t) \dot u(t) + \alpha^2 (\nabla f^\ast (q(t)) + u(t))^T (\mathcal F(u(t)) - q(t)) \\
  &\overset{(i)}{=} -\dot p^T(t) \dot u(t) + \alpha^2 (u(t) - v(t))^T (\nabla f(u(t)) - \nabla f(v(t)))\\
  &\overset{(ii)}{\leq} - \dot p^T(t) \dot u(t)
\end{align*}
where to derive $(i)$ and $(ii)$, we use strict concavity of $f$, the positive invariance of $\mathrm{int}(\mathbb D^\ast)$, and \cite[Theorem 26.5]{Rockafellar1996Convex-analysis} to show that $q(t) =  \nabla f(v(t))$ whenever $\nabla f^\ast (q(t)) = -v(t)$. \QED

\bibliography{MartinsRefs}
\bibliographystyle{IEEEtran}

\end{document}